\documentclass[11pt,a4paper]{article}
\usepackage[a4paper,hmargin=25mm,vmargin=25mm]{geometry}
\usepackage[english]{babel}
\usepackage[utf8]{inputenc}
\usepackage[T1]{fontenc}
\usepackage{lmodern}

\usepackage{multirow}

\usepackage{authblk}
\usepackage{amsmath}
\usepackage{amsbsy}
\usepackage{amsfonts}
\usepackage{amssymb}
\usepackage{amscd}
\usepackage{amsthm}
\usepackage{color}
\usepackage{appendix}

\usepackage{dsfont}
\usepackage{units}
\usepackage{pdfsync}
\usepackage{tikz}
\usepackage{pgfplots}
\usetikzlibrary{calc}
\usepackage{caption}
\usepackage{subcaption}
\usepackage{hyperref}
\hypersetup{colorlinks=true, linkcolor=black,citecolor=black}
\usepackage{stmaryrd}

\newtheorem{de}{Definition}
\newtheorem{remark}{Remark}
\newtheorem{lemma}{Lemma}

\newtheorem{property}{Property}
\newtheorem{thm}{Theorem}
\newtheorem{corollary}{Corollary}

\newcommand{\Diff}{\text{Diff}}

\DeclareMathOperator{\Com}{com}
\DeclareMathOperator{\vol}{vol}
\DeclareMathOperator{\Id}{Id}
\DeclareMathOperator{\myEnd}{End}

\providecommand{\M}{\ensuremath{\boldsymbol{M}} }
\providecommand{\q}{\ensuremath{\boldsymbol{q}}}

\providecommand{\btf}{\ensuremath{\boldsymbol{\tilde{f}}} }

\def\R{\mathbb{R}}
\def\N{\mathbb{N}}
\def\S{\mathbb{S}}

\def\Haus{\mathcal{H}}

\def\og{\overline{g}}
\def\onabla{\overline{\nabla}}
\def\cI{\mathcal{I}}
\def\cJ{\mathcal{J}}
\def\indset{\llbracket 1,d\rrbracket}
\def\Gampol{\Gamma_{\text{pol}}}
\def\Cpol{C_{\text{pol}}}
\def\op{\overline{p}}
\def\talpha{{\tilde{\alpha}}}
\def\tbeta{{\tilde{\beta}}}

\providecommand{\f}{\ensuremath{\boldsymbol{f}} }
\providecommand{\x}{\ensuremath{\boldsymbol{x}} }
\providecommand{\C}{\ensuremath{\boldsymbol{C}} }

\providecommand{\g}{\ensuremath{\boldsymbol{g}} }
\providecommand{\h}{\ensuremath{\boldsymbol{h}} }
\providecommand{\bzeta}{\ensuremath{\boldsymbol{\zeta}} }
\providecommand{\p}{\ensuremath{\boldsymbol{p}} }

\providecommand{\ie}{{\it i.e.} }

\newcommand\rst[2]{{#1}_{\restriction_{#2}}}

\newcommand{\mres}{\mathbin{\vrule height 1.6ex depth 0pt width
0.13ex\vrule height 0.13ex depth 0pt width 1.3ex}}

\usepackage{algorithm}
\usepackage{algpseudocode}

\algblock[corp]{Begin}{End}

\DeclareUnicodeCharacter{00A0}{~}

\author[1]{N. Charon}
\author[2]{B. Charlier}
\author[3]{A. Trouvé}
\affil[1]{\small CIS, Johns Hopkins University, USA}
\affil[2]{\small IMAG, UMR 5149, Université de Montpellier, France}
\affil[3]{\small CMLA, UMR 8536, \'Ecole normale supérieure de Cachan, France}

\title{Metamorphoses of functional shapes in Sobolev spaces}
\date{}

\begin{document}

\maketitle

\begin{abstract}
In this paper, we describe in detail a model of geometric-functional variability between fshapes. These objects were introduced for the first time by the authors in \cite{Charlier15} and are basically the combination of classical deformable manifolds with additional scalar signal map. Building on the aforementioned work, this paper's contributions are several. We first extend the original $L^2$ model in order to represent signals of higher regularity on their geometrical support with more regular Hilbert norms (typically Sobolev). We describe the bundle structure of such fshape spaces with their adequate geodesic distances, encompassing in one common framework usual shape comparison and image metamorphoses. We then propose a formulation of matching between any two fshapes from the optimal control perspective, study existence of optimal controls and derive Hamiltonian equations and conservation laws describing the dynamics of geodesics. Secondly, we tackle the discrete counterpart of these problems and equations through appropriate finite elements interpolation schemes on triangular meshes. At last, we show a few results of metamorphosis matchings on synthetic and several real data examples in order to highlight the key specificities of the approach.     
\end{abstract}

%

\tableofcontents

\section{Introduction}
Shape or pattern analysis is a long standing and still widely studied problem that has recently found many interesting connections with fields as varied as geometry mechanics, image processing, machine learning or computational anatomy. In its simplest form, it consists in estimating/quantifying deformations between geometric objects, typically a deformable template onto a target (registration) or multiple different subjects from a population group (atlas estimation).  

There are already many existing deformation models under which registration problems may be formulated, \cite{Rueckert1999,Arsigny2006,Ashburner2007,Vercauteren2009} are examples among others where deformations belong to specific groups of diffeomorphisms. This paper falls in the context of the Large Deformation Diffeomorphic Metric Mapping (LDDMM) model \cite{Dupuis1998,Beg2005,Younes} that has found quite a lot of attention over the past decade and triggered the development of \textit{diffeomorphometry}, roughly speaking the analysis through a common Riemannian framework of the shape variability for many modalities of geometric objects including landmarks \cite{Joshi2000}, images \cite{Beg2005}, unlabeled point clouds \cite{Glaunes2004}, curves and surfaces \cite{Glaunes2006,Durrleman4,Charon2} or tensor fields \cite{Miller2009}.  

Among numerous extensions of the original LDDMM model, some works have looked into enriching the pure diffeomorphic setting in order to account for shape variations that may not be retrieved solely by deformations. This was in particular the motivation behind the concept of \textit{metamorphosis} introduced in the seminal paper \cite{Trouve1}. Metamorphoses combine diffeomorphic transport with an additional dynamic evolution of the template, and elegantly extends Riemannian metrics to these types of transformations. So far, metamorphoses have been defined and studied in the situation of landmarks, images and more recently on measures \cite{Richardson2015}.

The main contribution of this paper is to construct a generalized metamorphosis framework and corresponding matching formulation for a class of objects coined as \textit{functional shapes} in a very recent article by the authors \cite{Charlier15}. These functional shapes or fshapes are essentially scalar signals but, unlike images, supported on deformable shapes as curves, surfaces or more generally submanifolds of given dimension. In other words, they encompass mathematical objects like textured surfaces (Figure \ref{fig:example_cortical_thickness}); these are increasingly found in datasets issued from medical imaging, one common example being thickness maps estimated on anatomical membranes \cite{Lee15} or functional maps measured on cortical surfaces by fMRI. 

One of the principal difficulty in analyzing the variability of fshapes in both their geometric and texture components is that it does not exactly fall in the standard approach of shape spaces and diffeomorphometry. In \cite{Charlier15}, a first tentative extension of LDDMM was introduced in place under the name of 'tangential model' where transformations of functional shapes are basically decoupled between a diffeomorphism of the support and an additive residual signal map living on the template coordinate system. This provides a fairly simple and easy-to-implement extension of the large deformation model. There are however several downsides to this approach. The main one is that signal evolution in this tangential model is static which results in a framework that lacks all the theoretical guarantees of a real metric setting like LDDMM.

A seemingly more adequate way is to adapt the idea of image metamorphosis to our situation of deformable geometric supports, which involves the introduction of a dynamic model and metric for signal variations. This has been summarily proposed as the \textit{fshape metamorphosis} framework in \cite{Charlier15} where it was shown that we can then recover a metric structure on fshape bundles. The former paper, however, restricted to the theoretical analysis of the model in the simplest case of signal functions in the $L^2$ space and did not study more in depth the dynamics of geodesics. It also evidenced some significant limitations due to the lack of regularity in the signal part.

The present paper is meant as both a comprehensive complement and extension to \cite{Charlier15}. More specifically, we redefine functional shapes' bundles and metamorphoses in the more general context of Sobolev spaces and show that we obtain again complete metric spaces of fshapes. We then go further in formulating, in the infinite-dimensional setting, the natural generalization of registration for fshapes as a well-posed optimal control problem and deriving the Hamiltonian equations underlying dynamics of the control system. This whole framework has the interest of including within an integrated setting both large deformation registration of submanifolds as well as metamorphosis of classical images. 

Based on these results, we formulate the equivalent discrete matching problem for fshapes represented as textured polyhedral meshes and deduce an fshape matching algorithm akin to geodesic shooting schemes. The algorithm is applied on a few synthetic as well as real data examples to illustrate, in the last section, the interest of metamorphosis over the simpler tangential model as well as the possible benefits of higher regularities for signal metrics.

Authors have intended to make the paper as self-contained as possible. Yet a few definitions and derivations are not repeated within the text for the sake of concision. This is the case in particular for the issue of data fidelity terms between shapes and fshapes, which have been thoroughly studied in several previous publications we point to in section \ref{ssec:existence_solutions}.

\section{Functional Shape spaces}
\subsection{Shape spaces of submanifolds}
\label{ssec:shape_space}
We start by recalling a few concepts and definitions about classical shape spaces that we borrow in part from \cite{arguillere14:_shape}. We shall consider shapes that are geometrical objects embedded into a given ambient vector space $\R^n$. More specifically, in the case of interest of this paper, these will be submanifolds (with or without boundary) in $\R^n$ of dimension $d$, for $1\leq d \leq n$ and such that $X$ and the boundary $\partial X$ are of regularity $s$ with $s\geq 0$. Any of such submanifold $X$ may be represented using a partition of unity and parametrization functions $q \in C^{s}(M,\R^n)$ where $M$ can be for instance an open domain of $\R^d$ or the $d$-dimensional sphere $\S^d$ (in the case of a closed manifold). Moreover, each $X$ carries a volume measure given by the restriction $\mathcal{H}^{d} \mres X$ of the $d$-dimensional \textit{Hausdorff measure} in $\R^n$.

In the special $s=0$, we shall assume by convention that $X$ is a $d$-dimensional bounded \textit{rectifiable subset} of $\R^n$ (cf \cite{Federer} or \cite{Simon} for more detailed definitions), in other words that there is a countable set of Lipschitz regular parametrization functions on $\R^d$ (not just continuous) that covers $\mathcal{H}^{d}$-almost all of $X$. Rectifiable subsets include regular submanifolds as well as polyhedral meshes for instance and thus constitute a nice setting to model both discrete and continuous shapes. 

As in classical shape space theory, geometrical shapes are acted on by groups of diffeomorphisms of the ambient space $\R^n$. We will denote by $\text{Diff}^{p}_{\text{Id}}$ the group of $C^{p}$-diffeomorphisms of $\R^n$ converging to $\text{Id}$ at infinity. This is an open subset of the Banach affine space $\text{Id} + C^{p}_{0}(\R^n,\R^n)$, with $\Gamma^{p}(\R^n)\doteq C^{p}_{0}(\R^n,\R^n)$ being the set of $C^p$ vector field of $\R^n$ vanishing at infinity together with all its derivatives up to order $p$, equipped with the norm 
$$\|v\|_{p,\infty} = \sum_{i=0}^{p} \sup \left\{ \left|\frac{\partial^i v}{\partial x_1^{i_1} \ldots \partial x_n^{i_n}}(x)\right| \ | \ x \in \R^n, (i_1,\ldots,i_n) \in \N^n, \ i_1+\ldots+i_n =i \right\}$$ 

Now, $\text{Diff}^{p}_{\text{Id}}$ acts on $d$-dimensional $C^s$ submanifold $X$ for any $p\geq s$ by the simple transport equation:
\begin{equation}
 \label{eq:action_shape}
 \phi \cdot X \mapsto \phi(X)
\end{equation} 
for all $\phi \in \text{Diff}^{p}_{\text{Id}}$. If $X$ is given through a parametrization $q \in \mathcal{S} \doteq \text{Emb}^{s}(M,\R^n)$ (assuming a unique parametrization to simplify), the set of $C^s$ embeddings of $M$ into $\R^n$, then this action is just equivalent to $\phi \circ q$. It is also transitive on the set of all submanifolds given by these embeddings. When $p \geq \max\{1,s\}$, the action has additional smoothness and regularity properties that make $\mathcal{S}$ a \textbf{shape space} of order $s$ in the general vocabulary and setting of \cite{Arguillere2015}. In particular, as shown in \cite{arguillere14:_shape}, for all $q \in \text{Emb}^{s}(M,\R^n)$ the mapping $R_{q}: \ \phi \mapsto \phi \circ q $ is differentiable and its differential, denoted by $\xi_{q}: \ \Gamma^{p}(\R^n) \rightarrow C^{s}(M,\R^n)$, $v\mapsto v\circ q$ is called the \textit{infinitesimal action} of the group. In addition, for any time-dependent smooth velocity field $v \in L^2([0,1],\Gamma^{p}(\R^n))$ that is square integrable in time, the flow equation $\dot{q}(t) = \xi_{q(t)} v(t)$ with any initialization $q(0)=q_0 \in \mathcal{S}$ has a unique solution $q(\cdot) \in H^1([0,1],\mathcal{S})$, $q(t)$ being the state at time $t$.

\subsection{Large deformation metrics and LDDMM framework}
\label{ssec:lddmm}
Defining a metric on the previous shape space is done in a general way by constructing right-equivariant metrics on the acting group of diffeomorphisms \cite{Younes}. This is what is addressed by the now well-studied Large Deformation Diffeomorphic Metric Mapping (LDDMM) model where deformations are generated from Hilbert spaces of smooth vector fields. We give a brief summary in the following paragraphs. 

One starts from a Hilbert space $V$ that is assumed to be continuously embedded into one of the previous space $\Gamma^{p}(\R^n)$. In that case, the metric on $V$ which we write $\|u\|_{V}$ is controlled by the supremum norms of $u$ and its derivatives up to order $p$. In most situations, $V$ is constructed as a \textbf{Reproducing Kernel Hilbert Space} (RKHS) in which case $V$ is generated from a vector-valued kernel $K_V(x,y)$ with desired smoothness and where for any $x,y \in \R^n$, $K_{V}(x,y)$ is a $n\times n$ matrix such that $K$ satisfies the usual positive-definiteness property: 
$$ \sum_{i,j} \alpha_i^{T} K_{V}(x_i,x_j) \alpha_j > 0 $$    
for all finite sets of distinct points $x_i$ and vectors $\alpha_i$ (not simultaneously vanishing). Such kernels generally corresponds to Green's functions of some differential operators $L_{V}: \ V \rightarrow L^2(\R^n,\R^n)$ and the metric $\| \cdot \|_{V}$ has the expression
\begin{equation}
 \|u\|_{V}^{2} = \langle L_{V} u, u \rangle_{L^2}
\end{equation}
More details and examples of such kernels and operators can be found in \cite{Younes} chap.13. 

Now, since these vector fields are regular enough, as already mentioned in the previous section, the flow application of any time-dependent vector field $v \in L^2([0,1],V)$ which is the mapping $\phi_{t}$ of $\R^n$ defined by:
\begin{equation}
 \label{eq:flow}
 \left\{ \begin{array}[h]{l}
\dot{\phi}_{t} = v_{t} \circ \phi_{t} \\
\phi_{0} = \text{Id}
\end{array}
\right. 
\end{equation}
exists at all times $t \in [0,1]$ which defines a curve of diffeomorphisms in $H^{1}([0,1],\text{Diff}^{p}_{\text{Id}})$. The set of all attainable flows at time $1$, $G_{V} \doteq \{\phi_{1}^{v} \ | \ v \in L^2([0,1],V) \}$ is a subgroup of $\text{Diff}^{p}_{\text{Id}}$. In addition, it can be equipped with a right-invariant distance defined as the minimal path length or action of all curves joining two given elements in $G_V$. In other words, for any $\phi \in G_{V}$:
\begin{equation}
 \label{eq:dist_GV}
d_{G_V}(\text{Id},\phi) = \inf \left\{ \int_{0}^{1} \|v_t\|_{V}^{2} dt \ | \ v \in L^2([0,1],V), \ \phi_{1}^{v} = \phi \right\}
\end{equation}
This whole setting does not exactly correspond to a Riemannian metric but finds a nice interpretation in (infinite-dimensional) sub-Riemannian geometry where the curves $\phi_{t}$ defined by \eqref{eq:flow} may be thought as horizontal curves in $\text{Diff}^{p}_{\text{Id}}$ for the sub-Riemannian structure induced by $V$ and $\|.\|_{V}$, cf \cite{Arguillere2015b,Arguillere2015}. Minimizing paths between two diffeomorphisms of $G_V$ are thus still called a \textit{geodesics}, although it is generally in a sub-Riemannian understanding. The dynamics of these geodesics can be further described within a Hamiltonian formulation, which we shall detail later on. 

The distance \eqref{eq:dist_GV} on the deformation group $G_V$ induces in turn a distance between the shapes introduced in the previous subsection. For two $C^s$ submanifolds $X_0$ (template) and $X_1$ (target) such that $X_1$ is in the orbit of $X_0$ for the action of $G_V$,
\begin{equation}
 \label{eq:dist_S}
 d_{\mathcal{S}}(X_0,X_1) = \inf\left\{ \int_{0}^{1} \|v_t\|_{V}^{2} dt \ | \ v \in L^2([0,1],V), \ \phi_{1}^{v}(X_0) = X_1 \right\}
\end{equation}
Note that when $X_0$ and $X_1$ are parametrized by $q_0$ and $q_1$, then $X_{t} = \phi_{t}^{v}(X_0)$ is parametrized by $q(t)=\phi_{t}^{v} \circ q_0$ and by differentiating, we get back the state evolution equation $\dot{q}(t) = \xi_{q(t)} v(t)$.   

This way of quantifying shape variation is however only well-defined within the orbit of a template shape $X_0$ under the action of $G_V$. In practice, the exact matching constraint $\phi_{1}^{v}(X_0) = X_1$ is not realizable, either because the group $G_V$ may not be big enough to account for all possible deformations or because shapes might not even be diffeomorphic due to noise perturbations. This issue can be resolved generically by considering instead a variational problem of the form:
\begin{equation}
 \label{eq:registration_shape}
 \inf\left\{ \int_{0}^{1} \|v_t\|_{V}^{2} dt + g(q(1)) \ | \ \dot{q}(t) = \xi_{q(t)} v(t), \ v \in L^2([0,1],V) \right\}
\end{equation}
where $g$ is a data attachment term measuring the discrepancy between the approximate matched shape $q(1)$ with the actual target $q_1$. The minimization of \eqref{eq:registration_shape} is exactly the formulation of registration between two shapes in the LDDMM model. This can be thought as an \textbf{optimal control} type of problem in infinite dimensions since the control is here given by the time-dependent velocity field $v$; this interpretation has been thoroughly studied in \cite{arguillere14:_shape} and used in the rigorous derivation of Hamiltonian equations for the deformation dynamics, which we shall come back to in section \ref{sec:optimal_control}. 

The actual construction of the discrepancy term $g$ in \eqref{eq:registration_shape} in the situation of submanifolds is also a delicate issue. For instance, defining $g$ through the parametrization space like the $L^2$ metric $g(q(1)) = \int_{M} |q(1)(m) - q_1(m)|^2 dm$ is problematic in several respects, first because parametrizations are generally not available in practical situations where shapes are rather given as vertices with meshes and second because this type of discrepancy term is a metric between parametrizations but not necessarily between shapes, in the sense that it is \textbf{not} invariant to reparametrization.

A lot of work has been done in order to propose data attachment terms that are geometrical (invariant to reparametrization). We may cite for example the quotient Sobolev metrics on spaces of immersed curves presented in \cite{Bauer2014139,Bauer2014_Rtransforms}. An alternative path that has been actively investigated is the one of discrepancy terms obtained from geometric measure theory representations like measures \cite{Glaunes2004}, currents \cite{Glaunes2006} and more recently varifolds \cite{Charon2,Durrleman-Charon}. These have the interesting advantage of being constructible for discrete and smooth shapes of all dimensions/codimensions while being fairly simple to compute numerically. We refer to the previous papers for more detailed discussions on this topic.

\subsection{Functional shapes}
The general setting of shape spaces and large deformation models being summarized in the previous sections, we now turn to the main topic of this paper, which is about proposing an extended mathematical setting for functional shapes. The notion of functional shapes in computational anatomy was introduced originally in \cite{Charon1} and later developed into a more complete framework in \cite{Charlier15}. However, the model presented there was restricted to signals in $L^2$ spaces and it has been observed that the corresponding metrics may be too weak in some situations and generate instability in matching algorithms, cf \cite{Charlier15} section 9 and \cite{Nardi2015}. In addition, the derivation of dynamical equations in the continuous setting was left aside and just expressed for the discrete problem. In the rest of this section, we intend to set up more formal and general definitions of functional shapes, fshape spaces and metrics on these spaces. 

Functional shapes are essentially objects that correspond to signals like images but defined on deformable geometries. 
\begin{de}
 We say that the couple $(X,f)$ is a functional shape (or fshape) of regularity $s$ in $\R^n$, with $s \in \mathbb{N}$, if $X$ is a bounded $C^s$ submanifold of $\R^n$ and $f: \ X \rightarrow \R$ is a real-valued function on $X$ that belongs to $H^{s}(X)$, the set of Sobolev functions of order $s$ on $X$. 
\end{de}
Typically, we will call $X$ the \textit{geometrical support} of the fshape and $f$ the \textit{signal} attached to this support. For $s=0$, $H^{0}(X)$ is by convention the space $L^2(X)$ of square integrable functions on $X$, i.e of measurable functions $f$ such that  
\begin{equation}
 \label{eq:L2_X}
 \|f\|_{L^2(X)}^2 \doteq \int_{X} |f(x)|^2 d \Haus^{d}(x) < \infty 
\end{equation}
For $s\geq 1$, the Sobolev space $H^{s}(X)$ on the submanifold $X$ is defined in several equivalent ways in the literature. Following \cite{Aubin1982,Hebey}, it can be defined for instance as the completion of the space of smooth functions on $X$ for the norm:
\begin{equation}
 \label{eq:Hs_X}
 \|f\|_{H^{s}(X)}^2 \doteq \sum_{k=0}^{s} \|\nabla^{k}f\|_{L^2(X)}^2
\end{equation}
These are all Hilbert spaces for the inner product defined by $\langle f_1,f_2 \rangle_{H^{s}(X)} = \sum_{k=0}^{s} \langle \nabla^{k}f_1, \nabla^{k}f_2\rangle_{L^2(X)}^2$. We should precise here that for $s\geq 1$ we interpret the $s$ times covariant derivative $\nabla^s f$ of the function $f$ as a $(0,s)$ type tensor on the manifold $X$ and that $|\nabla^s f|$ denotes the trace norm of tensors given by $\sqrt{T^* T}$ where $T^*$ is the adjoint for the Riemannian metric on $X$. For example, if $s=1$, $\nabla f \in TX$ and $|\nabla f |^2$ at each $x \in X$ reduces to the usual norm of vector for the Euclidean inner product on the ambient space $\R^n$. 

\begin{figure}
\centering
 \begin{tabular}{cc}
 \includegraphics[width=8cm]{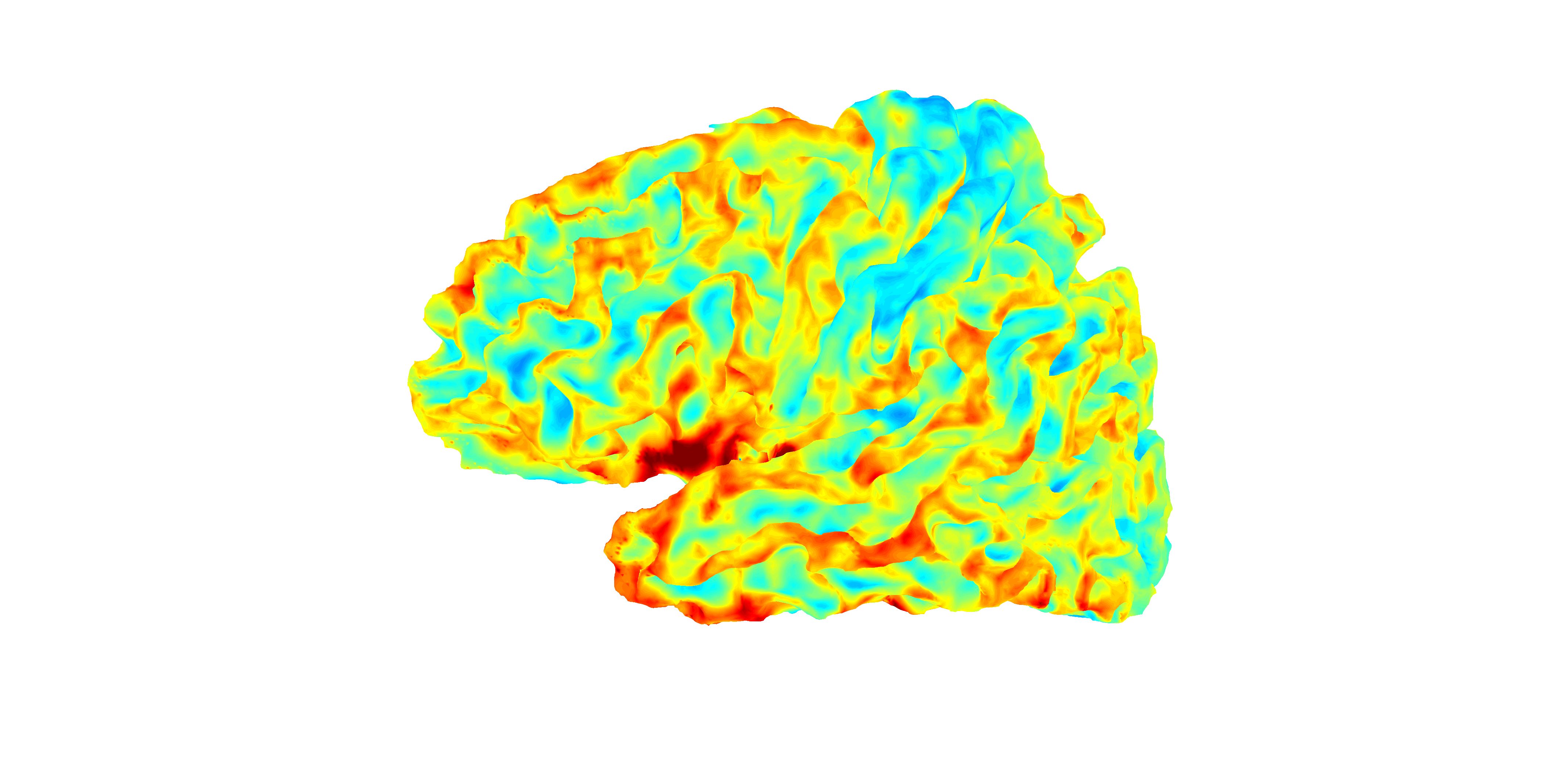} & \includegraphics[width=8cm]{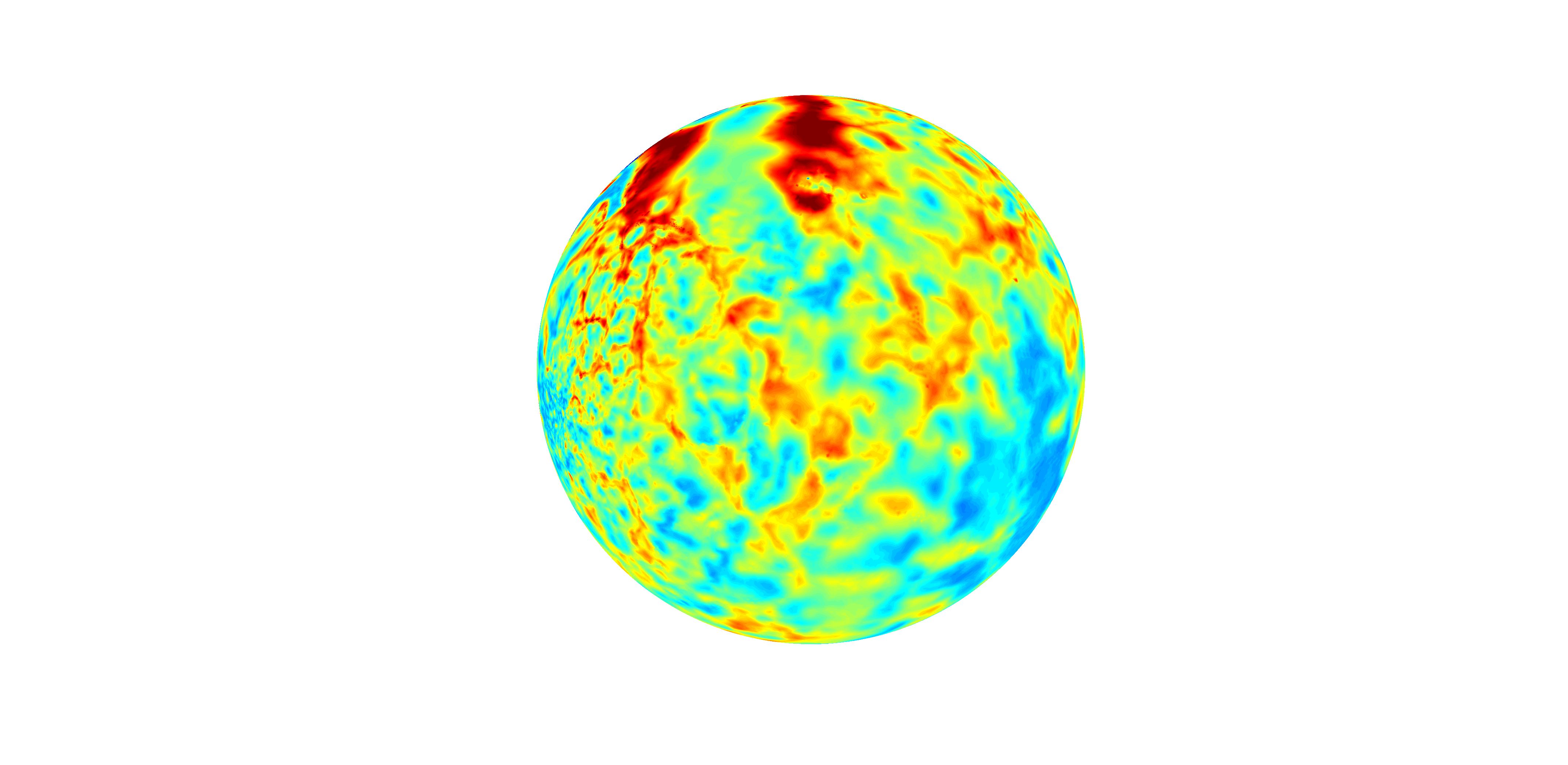} \\
 $(X,f)$ & $(q,\check{f})$
 \end{tabular}
\caption{Example of functional shape from Computational Anatomy: a cortical surface with thickness estimations (left) and a corresponding spherical parametrization (right).}
\label{fig:example_cortical_thickness}
\end{figure}

\begin{remark}
 Note that one may also define the $H^s$ norm on $X$ as follows
\begin{equation}
 \label{eq:Hs_X2}
 \|f\|_{H^{s}(X)}^2 = \sum_{k=0}^{s} \langle f, \Delta_X^k f \rangle_{L^2(X)} 
\end{equation}
where $\Delta_{X}$ denotes the \textbf{Laplace-Beltrami operator} on $X$, i.e minus the divergence of the tangential gradient on the manifold $X$. This gives a norm equivalent to \eqref{eq:Hs_X} on the subspace $H_{0}^{s}(X)$, the completion of the space of smooth compactly supported function in the interior of $X$. For $s=1$, \eqref{eq:Hs_X} and \eqref{eq:Hs_X2} are in fact exactly equal on $H_{0}^{1}(X)$ thanks to Stokes formula.
\end{remark}

We now seek a generalization of shape spaces presented in section \ref{ssec:shape_space} to structure sets of fshapes and account for combined variations in geometry and signal. Using the notations and definitions recalled in \ref{ssec:shape_space}, let $\mathcal{S}$ be a shape space of $C^s$ submanifolds (for the action of a deformation group $G \subset \text{Diff}^{p}_{\text{Id}}$, $p \geq \max\{1,s\}$). We introduce the following definition
\begin{de}
 The fshape bundle of regularity $s$ modeled on $\mathcal{S}$ is the vector bundle:
 \begin{equation}
  \label{eq:fs_bundle_def}
  \mathcal{F}_{\mathcal{S}}^{s} = \{(X,f) \ | \ X \in \mathcal{S}, \ f \in H^{s}(X) \} 
 \end{equation}
\end{de}
This is an extension to more general Sobolev spaces of the similar definition for $L^2$ that can be found in \cite{Charlier15}.

In the situations of interest for this paper, we will consider exclusively groups $G=G_V$ obtained as flows of time-dependent velocity fields modeled on an Hilbert space $V$ of vector fields with adequate regularity as explained in section \ref{ssec:lddmm}. In that case, shape spaces are generally taken as orbits for the action of $G_V$ of a particular bounded $C^s$ submanifold $X_0$ (called template), i.e $\mathcal{S} \doteq \{\phi(X_0) \ | \ \phi\in G_V\}$ which turns $\mathcal{S}$ into a homogeneous space. The previous action extends naturally to $\mathcal{F}_{\mathcal{S}}^{s}$ as follows:
\begin{equation}
\label{eq:action_defo_fshape}
  \phi \cdot (X,f) \doteq (\phi(X),f\circ \phi^{-1}) 
\end{equation}
which corresponds to the idea of deforming the geometry by $\phi$ while pulling the signal back onto the deformed shape $\phi(X)$. This is well-defined within our setting thanks to: 
\begin{lemma}
\label{lemma:Sobolev_transport1}
 For all $f \in H^s(X)$ and $\phi \in \text{Diff}^{s'}_{\text{Id}}$ with $s'=\max\{s,1\}$, $f\circ \phi^{-1} \in H^s(\phi(X))$.
\end{lemma}
This is a classical result for Sobolev spaces on compact manifolds (see for example \cite{Hebey} chap. 2). Yet, for the rest of this paper, we shall also need some more precise control of $\|f \circ \phi^{-1}\|_{H^{s}(\phi(X))}$ with respect to $\|f\|_{H^s(X)}$ and the deformation $\phi$. The essential result is the following: 
\begin{thm}
\label{theo:Sobolev_control_phi}
 There exists a polynomial function $P$ such that for any $f \in H^s(X)$ and $\phi \in \text{Diff}^{s'}_{\text{Id}}$ we have 
 \begin{equation}
  \|f \circ \phi^{-1}\|_{H^s(\phi(X))}^2 \leq P(\rho_{s'}(\phi)) \|f\|_{H^s(X)}^2
 \end{equation}
 where $\rho_s(\phi)\doteq \sum_{k\leq s'}\|d^k (\phi-\Id)\|_\infty+\|d^k(\phi^{-1}-\Id)\|_\infty$
\end{thm}
The proof is slightly technical and requires passing in local coordinates with partition of unity. It is presented with full details in Appendix \ref{appendix:proof_theorem_control_phi}. 

For diffeomorphisms belonging to a group $G_V$, Theorem \ref{theo:Sobolev_control_phi} implies the following bound:
\begin{corollary}
\label{cor:Sobolev_control_GV}
 If the Hilbert space $V$ is continuously embedded into $\Gamma^{s'}$, then there exists constants $C,\kappa \geq 0$ such that for all $\phi \in G_V$ and $f \in H^s(X)$, we have 
 \begin{equation}
  \|f \circ \phi^{-1}\|_{H^s(\phi(X))}^2 \leq C\exp(\kappa d_{G_V}(\text{Id},\phi)) \|f\|_{H^s(X)}^2
 \end{equation}
\end{corollary} 
\begin{proof}
 This is essentially a consequence of some properties of flows detailed in \cite{Younes} chap.8, in particular that when $V \hookrightarrow \Gamma^{s}$, for all $t \in [0,1]$:
 \begin{align*}
  \sum_{k\leq s}\|d^k (\phi-\Id)\|_\infty &\leq \alpha e^{\beta \int_{0}^{1} \|v_{t'}\|_V^2 dt'} \\
  \sum_{k\leq s}\|d^k (\phi^{-1}-\Id)\|_\infty &\leq \alpha e^{\beta \int_{0}^{1} \|v_{t'}\|_V^2 dt'}
 \end{align*}
where $\alpha$ and $\beta$ are two positive constants independent of $v$. In addition, using the fact that if $\phi \in G_V$, there exists $v \in L^2([0,1],V)$ such that $\phi = \phi_1^v$ and $d_{G_V}(\text{Id},\phi)^2 = \int_{0}^{1} \|v_t\|_V^2 dt$, we obtain directly the result thanks to Theorem \ref{theo:Sobolev_control_phi}.
\end{proof}

The action of $G_V$ on the fshape bundle considered so far only accounts for the geometrical part of fshape variability, or in other words for horizontal motions in the fshape bundle. To complete it, we also need to introduce vertical motions in $\mathcal{F}_{\mathcal{S}}^{s}$ which are essentially variations of signal functions within a given fiber. Thus we shall consider fshape transformations as combinations of a geometrical deformation $\phi \in G_V$ and addition of a residual signal function $\zeta$ on the signal part of the fshape. Namely, if $(X,f) \in \mathcal{F}_{\mathcal{S}}^{s}$ and $(\phi,\zeta) \in G_V \times H^{s}(X)$, we shall consider the 'action':
\begin{equation}
\label{eq:action_fshape}
  (\phi,\zeta) \cdot (X,f) \doteq (\phi(X),(f+\zeta)\circ \phi^{-1}) 
\end{equation}
Note that unlike the classical setting of shape spaces, without further assumptions, this can be no longer considered as an actual group action since the set of all transformations $(\phi,h)$ in $\mathcal{F}_{\mathcal{S}}^{s}$ is not even a group but should be rather thought as a section of the bundle $G_V \times \mathcal{F}_{\mathcal{S}}^{s}$. Yet, the previous notions together with equation \eqref{eq:action_fshape} provides a fairly natural generalization to fshapes. It is for instance easy to verify that we now recover a \textit{transitivity} property extending the one on $\mathcal{S}$, in the sense that for any fshapes $(X_1,f_1)$ and $(X_2,f_2) \in \mathcal{F}_{\mathcal{S}}^{s}$, there exists $\phi \in G_V$, $h \in H^s(X_1)$ such that $(\phi,\zeta) \cdot (X_1,f_1) = (X_2,f_2)$.

\subsection{Metamorphoses}
The question we address now is to extend the LDDMM metrics on the shape space $\mathcal{S}$ defined as in equation \eqref{eq:dist_S} to a Sobolev fshape bundle $\mathcal{F}_{\mathcal{S}}^{s}$ constructed over $\mathcal{S}$. The metrics we shall consider rely on the model of \textbf{metamorphosis}. Metamorphoses were first introduced in the case of $L^2$ images and landmarks in \cite{Trouve1} and regularly completed from the theoretical and numerical perspective thereafter. Among other references, one can quote the works of \cite{Holm2009} extending the Euler-Poincar\'e equations on diffeomorphisms to metamorphoses, or more recently \cite{Richardson2013} studying metamorphoses in spaces of discrete measures. 

Metamorphoses for fshapes have been approached (yet only superficially so far) in one previous paper by the authors \cite{Charlier15}, that partly treated the case of $L^2$ signals (fshapes of regularity 0) but mainly focused on a simplified so called 'tangential' model of fshape transformations. In the following, we build up on these results by proposing a more general metamorphosis framework also valid for fshapes of higher regularity. 

As we recalled previously, for the LDDMM model, distances on shape spaces are obtained by induction from right-invariant distances $d_{G_V}$ on the acting group of diffeomorphisms or equivalently from the infinitesimal metric $\|\cdot\|_V$ on the tangent space $V$ to $G_V$ at $\text{Id}$. In order to provide a similar sub-Riemannian structure on geometric-functional transformations, we start by introducing a dynamic model for those transformations named fshape metamorphosis. 

Let $\mathcal{F}_{\mathcal{S}}^{s}$ be a fshape bundle. If $(X,f)$ is a specific fshape in $\mathcal{F}_{\mathcal{S}}^{s}$, we define a metamorphosis of $(X,f)$ as a couple of a time-varying infinitesimal deformation $v \in L^2([0,1],V)$ and infinitesimal signal variation $h \in L^2([0,1],H^{s}(X))$. The time integration of $(v,h) \in L^2([0,1],V\times H^{s}(X))$ parametrizes an fshape transformation path $(\phi_t^v,\zeta_t^h)$ with $\phi_{t}^{v} \in G_{V}$ and $\zeta_{t}^{h} \in H^{s}(X)$ through the dynamical equations:
\begin{equation}
 \label{eq:flow_metam}
 \left\{ \begin{array}[h]{l}
 \dot{\phi_{t}^{v}} = v_{t} \circ \phi_{t}^{v} \\
 \dot{\zeta_t^h} = h_{t} \\
 \phi_0^{v} = \text{Id}, \ \zeta_0^{h} = 0
\end{array}
\right. 
\end{equation}
We then define the infinitesimal metric on $V \times H^{s}(X)$ by 
\begin{equation*}
 \|(v,h)\|_{(X,f)}^2 = \frac{\gamma_{V}}{2} \|v\|_{V}^{2} + \frac{\gamma_{f}}{2} \|h\|_{H^{s}(X)}^{2} 
\end{equation*}
where $\gamma_{V}, \gamma_f>0$ are weighting parameters. In integrated form, this gives the following energy of the path $(\phi_t^v,\zeta_t^h)$:
\begin{equation}
\label{eq:energy_def}
 E_{X}(v,h) = \frac{\gamma_{V}}{2} \int_{0}^{1} \|v_t\|_{V}^{2} dt + \frac{\gamma_{f}}{2} \int_{0}^{1} \|h_t \circ (\phi_t^v)^{-1}\|_{H^{s}(X_t)}^{2} dt 
\end{equation}
with $X_{t} \doteq \phi_{t}^{v}(X)$. Note that the penalty on the signal variation $h_t$ at each time is measured on the deformed submanifold $X_t$ with respect to the metric $\|\cdot\|_{H^{s}(X_t)}$. The framework presented in \cite{Charlier15} as \textit{tangential model} is obtained by precisely neglecting those metric changes and taking the approximation $\|h_t\|_{H^{s}(X_0)}^2$ instead. We first remark that $\|h_t \circ (\phi_t^v)^{-1}\|_{H^{s}(X_t)}$ is well-defined since thanks to Lemma \ref{lemma:Sobolev_transport1} and Corollary \ref{cor:Sobolev_control_GV}, we know that $h_t \circ (\phi_t^v)^{-1} \in H^{s}(X_t)$ and in addition we have for all $t\in [0,1]$, 
\begin{align*}
 \|h_t \circ (\phi_t^v)^{-1}\|_{H^{s}(X_t)}^2 &\leq C\exp(\kappa d_{G_V}(\text{Id},\phi_t^v)) \|f\|_{H^s(X)}^2 \\
 &\leq C\exp \left(\kappa t \left(\int_{0}^{t} \|v_{t'}\|_V^2 dt'\right)^{1/2} \right) \|f\|_{H^s(X)}^2 \\
 &\leq C\exp \left(\kappa \left(\int_{0}^{1} \|v_{t'}\|_V^2 dt'\right)^{1/2} \right) \|f\|_{H^s(X)}^2 
\end{align*}
which gives that $\int_{0}^{1} \|h_t \circ (\phi_t^v)^{-1}\|_{H^{s}(X_t)}^{2} dt$ is finite thanks to the assumptions on $v$ and $h$.  

Mimicking the previous setting on shape space, we can define a distance between two given fshapes $(X,f)$ and $(X',f')$ in the bundle $\mathcal{F}_{\mathcal{S}}^{s}$:
\begin{equation}
 \label{eq:dist_fsbundle}
 d_{\mathcal{F}_{\mathcal{S}}^{s}}((X,f),(X',f'))^{2} = \inf \{ E(v,h) \ | \ (\phi_1^v,\zeta_1^h) \cdot (X,f) = (X',f') \}
\end{equation}
This is a direct extension to fshapes of equation \eqref{eq:dist_S} in the sense that it is easy to verify that if $f$ and $f'$ are both constant and equal signals on $X$ and $X'$ then we have exactly $d_{\mathcal{F}_{\mathcal{S}}^{s}}((X,f),(X',f')) = d_{\mathcal{S}}(X,X')$. 
\begin{thm}
\label{theo:distance_fshape_bundle}
 $d_{\mathcal{F}_{\mathcal{S}}^{s}}$ is a distance on the fshape bundle $\mathcal{F}_{\mathcal{S}}^{s}$ and for all $(X,f)$ and $(X',f')$ there exists a geodesic path $(v,h) \in L^2([0,1],V\times H^s(X))$ i.e such that $d_{\mathcal{F}_{\mathcal{S}}^{s}}((X,f),(X',f'))^{2} = E_{X}(v,h)$.
\end{thm}
\begin{proof}
 The proof can be adapted from the ones of Theorems 1 and 2 in \cite{Charlier15} that deal with the case $s=0$. We repeat the essential steps with general $s$ for the sake of completeness.
 \begin{itemize}
  \item[$\bullet$] For symmetry, one simply needs to consider the time reversal of the geometric and functional velocities. For $(X,f),(X',f') \in \mathcal{F}_{\mathcal{S}}^{s}$ and any $(v,h)$ such that $\phi_1^v(X)=X'$ and $(f+\zeta_1^h)\circ (\phi_1^v)^{-1} = f'$, let:
  \begin{equation*}
   \tilde{v_t} = -v_{1-t}, \ \ \tilde{h_t} = -h_{1-t} \circ (\phi_1^{v})^{-1}
  \end{equation*}
Then it is clear that $\tilde{v} \in L^2([0,1],V)$ and from Corollary \ref{cor:Sobolev_control_GV} that $\tilde{h} \in L^2([0,1],H^s(X'))$. Also, with usual results on the flow (cf \cite{Younes} chap. 8), we know that $\phi_t^{\tilde{v}}\circ \phi_1^v \equiv \phi_{1-t}^v $ and thus $\phi_1^{\tilde{v}}(X')=X$. Similarly $(f'+\zeta_1^{\tilde{h}})\circ \phi_1^{\tilde{v}} =f$ and we have 
  \begin{align*}
   E_{X'}(\tilde{v},\tilde{h}) &=  \frac{\gamma_{V}}{2} \int_{0}^{1} \|\tilde{v}_t\|_{V}^{2} dt + \frac{\gamma_{f}}{2} \int_{0}^{1} \|\tilde{h}_t \circ (\phi_t^{\tilde{v}})^{-1}\|_{H^{s}(X_{1-t})}^{2} dt \\
   &= \frac{\gamma_{V}}{2} \int_{0}^{1} \|v_{1-t}\|_{V}^{2} dt + \frac{\gamma_{f}}{2} \int_{0}^{1} \|h_{1-t} \circ (\phi_t^{\tilde{v}} \circ \phi_1^v)^{-1}\|_{H^{s}(X_{1-t})}^{2} dt \\ 
   &= \frac{\gamma_{V}}{2} \int_{0}^{1} \|v_t\|_{V}^{2} dt + \frac{\gamma_{f}}{2} \int_{0}^{1} \|h_{1-t} \circ (\phi_{1-t}^{v})^{-1}\|_{H^{s}(X_{1-t})}^{2} dt \\
   &= E_{X}(v,h)
  \end{align*}
  By taking minimums over all $(v,h)$, we directly conclude that $d_{\mathcal{F}_{\mathcal{S}}^{s}}((X,f),(X',f')) = d_{\mathcal{F}_{\mathcal{S}}^{s}}((X',f'),(X,f))$.
  \item[$\bullet$] Triangular inequality can be obtained by concatenating path $(v,h)\in L^2([0,1],V\times H^s(X))$ from $(X,f)$ to $(X',f')$ and path $(v',h')\in L^2([0,1],V\times H^s(X'))$ from $(X',f')$ to $(X'',f'')$. The operation is defined in the following way:
  \begin{equation*}
   (\tilde{v_t},\tilde{h_t}) = (\alpha v_{\alpha t}, \alpha h_{\alpha t}) \mathds{1}_{0\leq t \leq 1/\alpha} + (\beta v_{\beta(t-1/\alpha)}, \beta h_{\beta(t-1/\alpha)} \circ \phi_1^{v}) \mathds{1}_{1/\alpha \leq t \leq 1}
  \end{equation*}
  where $\alpha, \beta$ are positive number such that $1/\alpha + 1/\beta=1$. This leads to $(\tilde{v},\tilde{h}) \in L^2([0,1],V \times H^s(X))$ with $\phi_1^{\tilde{v}}(X)=X''$ and $(f+\zeta_1^{\tilde{h}}) \circ (\phi_1^{\tilde{v}})^{-1} = f''$. Therefore, $d_{\mathcal{F}_{\mathcal{S}}^{s}}((X,f),(X',f')) \leq E_{X}(\tilde{v},\tilde{h})^{1/2}$. Moreover, it's easy to check that:
  \begin{equation*}
   E_{X}(\tilde{v},\tilde{h})^{1/2} = \frac{1}{\alpha} E_{X}(v,h)^{1/2} + \frac{1}{\beta} E_{X'}(v',h')^{1/2}
  \end{equation*}
which, by choosing $\alpha = (E_{X}(v,h)^{1/2} + E_{X'}(v',h')^{1/2})/E_{X}(v,h)^{1/2}$, gives $E_{X}(\tilde{v},\tilde{h})^{1/2} = E_{X}(v,h)^{1/2} + E_{X'}(v',h')^{1/2}$. The triangular inequality for $d_{\mathcal{F}_{\mathcal{S}}^{s}}$ follows immediately.

 \item[$\bullet$] The distance between any $(X,f),(X',f') \in \mathcal{F}_{\mathcal{S}}^{s}$ is finite. This is simply because of the transitivity of the action of $G_V \times H^s(X)$ on $\mathcal{F}_{\mathcal{S}}^{s}$. More specifically, as $X$ and $X'$ belong to $\mathcal{S}$, there exists $v\in L^2([0,1],V)$ such that $\phi_1^v(X) =X'$ by definition of $\mathcal{S}$. Now, one can set $h_t = f'\circ \phi_1^v - f$ for all $ t\in [0,1]$ and evidently $h_t \in H^s(X)$, $(f + \zeta_1^{h}) \circ (\phi_1^v)^{-1} = f'$. It results that $d_{\mathcal{F}_{\mathcal{S}}^{s}}((X,f),(X',f')) \leq E_{X}(v,h)^{1/2} < \infty$

 \item[$\bullet$] We next show that given $(X,f),(X',f') \in \mathcal{F}_{\mathcal{S}}^{s}$ there exists $(v,h)\in L^2([0,1],V\times H^s(X))$ such that $d_{\mathcal{F}_{\mathcal{S}}^{s}}((X,f),(X',f')) = E_{X}(v,h)^{1/2}$. Using the previous point and the definition of the distance, we know that there exists a sequence $(v^n,h^n) \in \left(L^2([0,1],V\times H^s(X))\right)^{\mathbb{N}}$ such that $E_{X}(v^n,h^n)^{1/2}\rightarrow d_{\mathcal{F}_{\mathcal{S}}^{s}}((X,f),(X',f'))< \infty$. This implies that the sequences $(v^n)$ is bounded in $L^2([0,1],V)$. Therefore up to an extraction, we can assume that $v^n \rightharpoonup v^{\infty}$ where $\rightharpoonup$ denotes the weak convergence. Now, $v^n \rightharpoonup v$ in $L^2([0,1],V)$ implies that $\phi_{t}^{v^{n}}$ converges to $\phi_{t}^{v^{\infty}}$ as well as all derivatives up to order $s$ uniformly on $t \in [0,1]$ and on $x$ in any compact subset of $\R^n$ (cf \cite{Younes} Theorem 8.11), and thus uniformly on $X$. On the other hand, we have that the sequence 
 \begin{equation*}
  \int_{0}^{1} \|h_t^n \circ (\phi_t^{v^n})^{-1}\|_{H^s(X_t^{n})}^2 dt
 \end{equation*}
 with $X_t^{n} \doteq \phi_t^{v^n}(X)$ is bounded. Applying Theorem \ref{theo:Sobolev_control_phi} with $\phi_{t}^{v^\infty} \circ (\phi_{t}^{v^n})^{-1}$ and using the previous uniform convergence of the $\phi_{t}^{v^{n}}$, we can see that $(h^n)$ is also bounded for the metric defined by:
 \begin{align*}
  \|h\|_{L^2([0,1],H^{s,\phi^\infty})}^2 &\doteq \int_{0}^{1} \|h_t\|_{H^{s,\phi_t^{v^\infty}}}^2 dt \\
  & = \int_{0}^{1} \|h_t \circ (\phi_t^{v^\infty})^{-1}\|_{H^s(X_t^{\infty})}^2 dt
 \end{align*}
with $X_t^{\infty} \doteq \phi_t^{v^\infty}(X)$. Therefore, up to a second extraction we can assume that there exists $h^{\infty} \in L^2([0,1],H^s(X))$ such that $h^n \rightharpoonup h^\infty$ weakly for the above metric. The next thing to show is that the functional $(v,h) \mapsto E_{X}(v,h)$ is lower (semi-)continuous for these topologies on $v$ and $h$. For the velocity field $v$, it is clear that $v\mapsto \int_0^{1} \|v_t\|_V^2 dt$ is lower semicontinuous. As for the second term, we have, using the weak semicontinuity with respect to the metric $L^2([0,1],H^{s,\phi^\infty})$, 
 \begin{align}
 \label{eq:dist_semcont1}
  &\int_{0}^{1} \|h^\infty_t \circ (\phi_t^{v^\infty})^{-1}\|_{H^{s}(X_t^{\infty})}^{2} dt \nonumber \\
  &\leq \lim \inf_{n\rightarrow +\infty} \int_{0}^{1} \|h_t^{n} \circ (\phi_t^{v^\infty})^{-1}\|_{H^{s}(X_t^{\infty})}^{2} dt
 \end{align}
Next, since $(\phi_t^{v^n})$ converges to $\phi_t^{v^\infty}$ uniformly on every compact as well as all derivatives up to order $s$, with Lemma \ref{lemma:continuity_Hsphi} in Appendix \ref{appendix:proof_theorem_control_phi}, we have for any $h \in L^2([0,1],H^s(X))$ and $t \in [0,1]$,
\begin{align}
\label{eq:dist_semcont2}
 \|h_t \circ (\phi_t^{v^\infty})^{-1}\|_{H^{s}(X_t^{\infty})} &= \|h_t\|_{H^{s,\phi_t^{v^\infty}}} \nonumber \\
 &= \lim_{n\rightarrow +\infty} \|h_t\|_{H^{s,\phi_t^{v^n}}}
\end{align}
It results from \eqref{eq:dist_semcont1} and \eqref{eq:dist_semcont2} that:
\begin{align*}
  &\int_{0}^{1} \|h^\infty_t \circ (\phi_t^{v^\infty})^{-1}\|_{H^{s}(X_t^{\infty})}^{2} dt \\
  &\leq \lim \inf_{n\rightarrow +\infty} \int_{0}^{1} \|h_t^{n} \circ (\phi_t^{v^n})^{-1}\|_{H^{s}(X_t^{n})}^{2} dt
\end{align*}
and consequently
\begin{align*}
 d_{\mathcal{F}_{\mathcal{S}}^{s}}((X,f),(X',f')) \leq E_X(v^\infty,h^\infty) \leq \lim \inf_{n\rightarrow +\infty} E_X(v^n,h^n) = d_{\mathcal{F}_{\mathcal{S}}^{s}}((X,f),(X',f'))
\end{align*}
leading to the result.

 \item[$\bullet$] Finally, we can prove that $d_{\mathcal{F}_{\mathcal{S}}^{s}}((X,f),(X',f')) = 0 \ \Rightarrow \ X=X', \ f=f'$. This is because, with the previous point, there exists $(v,h) \in L^2([0,1],V\times H^s(X))$ such that $d_{\mathcal{F}_{\mathcal{S}}^{s}}((X,f),(X',f')) = E_{X}(v,h)^{1/2} = 0$. Then $v=0$ and $h=0$ which leads to $\phi_1^{v} = \text{Id}, \ \zeta_1^{h} = 0$ and gives the desired result.

 \end{itemize} 
 \end{proof}
The fact that we eventually obtained a distance on the fshape bundle is not trivial and precisely originates from the way the energy of infinitesimal metamorphoses was defined. It's also important to remind that the simpler ``tangential'' model for fshape transformations that was detailed and exploited in \cite{Charlier15} does not provide a real distance nor even a pseudo-distance as opposed to metamorphoses. We can add to Theorem \ref{theo:distance_fshape_bundle} a few other properties of the spaces $\mathcal{F}_{\mathcal{S}}^{s}$, in particular:
\begin{property}
\label{prop:completeness}
 The space $\mathcal{F}_{\mathcal{S}}^{s}$ equipped with its distance \eqref{eq:dist_fsbundle} is a complete metric space.
\end{property}
\begin{proof}
 Consider a Cauchy sequence $(X^p,f^p)_{p \in \N}$ in $\mathcal{F}_{\mathcal{S}}^{s}$. We can assume that up to the extraction of a subsequence, we have $d_{\mathcal{F}_{\mathcal{S}}^{s}}((X^{p-1},f^{p-1}),(X^{p},f^{p})) \leq 2^{-(p-1)/2}$. Thanks to Theorem \ref{theo:distance_fshape_bundle}, we can write $X^{p} = \phi_{1}^{v^{p-1}}(X^{p-1})$ and $f^{p}= (f^{p-1} + \zeta_1^{h^{p-1}}) \circ (\phi_{1}^{v^p})^{-1}$ with $E_{X^p}(v^p,h^p) \leq 2^{-p}$. This implies in particular that $\int_{0}^{1} \|v^p_t\|_V^2 dt \leq 2^{-p}$ and consequently $\psi^{p} \doteq \phi_1^{v^{p-1}} \circ \phi_1^{v^{p-2}} \circ \ldots \circ \phi_1^{v^0}$ is a Cauchy sequence in the group $G_V$. It was shown (Theorem 8.15 in \cite{Younes}) that $G_V$ is itself a complete metric space; therefore $\psi^{p}$ converges to $\psi^{\infty}$. Let's write $X^{\infty} \doteq \psi^{\infty}(X_0)$. On the other hand, we have that $\xi^{p} = \zeta_1^{h^{p-1}} \circ \psi^{p-1} +\ldots+ \zeta_1^{h^0} \in H^{s}(X^0)$ thanks to Lemma \ref{lemma:Sobolev_transport1}. Now for all $p \in \N$, 
 \begin{align}
 \label{eq:bound_xi}
  \|\xi^{p} - \xi^{p-1}\|_{H^s(X^0)}^2 &= \|\zeta_1^{h^{p-1}} \circ \psi^{p-1}\|_{H^s(X_0)}^2 \nonumber \\
  &\leq \int_{0}^{1} \|h_t^{p-1} \circ \psi^{p-1}\|_{H^s(X_0)}^2 dt \nonumber \\
  &\leq \int_{0}^{1} \|h_t^{p-1} \circ (\phi_t^{v^{p-1}})^{-1} \circ(\phi_t^{v^{p-1}} \circ \psi^{p-1})\|_{H^s(X_0)}^2 dt \nonumber \\
  &\leq \int_{0}^{1} C\exp \left( \kappa d_{G_V}(\text{Id},\phi_t^{v^{p-1}} \circ \psi^{p-1})\right) \|h_t^{p-1} \circ (\phi_t^{v^{p-1}})^{-1}\|_{H^s(\phi_t^{v^{p-1}}(X_{p-1}))}^2 dt
 \end{align}
by using the bound of Corollary \ref{cor:Sobolev_control_GV}. Now 
\begin{align*}
 d_{g_V}(\text{Id},\phi_t^{v^{p-1}} \circ \psi^{p-1}) &\leq d_{g_V}(\text{Id},\psi^{p-1}) + d_{g_V}(\psi^{p-1},\phi_t^{v^{p-1}} \circ \psi^{p-1}) \\
 &\leq d_{g_V}(\text{Id},\psi^{p-1}) + d_{g_V}(\text{Id},\phi_t^{v^{p-1}})
\end{align*}
thanks to the right-invariance of $d_{G_V}$, and we know that $d_{g_V}(\text{Id},\psi^{p-1})$ converges to $d_{g_V}(\text{Id},\psi^{\infty})$ as $p\rightarrow \infty$ while $d_{g_V}(\text{Id},\phi_t^{v^{p-1}}) \leq 2^{p-1}$ so the first term on the right of inequality \eqref{eq:bound_xi} is bounded. It gives eventually that:
\begin{align}
 \|\xi^{p} - \xi^{p-1}\|_{H^s(X^0)}^2 &\leq \text{cst}.\int_{0}^{1} \|h_t^{p-1} \circ (\phi_t^{v^{p-1}})^{-1}\|_{H^s(\phi_t^{v^{p-1}}(X_{p-1}))}^2 dt \\
  &\leq \text{cst}.2^{p-1}
\end{align}
This shows that $\xi^{p}$ is also a Cauchy sequence in $H^{s}(X^0)$ and therefore $\xi^{p} \rightarrow \xi^{\infty} \in H^s(X^0)$. We write $f^{\infty} \doteq (f_0 + \xi^{\infty}) \circ (\psi^{\infty})^{-1} \in H^{s}(X^{\infty})$. 

The previous points show that $(X^{\infty},f^{\infty}) \in \mathcal{F}_{\mathcal{S}}^{s}$ and we only need to verify that $(X^p,f^p)$ indeed converges to $(X^{\infty},f^{\infty})$ for the metric $d_{\mathcal{F}_{\mathcal{S}}^{s}}$. To do so, we construct a path parametrized by a certain $(v,h)$ connecting $(X^p,f^p)$ to $(X^{\infty},f^{\infty})$. It is defined on dyadic intervals $[t^k,t^{k+1}]$ with $t_k \doteq \sum_{j=1}^{k} 2^{-j}$ by:
\begin{align}
 &v_t = 2^{k+1} v^{p+k}_{2^{k+1}(t-t^{k})}, \nonumber \\
 &h_t = 2^{k+1} (h^{p+k}_{2^{k+1}(t-t^{k})} \circ \psi^{p+k} \circ (\psi^{p})^{-1})
\end{align}
One can check that $v \in L^{2}([0,1],V)$, $h\in L^{2}([0,1],H^{s}(X^p))$ and that the flow of $(v,h)$ on the interval $[t^k,t^{k+1}]$ is given by $t\mapsto (\phi_{2^{k+1}(t-t_k)}^{v^{p+k}} \circ \psi^{p+k} \circ (\psi^{p})^{-1}, \zeta^{h^{p+k}_{2^{k+1}(t-t^{k})}}\circ \psi^{p+k} \circ (\psi^{p})^{-1})$. It results that for all $k \in \N$, $\phi_{t_k}^{v}(X^p)=X^{p+k}$, $(f^p + \zeta_{t_k}^{h}) \circ (\phi_{t_k}^{v})^{-1} = f^{p+k}$. Moreover, $\phi_{1}^{v}(X^p) = X^{\infty}$ and $(f^p + \zeta_{1}^{h}) \circ (\phi_{1}^{v})^{-1} = f^{\infty}$ using the convergence shown before. From the definition of the distance, we have that 
\begin{align*}
 d_{\mathcal{F}_{\mathcal{S}}^{s}}( (X^p,f^p),(X^{\infty},f^{\infty}))^2 &\leq E_{X^p}(v,h)^2 \\
 &\leq \sum_{k=0}^{\infty} \left( \int_{t_k}^{t_{k+1}} \frac{\gamma_V}{2} \|v_t\|_{V}^2 + \frac{\gamma_f}{2} \|h_t \circ \phi_{t}^{v} \|_{H^s}^2 dt \right) \\
 &\leq \sum_{k=0}^{\infty} \underbrace{\left( \int_{0}^{1} \frac{\gamma_V}{2} \|v_t^{p+k}\|_{V}^2 + \frac{\gamma_f}{2} \|h_t^{p+k} \circ (\phi_{t}^{v^{p+k}})^{-1} \|_{H^s}^2 dt \right)}_{=E_{X^{p+k}}(v^{p+k},h^{p+k})} \\
 &\leq \sum_{k=0}^{\infty} 2^{-p-k} \\
 &\leq 2^{-(p-1)}
\end{align*}
which completes the proof of Property \ref{prop:completeness}.
\end{proof}

\subsection{The embedding point of view}
All previous notions of functional shapes and metamorphoses may be transposed to the representation of shapes as parametrizations, which will be essential in particular for the theoretical derivations of the following section. Namely, we can represent any geometrical support $X$ by a $C^{s'}$-regular embedding ($s'=\max\{s,1\}$) $q \in \text{Emb}^{s'}(M,\R^n)$ where $M$ is the parameter set which is typically a compact manifold (possibly with boundary) of dimension $d$ and regularity at least $s'$, for example an open subset of $\R^d$ in the simplest situation. 

In this embedded setting, a functional shape may be equivalently given by a couple $(q,\check{f})$ where $\check{f}$ is a function on the parameter space related to $f$ by $\check{f} = f \circ q$. We give an illustration of an fshape and one parametric version in Figure \ref{fig:example_cortical_thickness}. The Sobolev metric of equation \eqref{eq:Hs_X} can be also expressed in the parameter space $M$ based on the pullback metric and covariant derivatives of tensors. For example, in the case $s=1$ and $M=\Omega \subset \R^d$ an open subset, we have
 \begin{equation}
 \label{eq:H1_Om}
  \|f\|_{H^{1}(X)}^2 = \frac{1}{2} \left( \int_{\Omega} \check{f}(m)^2 . |G_q(m)|^{1/2} dm + \int_{\Omega} (\nabla \check{f}(m))^{T} G_q(m)^{-1} \nabla \check{f}(m).|G_q(m)|^{1/2} dm \right)
 \end{equation}
where $G_q(m)$ denotes the pullback metric to $M$ from the one on $X$ induced by the Euclidean structure of $\R^n$, i.e $G_q(m) = (\partial_i q(m) \cdot \partial_j q(m))_{i,j=1,\ldots,d}$, the square root of its determinant $|G_q(m)|^{1/2}$ giving the induced volume density. More generally, for $q \in \text{Emb}^s(M,\R^n)$ and a signal $\check{f} \in H^s(M)$ the pullback $H^s$ norm on $M$ that we denote $\|\cdot\|_{H^s_q}$ can be expressed as follows:
\begin{equation}
 \label{eq:Hs_norm_immersion}
 \|\check{f}\|_{H^{s}_q}^2 = \|\check{f}\circ q^{-1}\|_{H^{s}(X)}^2 = \sum_{k=0}^{s} \int_{M} g_{k}^{0}(\nabla^{k} \check{f},\nabla^{k} \check{f}) \vol(g)  
\end{equation}
with $\nabla^k$ being a shortcut for $\nabla^{k,q}$, the $k$ times covariant derivative induced on $M$ by the embedding $q$, $g_k^0$ the induced product metric on $(0,k)$-tensors of $M$ and $\vol(g)$ the corresponding volume density as previously. We also refer to \cite{Bauer2012} for a more detailed exposition. 

With a given $q\in \text{Emb}^{s'}(M,\R^n)$, the equivalence between $f$ and the parametric representation $\check{f}$ is justified by:
\begin{lemma}
 \label{lemma:Sobolev_parametric}
 The application $f\mapsto f\circ q$ is an isomorphism between $H^s(X)$ and $H^s(M)$. In addition, there exists a constant $C\geq 0$ (depending on $q$) such that for all $f \in H^s(X)$:
 \begin{equation}
  \frac{1}{C} \|\check{f}\|_{H^s(M)} \leq \|f\|_{H^s(X)} = \|\check{f}\|_{H^s_q(M)} \leq C \|\check{f}\|_{H^s(M)}
 \end{equation}
\end{lemma}
\begin{proof}
The proof may be adapted using similar elements as in the proof of Theorem \ref{theo:Sobolev_control_phi} in Appendix \ref{appendix:proof_theorem_control_phi}. We will just indicate the main lines here. The first part of the statement is a consequence of Proposition 2.2 in \cite{Hebey}. If $\bar{g}$ and $g$ denote respectively the original Riemannian metric on $M$ and the one induced on $M$ from the restriction of the Euclidean metric on the submanifold $X$ by the embedding $q$, we know from e.g \cite{Hebey} that there exists a constant $\widetilde{C}>0$ depending on the bounds of $q$ and its first order derivatives on the compact manifold $M$ such that:
\begin{equation*}
 \frac{1}{\widetilde{C}} \bar{g} \leq g \leq \widetilde{C} \bar{g}
\end{equation*}
in the sense of bilinear forms, and similarly for the cometrics. Now, given a coordinate system on a certain neighborhood $K \subset M$, following the same reasoning as in Lemma \ref{lem:20.8.1}, we can show an equivalent equality eq.\eqref{eq:20.8.2} between coordinate derivatives of $f$ and the covariant derivatives with respect to the metric $g$ where coefficients are all bounded from above on $K$ by a certain constant (dependent on $q$ and its derivatives up to order $k$). Then we can invoke the same arguments as in the end of the proof of Theorem \ref{theo:Sobolev_control_phi} and thus obtain successively constants $C'$ and $C''$ such that:
\begin{align*}
 \sum_{k=0}^s \int_K g^0_k(\nabla^k \check{f},\nabla^k \check{f})\vol(g) \leq C' \sum_{k=0}^s\int_K|\partial^k \check{f}|^2dx \leq C' C''\int_K \bar{g}^0_k(\bar{\nabla}^k \check{f},\bar{\nabla}^k \check{f})\vol(\bar{g})
\end{align*}
and thus $\|\check{f}\|_{H^s_q} \leq C \|\check{f}\|_{H^s(M)}^2$. A reverse inequality is obtained by simply redoing the previous reasoning with $q^{-1}:X \rightarrow M$.
\end{proof}

Following these lines, we can then basically identify the previous bundle $\mathcal{F}_{\mathcal{S}}^{s}$ with the product space $\text{Emb}^{s'}(M,\R^n) \times H^s(M)$. Any fshape transformation $(\phi,\zeta)$ becomes, once put in parametrization, an element $(\phi,\check{\zeta})$ of $G_V \times H^s(M)$ that acts on $(q,\check{f})$ by: 
\begin{equation}
\label{eq:action_fshape_param}
  (\phi,\check{\zeta}) \cdot (q,\check{f}) \doteq (\phi \circ q,\check{f} + \check{\zeta} ) 
\end{equation}
It is then quite clear that this is now a group action of the direct product group $G_V \times H^s(M)$ on $\text{Emb}^{s'}(M,\R^n) \times H^s(M)$ and that the action is transitive, which turns this fshape space into a more usual shape space \cite{Arguillere2015} but for an extended group of transformations. 

The dynamics of a metamorphosis of an fshape $(q_0,\check{f}_0)$ writes:
\begin{equation}
 \label{eq:dynam_param}
 \left\{ \begin{array}[h]{l}
 \dot{q}_t = v_{t} \circ q_t \\
 \dot{\check{f_t}} = h_{t}
\end{array}
\right. 
\end{equation}
with $v \in L^{2}([0,1],V)$ and $\check{h} \in L^{2}([0,1],H^{s}(M))$. The energy of $(v,\check{h})$ corresponding to \eqref{eq:energy_def} for the embedding representation becomes:
\begin{align}
 \label{eq:energy_param1}
 E_{q_0}(v,\check{h}) &= \frac{\gamma_{V}}{2} \int_{0}^{1} \|v_t\|_{V}^{2} dt + \frac{\gamma_{f}}{2} \int_{0}^{1} \|\check{h}_t \circ q_t^{-1}\|_{H^{s}(q_{t}(M))}^{2} dt \nonumber \\
 &= \frac{\gamma_{V}}{2} \int_{0}^{1} \|v_t\|_{V}^{2} dt + \frac{\gamma_{f}}{2} \sum_{k=0}^{s} \int_{0}^{1} \int_{M} g^0_{t,k}(\nabla^k \check{f_t},\nabla^k \check{f_t})\vol(g_t) dt 
\end{align}
where we use the shortcut notation $g_t$ for the metric on $M$ obtained by pullback from the embedding $q_t$, and $\nabla$, unless stated otherwise, denote the covariant derivative for that metric.

At this point, it's important to note that if the representation of shapes as embeddings does provide an alternative setting for fshapes analysis that will be exploited in the following paragraphs, it does not directly embody the invariance of the objects to reparametrizations. This issue will be addressed separately in \ref{ssec:conservation_laws}.

\section{Matching between fshapes: optimal control formulation}
\label{sec:optimal_control}

\subsection{Inexact matching}
In the previous section, we have presented the mathematical setting to model functional shapes of Sobolev regularity, defined metamorphoses of fshapes and quantified distances on these spaces. The distance $d_{\mathcal{F}_{\mathcal{S}}^{s}}$ is only well-defined between two fshapes belonging to the same bundle. In that case, computing the distance amounts in finding a geodesic path mapping the first fshape \textbf{exactly} on the second one. As already discussed at the end of section \ref{ssec:lddmm}, this is only achievable if the geometric supports are themselves equivalent up to a diffeomorphism in the group $G_V$. 

For practical applications in shape analysis, exact registration under the previous framework is generally not relevant either because actual deformations of the geometric supports in a population of fshapes are not entirely modeled by diffeomorphisms in $G_V$ and Sobolev signal variations or because it is essential to regularize the estimated transformation to obtain more significant results from the point of view of statistical analysis. Thus, it is common to solve instead \textbf{inexact matching} problems that involve an additional data attachment (or dissimilarity) term. 

In the context of functional shapes with the metamorphosis setting that was introduced above, given parametrized template fshape $(q_0,\check{f_0})$ and a target $(q^{\text{tar}},\check{f}^{\text{tar}})$, we will focus on variational problems that have the general form:
\begin{equation}
 \label{eq:registration_fshape}
\left\{ \begin{array}[h]{l}
 (v^*,\check{h}^*) = \text{arg} \inf\left\{ E_{q_0}(v,\check{h}) + A(q_1,\check{f}_1) \ | \ v \in L^2([0,1],V), \ \check{h} \in L^2([0,1],H^s(M)) \right\} \\
 \dot{q_t}= v_t \circ q_t = \xi_{q_t} v_t \\
 \dot{\check{f_t}} = \check{h_{t}}
\end{array}
\right. 
\end{equation}
where $A$ is the data attachment term between the transformed fshape $(q_1,\check{f}_1)$ and the target, therefore measuring the registration mismatch. In other words, while $(q_1,\check{f}_1)$ belongs to the same bundle as the template by construction, $A$ can be thought as a cross-bundle term that accounts for possible variability outside the bundle. We shall keep this term as general as it can be for now but specific choices will be discussed below. Note that we have adopted here the point of view of parametrizations instead of fshapes strictly speaking, essentially as a necessary theoretical intermediate for the next developments of this section.

Equation \eqref{eq:registration_fshape} is once again an optimal control problem, this time with two controls given by the deformation field $v_t$ and the variable $\check{h_t}$ of signal transformation. The fundamental questions that are addressed in the following sections deal with the existence of such optimal controls as well as their characterization in terms of Hamiltonian dynamics that will be later exploited for the design of matching algorithms.  
 
\subsection{Existence of solutions}
\label{ssec:existence_solutions}
The existence of solutions to the problem of equation \eqref{eq:registration_fshape} depends on the properties of the data attachment term $g$. Using classical arguments of functional analysis, we have that:
\begin{thm}
\label{thm:existence1}
 If the functional $(v,\check{h})\mapsto A(q_1,\check{f_1})=A(\phi_{1}^{v}\circ q_0,\check{f_0}+\zeta_1^{\check{h}})$ is weakly lower semicontinuous in $L^2([0,1],V\times H^{s}(M))$, then there exists at least one solution to the optimal control problem in equation \eqref{eq:registration_fshape}. 
\end{thm}
\begin{proof}
 Let $(v^{n},\check{h}^{n})$ be a minimizing sequence. Then, it is clear that $(v^{n})$ must be bounded in $L^2([0,1],V)$ which, up to an extraction, implies that $v^n \rightharpoonup v^{*}$ and $\phi_{t}^{n}\doteq \phi_{t}^{v^{n}}$ converges to $\phi_{t}^{*} \doteq \phi_{t}^{v^{*}}$ uniformly on every compact and for all $t \in [0,1]$ as well as all derivatives of order at most $s$. In addition, the quantity
 \begin{equation*}
  \int_{0}^{1} \|\check{h}^n_t \circ (q^n_t)^{-1}\|_{H^{s}(q^n_{t}(M))}^{2} dt = \int_{0}^{1} \|h^n_t \circ (\phi_t^{n})^{-1}\|_{H^{s}(\phi_t^{n}\circ q(M))}^{2} dt
 \end{equation*}
is also bounded. Applying Theorem \ref{theo:Sobolev_control_phi} with $\phi_{t}^{*} \circ (\phi_{t}^{n})^{-1}$ and the previous uniform convergence of the $\phi_{t}^{n}$, we obtain that the sequence $(h^n)$ is bounded for the metric:
 \begin{equation*}
  \|h\|_{L^2([0,1],H^{s,\phi^{*}})}^2 \doteq \int_{0}^{1} \|h^n_t \circ (\phi_t^{*})^{-1}\|_{H^{s}(\phi_t^{*}\circ q(M))}^{2} dt
 \end{equation*}
 It results that we can assume, up to another extraction, that $(h^n)$ weakly converges to a certain $h^*$ in $L^2([0,1],H^{s,\phi^{*}})$. In addition, once again with Corollary \ref{cor:Sobolev_control_GV} applied to $(\phi_{t}^{*})^{-1}$, we get that there exists a constant $C$ (depending on $\phi^{*}$) such that for all $\check{h} \in L^{2}([0,1],H_{q}^{s}(M))$:
 \begin{equation*}
  \|\check{h}\|_{L^2([0,1],H_{q}^{s}(M))} \leq C \|h\|_{L^2([0,1],H^{s,\phi^{*}})}
 \end{equation*}
and adding the result of Lemma \ref{lemma:Sobolev_parametric}, there is a constant $C'$ (depending on $\phi^{*}$ and $q=q_0$) such that $\|\check{h}\|_{L^2([0,1],H^{s}(M))} \leq C' \|h\|_{L^2([0,1],H^{s,\phi^{*}})}$. Therefore, since the sequence $(h^n)$ is weakly converging to $h^*$ in $L^2([0,1],H^{s,\phi^{*}})$, we also have that $(\check{h}^n)$ is weakly converging to $\check{h}^*$ in $L^2([0,1],H^{s}(M))$. Now, repeating the same reasoning as in the proof of Theorem \ref{theo:distance_fshape_bundle}, we have on the one hand $E_{q_0}(v^{*},\check{h}^{*}) \leq \lim \inf_{n\rightarrow +\infty} E_{q_0}(v^{n},\check{h}^{n})$ using the weak convergence in $L^2([0,1],H^{s,\phi^{*}})$ and 
$$A(q_1^{*},\check{f_0}+\zeta_1^{\check{h}^*}) \leq \lim \inf_{n\rightarrow +\infty} A(q_1^{n},\check{f_0}+\zeta_1^{\check{h}^n})$$ 
since $\check{h}^{n} \rightharpoonup \check{h}^{*}$ in $L^2([0,1],H^{s}(M))$. We conclude that $(v^{*},h^{*})$ is a minimizer of \eqref{eq:registration_fshape}.
\end{proof}
The general assumption in Theorem \ref{thm:existence1} is not necessarily straightforward to verify for relevant choices of functional shapes' data attachment terms. We will quickly review a few possibilities in the following. The easiest choice for fshape parametrizations would be quite naturally:
\begin{equation*}
 A(q_1,\check{f_1}) \doteq \int_{M} |q_1(m) - q^{\text{tar}}(m)|^2 d\mathcal{H}^d(m) + \int_{M} (\check{f_1}(m) - \check{f}^{\text{tar}}(m))^2 d\mathcal{H}^d(m)
\end{equation*}
This is a simple squared $L^2$ distance of the functions' couple $(q,\check{f})$. It's not difficult to verify that this choice of $g$ leads to the desired weak semicontinuity property and thus to existence of solutions for the control problem. The fundamental issue is that such terms are comparing the parametric functions $q$ and $\check{f}$ provided such parametrizations are even obtainable in practice, and more importantly they do not compare the fshapes represented by these parametrizations. If signals $\check{f_1}$ and $\check{f}^{\text{tar}}$ are both constants on $M$, we end up again with the term of the end of section \ref{ssec:lddmm}, which is not invariant through reparametrizations.

For pure geometry, as mentioned above, there are different frameworks constructing parametrization-invariant data attachment terms. However, the adjunction of signal functions on the shapes can make some of these frameworks rather difficult to extend. The viewpoint of geometric measure theory and representation of shapes by currents or varifolds has the advantage of being fairly easy to adapt to the situation of fshapes. This has been done respectively in \cite{Charon1} and \cite{Charlier15}. We will not redo a comprehensive presentation of these concepts. To keep this section brief, let's simply recall that such terms derive from the representation of a fshape as a distribution on an extended space of point position, signal values and Grassmannian, and that, as distributions, these objects are then compared based on reproducing kernel Hilbert metrics or pseudo-metrics. For the \textbf{fvarifold} case, data attachment terms eventually take the following form:
\begin{align}
 \label{eq:fvarifold}
 &A(q_1,\check{f_1}) = \nonumber \\
 &\iint_{M \times M} k_{p}(q_{1}(m),q_1(m')) \, k_{f}(\check{f_1}(m),\check{f_1}(m')) \, k_{t}(Tq_1(m),Tq_1(m')) \, \vol(g^{q_1}) \vol(g^{q_1}) \nonumber \\
 &-2\iint_{M \times M} k_{p}(q_{1}(m),q^{\text{tar}}(m')) \, k_{f}(\check{f_1}(m),\check{f}^{\text{tar}}(m')) \, k_{t}(Tq_1(m),Tq^{\text{tar}}(m')) \, \vol(g^{q_1}) \vol(g^{q^{\text{tar}}})  \nonumber \\
 &+\iint_{M \times M} k_{p}(q^{\text{tar}}(m),q^{\text{tar}}(m')) \, k_{f}(\check{f}^{\text{tar}}(m),\check{f}^{\text{tar}}(m')) \, k_{t}(Tq^{\text{tar}}(m),Tq^{\text{tar}}(m')) \, \vol(q^{\text{tar}}) \vol(q^{\text{tar}})
\end{align}
where $Tq_1(m)$ is a shortcut notation to denote the $d$-dimensional linear subspace given by the range of $dq_1(m)$ while $k_{p}$, $k_{f}$ and $k_{t}$ are three positive kernels respectively on $\R^n$, $\R$ and the Grassmann manifold of all $d$-dimensional subspaces in $\R^n$. The essential difference with the previous $L^2$ metric is that $g$ in equation \eqref{eq:fvarifold} only depends on the fshape $(X_1,f_1)$ represented by the parametrization $(q_1,\check{f_1})$. 

The variation of these terms with respect to variations $q+\delta q$ and $\check{f}+\delta \check{f}$ has also been computed for the purely geometrical situation \cite{Charon2,Miller2015} and generalized to the functional case in \cite{Charlier15}. Without entering into all the details and proofs, if we assume that $q \in C^2(M,\R^n)$ and $\check{f} \in C^1(M)$, this variation has the form described below:
\begin{equation}
 \label{eq:fvarifold_variation}
(\delta A|(\delta q ,\delta f)) = \int_{M} [\langle \alpha, (\delta q)^{\bot} \rangle + \left( \beta (dq^{-1})^*(\nabla \check{f}) | (\delta q)^{\top} \right) + \gamma \delta \check{f}] \vol(g) + \int_{\partial M} \left \langle \eta, \delta q \right \rangle dl
\end{equation}
where $\alpha$ is a normal vector field, $\beta, \gamma$ are scalar functions on $M$ which regularities depend on the chosen kernels, $\eta$ is defined on the boundary $\partial M$ and is a vector field normal to the boundary of the submanifold, $(\delta q)^{\top},(\delta q)^{\bot}$ are respectively the tangential and orthogonal components of $\delta q$.

The only issue that we intend to address here is the one of the existence of solutions to \eqref{eq:registration_fshape} when $A$ is given by a fvarifold data attachment term. The case of metamorphoses in $L^2$ (i.e for $s=0$) was treated extensively in \cite{Charlier15} section 5. In that case, theorem \ref{thm:existence1} does not apply because fvarifold terms are generally not lower semicontinuous for the weak convergence in $L^2([0,1],L^2(M))$. Instead, the proof was based on geometric measure theory type of arguments. By omitting the technical assumptions on the required regularities of kernels, the result proven in \cite{Charlier15} (Theorem 7) translates to our situation as the following:
\begin{thm}
\label{thm:existence2}
 For sufficiently regular kernels $k_e, k_f, k_t$ and if $\gamma_V$ and $\gamma_f$ are large enough, there exists a solution in $L^2([0,1],V\times L^2(M))$ of the control problem \eqref{eq:registration_fshape} with $s=0$.  
\end{thm}
This result shows the important restriction that occurs when doing fshape metamorphoses in the space $L^2$; the existence of solutions only holds when the weight of the energy term relative to the data attachment one is large enough. This condition may also be crucial in numerical applications from a stability perspective, as evidenced in section 9 of \cite{Charlier15}. 

This is one of the motivation to extend the framework to higher regularity norms. Indeed, in the Sobolev case for $s\geq 1$, one can recover a stronger existence result using weak continuity arguments. The important result on the data attachment term that is needed is the following:
\begin{lemma}
\label{lemma:control_fvarifold}
 If $\check{\zeta}^{n} \rightarrow \check{\zeta}^*$ in $L^{2}(M)$, then $A(q,\check{f}+\check{\zeta}^{n})\rightarrow A(q,\check{f}+\check{\zeta}^{*})$ for any $q \in C^{s}(M,\R^n)$ and $\check{f} \in H^{s}(M)$, $s\geq 0$.  
\end{lemma}
An equivalent result was proven for functional currents in \cite{Charon1} (Proposition 3). The proof for functional varifolds data attachment terms can be adapted straightforwardly and is not repeated here. Note that the conclusion does \textbf{not} hold if we only have weak and not strong convergence in $L^2(M)$. Now, the consequence is the following existence theorem:
\begin{thm}
 For sufficiently regular kernels $k_e, k_f, k_t$ and $\gamma_V,\gamma_f>0$, there exists a solution $(v^*,h^*) \in L^2([0,1],V\times H^s(M))$ of the control problem \eqref{eq:registration_fshape} with $s\geq 1$.
\end{thm}
\begin{proof}
Let $(v^n,h^n)$ be a minimizing sequence for \eqref{eq:registration_fshape} with data attachment term of the form \eqref{eq:fvarifold}. Then, since $\gamma_V, \gamma_f >0$, both sequences: 
 \begin{equation*}
  \int_{0}^{1} \|v^n_t\|_{V}^{2} dt 
 \end{equation*}
and with $q^n_t \doteq \phi_t^{v^n} \circ q_0$
 \begin{equation*}
  \int_{0}^{1} \|\check{h}^n_t \circ (q^n_t)^{-1}\|_{H^{s}(q^n_{t}(M))}^{2} dt 
 \end{equation*}
are bounded. In particular, since $v^n$ is bounded in $L^2([0,1],V)$, we can assume up to an extraction that $v^n\rightharpoonup v^*$ which implies that for all $t\in [0,1]$, $\phi_{t}^{v^n}$ and its derivatives up to order $s$ converge uniformly on every compact towards $\phi^* \doteq \phi_{t}^{v^*}$. Following the same steps and notations as in the proof of Theorem \ref{thm:existence1}, we can assume that $h^n=\check{h}^n\circ q_0^{-1}$ converges to $h^*$ weakly in $L^2([0,1],H^{s,\phi^{*}})$ which implies that $E_{q_0}(v^{*},\check{h}^{*}) \leq \lim \inf_{n\rightarrow +\infty} E_{q_0}(v^{n},\check{h}^{n})$. We also have that $(\check{h}^n)$ is bounded in the space $L^2([0,1],H^s(M))$ and converges weakly to $h^*$ in $L^2([0,1],H^s(M))$. Now, since $\zeta_1^{\check{h}^n} = \int_{0}^{1} \check{h}^n_t dt$, we deduce that:
\begin{equation*}
 \|\zeta_1^{\check{h}^n}\|_{H^s(M)}^2 \leq \int_{0}^{1} \|\check{h}^n_t\|_{H^s(M)}^2 dt 
\end{equation*}
and consequently $(\zeta_1^{\check{h}^n})$ is bounded in $H^s(M)$. As $s\geq 1$, from Rellich-Kondrachov theorem (Theorem 10.1 in \cite{Hebey}), we deduce that $\zeta_1^{\check{h^n}}$ converges (strongly) to a certain $\zeta^{\infty}$ in $L^2(M)$ up to another extraction. Moreover, for any $\nu \in L^2(M)$, using the weak convergence of $\check{h}^n \rightharpoonup \check{h}^{*}$  
\begin{align*}
 \langle \nu , \zeta_{1}^{h^n} \rangle_{L^2(M)} = &\int_{0}^{1} \langle \nu , h_{t}^{n} \rangle_{L^2(M)} dt \\
 \xrightarrow[n\rightarrow \infty]{} &\int_{0}^{1} \langle \nu , h_{t}^{*} \rangle_{L^2(M)} dt \\
 = &\langle \nu , \zeta_{1}^{h^*} \rangle_{L^2(M)} \ .
\end{align*}
Since we also have $\zeta_1^{\check{h^n}} \rightarrow \zeta^{\infty}$ strongly in $L^2(M)$, it must hold that $\zeta^{\infty} = \zeta_1^{\check{h}^*}$. With the result of Lemma \ref{lemma:control_fvarifold}, it results that $A(q_1,\check{f}_0+\check{\zeta}^{n}) \xrightarrow[n\rightarrow \infty]{} A(q_1,\check{f}_0+\check{\zeta}^{*})$ and eventually
\begin{equation*}
 E_{q_0}(v^{*},\check{h}^{*}) + A(q_1,\check{f}_0+\check{\zeta}^{*}) \leq \lim \inf_{n\rightarrow +\infty} E_{q_0}(v^{n},\check{h}^{n}) + A(q_1,\check{f}_0+\check{\zeta}^{n})
\end{equation*}
leading to the fact that $(v^{*},\check{h}^{*})$ is a minimizer of \eqref{eq:registration_fshape}.
\end{proof}

\subsection{Hamiltonian equations}
\subsubsection{PMP and general equations}
Following the existence of solutions, we are now interested in their characterization. For shape matching, this is traditionally done invoking the \textbf{Pontryagin Maximum Principle} (PMP) in order to derive Hamiltonian equations of optimal solutions' dynamics \cite{Arguillere2015b,Miller2015}. We extend this approach to fshape metamorphoses and the optimal control problem given by equation \eqref{eq:registration_fshape}. Here we have two state variables $\check{f} \in H^s(M)$ and the immersion $q$ that we take in the space $C^{s'}(M,\R^n)$ with $s'=\max(s,1)$, and two time-dependent controls $v_t \in V$ and $\check{h_t} \in H^s(M)$. We introduce two additional co-state variables $p \in C^{s'}(M,\R^n)^*$ and $p^{f} \in H^{s}(M)^*$ that we call respectively the geometric and functional \textbf{momenta}, and the following Hamiltonian $H: \ C^{s'}(M,\R^n) \times H^{s}(M) \times C^{s'}(M,\R^n)^* \times H^{s}(M)^* \times V \times H^{s}(M) \rightarrow \R$ corresponding to our problem:
\begin{equation}
 \label{eq:Hamiltonian}
 H(q,\check{f},p,p^{f},v,\check{h}) \doteq (p|\xi_{q}v) + (p^{f}|\check{h}) - \frac{\gamma_V}{2} \|v\|_V^2 - \frac{\gamma_f}{2} \|\check{h}\|_{H^s_q}^2 
\end{equation}
where we remind that $\xi_{q}v = v \circ q$ is the infinitesimal action of $v$ on $q$ and $(p|\xi_{q}v)$, $(p^{f}|\check{h})$ are shortcuts notations for the duality brackets in respectively $C^{s'}(M,\R^n)$ and $H^{s}(M)$, and $\|\cdot\|_{H^s_q}$ is given by equation \eqref{eq:Hs_norm_immersion}.

Assuming additional regularity for vector fields of $V$, we obtain the following Hamiltonian equations along optimal solutions:
\begin{thm}
\label{theo_Hamiltonian_eq_metam}
 We assume that $V$ is continuously embedded into $\Gamma^{s'+1}$. If $(v,\check{h})$ is an optimal solution for \eqref{eq:registration_fshape}, there exists time-dependent co-states $p \in H^1([0,1],C^{s'}(M,\R^n)^*)$ and $p^{f} \in H^1([0,1],H^s(M)^*)$ such that:
\begin{equation}
 \label{eq:Hamiltonian_evol_fshape}
\left\{ \begin{array}[h]{l}
 \dot{q_t} = \partial_{p} H(q_t,\check{f}_t,p_t,p_t^{f},v_t,\check{h}_t) \\
 \dot{\check{f_t}}= \partial_{p^f} H(q_t,\check{f}_t,p_t,p_t^{f},v_t,\check{h}_t)\\
 \dot{p_t} = -\partial_{q} H(q_t,\check{f}_t,p_t,p_t^{f},v_t,\check{h}_t) \\
 \dot{p^f_t} = -\partial_{f} H(q_t,\check{f}_t,p_t,p_t^{f},v_t,\check{h}_t)=0 \\
 \partial_{v}H(q_t,\check{f}_t,p_t,p_t^{f},v_t,\check{h}_t)=0, \ \partial_{\check{h}}H(q_t,\check{f}_t,p_t,p_t^{f},v_t,\check{h}_t)=0
\end{array}
\right. 
\end{equation}
with the endpoint conditions
\begin{equation}
\label{eq:time_boundary_conditions}
 p_1 = -\partial_{q} A(q_1,\check{f_1}), \ \ p^f_1 = -\partial_{\check{f}} A(q_1,\check{f_1})
\end{equation}
\end{thm}
The proof is detailed in Appendix \ref{appendix:proof_theorem_Hamiltonian}.

We can go a little further by expressing the last two conditions on the controls in the previous Hamiltonian equations and get the so-called \textit{reduced Hamiltonian equations}.
\begin{corollary}
 If $(v,\check{h})$ is an optimal solution for \eqref{eq:registration_fshape}, there exists co-states $p \in H^1([0,1],C^{s'}(M,\R^n)^*)$ and $p^{f} \in H^1([0,1],H^s(M)^*)$ such that:
 \begin{equation*}
  v_t = \frac{1}{\gamma_V} K_{V} \xi_{q_t}^* p_t, \ \check{h_t} = \frac{1}{\gamma_f} F_{q_t}^s p^f_t 
 \end{equation*}
and the state variables evolution is described by the following reduced Hamiltonian equations
\begin{equation}
 \label{eq:reduced_Hamiltonian_evol_fshape}
\left\{ \begin{array}[h]{l}
 \dot{q_t} = \partial_{p} H_r(q_t,\check{f}_t,p_t,p_t^{f}) \\
 \dot{\check{f_t}}= \partial_{p^f} H_r(q_t,\check{f}_t,p_t,p_t^{f})\\
 \dot{p_t} = -\partial_{q} H_r(q_t,\check{f}_t,p_t,p_t^{f}) \\
 \dot{p^f_t} = -\partial_{\check{f}} H_r(q_t,\check{f}_t,p_t,p_t^{f})=0
\end{array}
\right. 
\end{equation}
with $p_1 = -\partial_{q} A(q_1,\check{f_1})$, $p^f_1 = -\partial_{\check{f}} A(q_1,\check{f_1})$ and
\begin{equation}
\label{eq:reduced_Hamiltonian}
 H_r(q,\check{f},p,p^f) \doteq \frac{1}{2\gamma_V} (p|K_q p) + \frac{1}{2\gamma_f} \|F_{q}^s p^f\|_{H^s_{q}(M)}^2
\end{equation}
\end{corollary}
\begin{proof}
 The optimal $v$ must satisfy for almost $t\in [0,1]$, $(\partial_{v}H(q_t,\check{f}_t,p_t,p_t^{f},v_t,\check{h}_t)|\delta v)=0$ for all $\delta v \in V$. Introducing the dual application of the infinitesimal action $\xi_{q}^*: C^{s'}(M,\R)^* \rightarrow V^*$, this  gives:
 \begin{equation*}
  (\xi_{q_t}^{*} p_t|\delta v) - \gamma_V \langle v_t,\delta v \rangle_V = 0 \Rightarrow v_t = \frac{1}{\gamma_V} K_V \xi_{q_t}^{*} p_t 
 \end{equation*}
with $K_V: V^* \rightarrow V$ being the duality operator of $V$. On the other hand, $(\partial_{\check{h}}H(q_t,\check{f}_t,p_t,p_t^{f},v_t,\check{h}_t)|\delta \check{h})=0$ for all $\delta \check{h} \in H^s(M)$ leading to:
 \begin{equation*}
  \left( p^{f}_{t}|\delta \check{h}\right) - \gamma_{f} \langle \check{h_{t}},\delta \check{h} \rangle_{H^{s}_{q}(M)} = 0
 \end{equation*}
 Note that the previous equation involves the duality in $H^{s}(M)$ for the left term and the duality in $H_{q}^{s}(M)$ for the right one. The two Hilbert norms being equivalent on $H^{s}(M)$ (Lemma \ref{lemma:Sobolev_parametric}), we can introduce the linear mapping $F_{q}^{s}:H^{s}(M)^* \rightarrow H^{s}(M)$ defined by the property:
 \begin{equation}
 \label{eq:def_Fqs}
  \langle F_{q}^{s} p^{f}, \check{h} \rangle_{H^{s}_q(M)} = \left( p^{f} | \check{h} \right)
 \end{equation}
for all $p^{f},\check{h} \in H^{s}(M)$. This leads to 
 \begin{equation*}
  \check{h}_t = \frac{1}{\gamma_f} F_{q_t}^{s} p^f_t
 \end{equation*}
 Now, plugging the expressions of the optimal $v$ and $\check{h}$ in \eqref{eq:Hamiltonian}, we obtain the so-called reduced Hamiltonian of the problem:
 \begin{align*}
  H_{r}(q,\check{f},p,p^f) &= \frac{1}{\gamma_V} (p|\xi_{q}K_V \xi_{q}^{*} p) + \left( p^{f} | \check{h} \right) - \frac{1}{2\gamma_V} \underbrace{\langle K_V \xi_{q}^{*} p,K_V \xi_{q}^{*} p \rangle_V}_{=(\xi_{q}^{*} p|K_V \xi_{q}^{*} p)}  - \frac{1}{2\gamma_f} \|F_{q}^{s} p^f\|_{H_q^s}^2 \\
  &= \frac{1}{2\gamma_V} (p|\xi_q K_V \xi_q^*p) + \frac{1}{2\gamma_f} \|F_{q}^{s} p^f\|_{H_q^s}^2 \\
  &= \frac{1}{2\gamma_V} (p|K_q p) + \frac{1}{2\gamma_f} \|F_{q}^{s} p^f\|_{H_q^s}^2
 \end{align*}
 with $K_q : C^{s'}(M,\R^n)^* \rightarrow C^{s'}(M,\R^n), \ p\mapsto \xi_q K_V \xi_q^*p$, as well as the reduced Hamiltonian equations \eqref{eq:reduced_Hamiltonian_evol_fshape}. We notice that the reduced Hamiltonian does not depend on the variable $\check{f}$ giving once again $\dot{p^f_t} = \partial_{\check{f}} H_r(q_t,\check{f}_t,p_t,p_t^{f})=0$.
\end{proof}
The operator $K_{q}$ in the expression of $H_r$ can be also written based on the expression of the kernel $K_V$: for $p \in C^{s'}(M,\R^n)^*$ associated to the vector-valued measure $dp$ on $M$, $K_{q} p$ is the vector field given by
\begin{equation}
\label{eq:expression_Kq}
 K_{q} p (\omega)= \int_{M} K_{V}(q(\omega),q(\omega')) d p(\omega')
\end{equation}
The other term in the reduced Hamiltonian involves the 'change of metric' operator $F_{q}^{s}$ that depends on the regularity of the considered signal space; it is easy to express it explicitly in the case $s=0$ (cf section \ref{sec:equations_L2} below) or implicitly with the adequate elliptic operator as in the proof of Property \ref{lemma:regularity}. 

\subsubsection{Conservation laws}
\label{ssec:conservation_laws}
There are additional symmetries that can be uncovered from the particular form of the Hamiltonian and derived from a Noether theorem's type of argument. Indeed, in the case of fshape metamorphoses, we recall the expression of the Hamiltonian:
\begin{equation}
H(q,\check{f},p,p^{f},v,\check{h}) \doteq (p|\xi_{q}v) + (p^{f}|\check{h}) - \frac{\gamma_V}{2} \|v\|_V^2 - \frac{\gamma_f}{2} \|\check{h}\|_{H^s_q}^2
\end{equation}
We can consider another right group action on the state variables by the reparametrization group $\text{Diff}^s(M)$ defined for all $\tau \in \text{Diff}^s(M)$:
\begin{equation*}
\tau \cdot (q,f) = (q\circ \tau, f \circ \tau) 
\end{equation*}
On the other hand, defining the action on the costates $(p,p^f)$ as the following pushforward operations:
\begin{equation*}
(\tau^* p | \delta q ) \doteq (p | \delta q \circ \tau^{-1}), \ \  (\tau^* p^f | \delta f ) \doteq (p^f | \delta f \circ \tau^{-1})
\end{equation*}
we observe that the Hamiltonian is then invariant to the action in the sense that: 
\begin{equation}
H(q\circ \tau,\check{f}\circ \tau,\tau^* p,\tau^* p^{f},v,\check{h}\circ \tau) = H(q,f,p,p^{f},v,\check{h})
\end{equation}
This can be checked easily by using the equivariance of the norm $\|\cdot\|_{H^s_q}$, i.e that $\|\check{h} \circ \tau\|_{H^s_{q\circ \tau}} = \|\check{h}\|_{H^s_q}$. Denoting $\mathcal{X}(M)$ the space of continuous vector fields on $M$ that are tangential to the boundary $\partial M$, this leads to the following conservation law: 
\begin{thm}
\label{thm:conservation_law}
 Along each optimal trajectory $t\mapsto (q_t,\check{f}_t)$ such that $\check{f}_t \in H^{s+1}(M)$ , we have that the following $\mu_t \in \mathcal{X}(M)^*$:
\begin{equation}
\label{eq:conservation_law_3}
\mu_t \doteq dq_t^* p_t + d\check{f_t}^* p_1^f = 0
\end{equation}
for all $t \in [0,1]$.
\end{thm}
\begin{proof}
 We introduce a one-parameter group of diffeomorphic reparametrizations of $M$, $z\mapsto \tau_z$, $z \in [-\epsilon,\epsilon]$, with $\tau_z \in \text{Diff}(M)$, $\tau_0 = \text{Id}_{M}$ and $\dot{\tau}_0=u$ with $u$ a $C^1$ vector field on $M$. Since $\tau_z(M) = M$ for all $z$, it implies that the normal component of $u$ along the boundary of the domain vanishes and so $u \in \mathcal{X}(M)$. With the actions introduced above, we have seen that for all $z \in [-\epsilon,\epsilon]$
 $$ H(\tau_z \cdot q,\tau_z \cdot \check{f},\tau_z ^* p,\tau_z ^*p^{f},v,\check{h}) = H(q,\check{f},p,p^{f},v,\check{h}) .$$
 With the assumptions made, we have $\check{f}_t \in H^{s+1}(M)$ and thus $\check{h}_t \in H^{s+1}(M)$ for all $t$, and differentiating the previous expression at $z=0$ leads to:
 \begin{align*}
  &-(p|dv \circ q(dq(u))) + (\partial_{q} (p|\xi_q v)|dq(u)) -\frac{\gamma_f}{2} (\partial_{q} \|\check{h}\|_{H^s_{q}}^2|dq(u)) + \gamma_f \langle \check{h},\nabla \check{h} \cdot u \rangle_{H^s_{q}}= 0 \\
  &\Rightarrow -(p|dv \circ q(dq(u))) + (\partial_{q} H | dq(u)) + \langle F_{q}^{s} p^f , \nabla h \cdot u \rangle_{H^s_{q}}= 0
 \end{align*}
which, by the definition of $F_q^s$, gives
\begin{equation}
\label{eq:conservation_law_proof1}
 -(p|dv \circ q(dq(u))) + (\partial_{q} H | dq(u)) +  (p^f | \nabla \check{h} \cdot u)= 0 .
\end{equation}
Now, defining $\mu_t$ as in equation \eqref{eq:conservation_law_3}, if $(q_t,\check{f_t},p_t,p_t^f)$ satisfies the Hamiltonian equations \eqref{eq:Hamiltonian_evol_fshape}, we obtain:
\begin{align*}
 (\dot{\mu}_t|u) &= \frac{d}{dt} [(p_t|dq_t(u)) + (p_1^f|d\check{h_t}(u))] \\ 
 &= (\dot{p}_t |dq_t(u)) + (p_t|dv_t \circ q_t(dq_t (u))) + (p_1^f|d\check{h}_t(u)) \\
 &= (- \partial_{q} H |dq_t(u)) + (p_t|dv_t \circ q_t(dq_t (u))) + (p_1^f|\nabla \check{h}_t \cdot u)
\end{align*}
Using \eqref{eq:conservation_law_proof1} at $(q_t,\check{f_t},p_t,p_t^f)$, we find that $(\dot{\mu}_t|u) = 0$ for any $u$ and thus the conservation of $\mu_t$. In addition, we have with \eqref{eq:time_boundary_conditions} the endpoint conditions $p_1 = - \partial_{q} A(q_1,\check{f_1}), \ p^f = - \partial_{\check{f}} A(q_1,f_1)$. Since fvarifold data attachment terms are invariant to reparametrization, i.e $A(q\circ \tau_z,\check{f}\circ \tau_z) = A(q,\check{f})$ for all $z \in [-\epsilon,\epsilon]$ we obtain by differentiating with respect to $z$ that for all $u$:
\begin{equation*}
 \left(\partial_{q} A(q_1,\check{f}_1) | dq_1.u \right) + \left(\partial_{\check{f}} A(q_1,\check{f}_1) | \nabla \check{f}_1.u \right) = -\left(p_1 | dq.u \right) -\left(p^f | \nabla \check{f}_1.u \right) = 0
\end{equation*}
or, in other words:
\begin{equation}
 dq_1^* p_1 + d\check{f_1}^* p^f = 0 .
\end{equation}
With the previous conservation of $\mu_t$, we get the result claimed in Theorem \ref{thm:conservation_law}. 
\end{proof}
This conservation law leads in particular to some properties of orthogonality for the momentum $p_t$. Indeed, since for any vector field $u \in \mathcal{X}(M)$,
\begin{equation*}
 \left( p_t | dq_t.u \right) = -\left(p^f | \nabla \check{f_t}.u \right)
\end{equation*}
we can see that $p_t$ vanishes for all tangent vector fields to $q_t(M)$ that satisfy $\nabla \check{f_t}.u$ i.e that are tangential to the level lines of the signal $f_t$.  

The only non-trivial assumption in Theorem \ref{thm:conservation_law} is the $H^{s+1}$ regularity of the signal $\check{f}_t$ (or equivalently $\check{h}_t$) along the entire trajectory. In lack of a more general result, we provide at least a sufficient condition (when $M$ has no boundary) in the property below:
\begin{property}
\label{lemma:regularity}
 Assume that $M$ is a manifold without boundary. Provided the kernels defining the fidelity term in \eqref{eq:fvarifold} are sufficiently regular, if $V$ is continuously embedded into the space $C^{2s'}(\R^n,\R^n)$ with $s'=\min\{s,1\}$, and $q_1 \in C^{2s'}(M,\R^n), \ \check{f}_1 \in H^{s+1}(M)$ then optimal solutions of \eqref{eq:registration_fshape} satisfy for all $t \in [0,1]$, $q_t \in C^{2s'}(M,\R^n)$ and $\check{f}_t \in H^{s+1}(M)$.  
\end{property}
\begin{proof}
 With the equation $\dot{q_t} = v_t \circ q_t$, it is clear that with $q_1 \in C^{2s'}(M,\R^n)$ and $v_t \in C^{2s'}(\R^n,\R^n)$ for all $t$, we get $q_t \in C^{2s'}(M,\R^n)$ for all $t$. On the other hand, the evolution of $\check{f}$ is governed by the equation $\dot{\check{f_t}} = h_t = \frac{1}{\gamma_f} F_{q_t}^{s} p^f$. Since $q_1 \in C^{2s'}(M,\R^n)$ and $\check{f}_1 \in H^{s+1}(M)$ and with the regularity assumptions on the kernels defining the fidelity term, it can be seen from \eqref{eq:fvarifold} and \eqref{eq:fvarifold_variation} that 
 \begin{equation*}
  (p^{f} | \check{h}) = -(\partial_{\check{f}} A(q_1,\check{f}_1) | \check{h}) = -\int_{M} \gamma  \check{h} \vol(g)
 \end{equation*}
 where $\gamma$ is a function which we can assume to belong to $H^1(M)$ with appropriate regularity for kernels (and since $\check{f}_1 \in H^{s+1}(M) \subset H^{1}(M)$). Now we examine the two cases:
 \begin{itemize}
  \item[$\bullet$] $s=0$: in that case, as shall be detailed in section \ref{sec:equations_L2}, $F_{q_t}^{0} p^f = -|g_t|^{-1/2} \gamma$ where $|g_t|^{1/2}$ is the volume density induced on $M$ by the embedding $q_t$. Since $q_t \in C^{2}(M,\R^n)$, we have $h_t = \frac{1}{\gamma_f} F_{q_t}^{0} p^f \in H^1(M)$ for all $t \in [0,1]$ and therefore $\check{f}_t \in H^{1}(M)$.
  \item[$\bullet$] $s\geq 1$: then $s'=s$ and we can introduce the operators $A_{q_t} \doteq \sum_{k=0}^{s} (\nabla^*)^s (\nabla)^s$ where once again $\nabla$ is the covariant derivative operator associated to the metric $g_t$ and $\nabla^*$ its adjoint for that metric. As such, $A_{q_t}$ is an elliptic self-adjoint positive differential operator on $M$ of order $2s$ and from the results of \cite{Hormander} Theorem 19.2.1, $A_{q_t}$ is a Fredholm operator from $H^{2s}(M)$ to $L^2(M)$ and since it is self-adjoint the index of the operator vanishes. Moreover, $A_{q_t}$ being positive and thus injective, it results that it is also surjective. Consequently, there exists $u \in H^{2s}(M)$ such that $A_{q_t} u = \gamma$ and by definition of $A_{q_t}$, we deduce that $\check{h}_t = \frac{1}{\gamma_f} F_{q_t}^{s} p^f = -\frac{1}{\gamma_f} u \in H^{2s}(M) \subset H^{s+1}(M)$. Now, with $\check{f}_1 \in H^{s+1}(M)$, we obtain eventually $\check{f}_t \in H^{s+1}(M)$ for all $t$.     
 \end{itemize}      
\end{proof}

\subsubsection{Link to image metamorphosis}
\label{ssec:image_metamorphosis}
As presented so far, the model of fshape metamorphoses generalizes, on the one hand, submanifold deformation and registration that corresponds to the limit case of $\gamma_f \rightarrow +\infty$ in the expression of the energy \eqref{eq:energy_def} and $k_f \equiv 1$ in the fidelity term \eqref{eq:fvarifold}. 

But it can be also viewed as extending metamorphoses of classical images studied in previous works like \cite{Trouve1,Holm2009,Richardson2015}. In the fshape perspective, this is the situation where $M=\Omega$ is a bounded domain of $\R^d$ and all geometrical shapes are fixed to $X=q(\Omega)=\Omega$. In other words, keeping the notation $\Omega \subset \R^n$ for the image domain itself, we take $V$ to be embedded into $C^{s+2}_0(\Omega)$, the space of velocity fields of class $C^{s+2}$ on $\Omega$ such that, together with all derivatives of order $\leq s$, vanish on the boundary of $\Omega$. We then obtain paths $t\mapsto q_t \in C^{s+2}(\Omega,\Omega)$ with $q_0 = \text{Id}_{\Omega}$ and $\dot{q_t} = v_t \circ q_t$. In that particular setting, this implies that for all $t$, $q_t$ identifies to the deformation $\phi_t$ itself and is in that case a $C^{s+2}$-diffeomorphism of $\Omega$. We can then introduce the change of variable $\check{h} = \zeta \circ \phi \Leftrightarrow \zeta = \check{h} \circ \phi^{-1}$, and the Hamiltonian of \eqref{eq:Hamiltonian} becomes:
 \begin{align}
 \label{eq:Hamiltonian_image}
 H(\phi,\check{f},p,p^{f},v,\zeta) &= (p|v\circ \phi) + (p^{f}|\zeta \circ \phi) - \frac{\gamma_V}{2} \|v\|_V^2 - \frac{\gamma_f}{2} \|\zeta \circ \phi\|_{H^s_q}^2  \nonumber \\
 &= (p|v\circ \phi) + (p^{f}|\zeta \circ \phi) - \frac{\gamma_V}{2} \|v\|_V^2 - \frac{\gamma_f}{2} \|\zeta\|_{H^s(\Omega)}^2
 \end{align}
With $q=\phi$ and $\zeta = \check{h} \circ \phi^{-1}$ and introducing the application $\tilde{\xi}_{\phi}: H^s(\Omega) \rightarrow H^s(\Omega), \ \zeta \mapsto \zeta \circ \phi$ and the Riesz isometry $K_{H^s}: H^s(\Omega)^*\rightarrow H^s(\Omega)$, Hamiltonian equations \eqref{eq:Hamiltonian_evol_fshape} and \eqref{eq:reduced_Hamiltonian_evol_fshape} may be rewritten as:
\begin{equation}
\label{eq:Hamiltonian_evol_image}
\left\{ \begin{array}[h]{l}
 \dot{\phi} = v \circ \phi \\
 \dot{\check{f}}= \zeta \circ \phi \\
 \dot{p} = -\partial_{\phi} (p|v\circ \phi) - \partial_{\phi} (p^{f}|\zeta \circ \phi) \\
 \dot{p^f} = 0 \\
 v=\frac{1}{\gamma_V} K_V \xi_{\phi}^{*} p , \ \zeta = \frac{1}{\gamma_f} (F_{q}^s p^f) \circ \phi^{-1} = \frac{1}{\gamma_f} K_{H^s} \left(\tilde{\xi}_{\phi}^{*} p^f \right)
\end{array}
\right. 
\end{equation}
the last equality resulting from the fact that $F_{q}^s p^f = K_{H^s} \left(\tilde{\xi}_{\phi}^{*} p^f \right) \circ \phi$ since for all $u \in H^s(\Omega)$
\begin{align*}
 \left \langle \left(K_{H^s} \tilde{\xi}_{\phi}^{*} p^f \right) \circ \phi , u \right \rangle_{H^s_q} &= \left \langle K_{H^s} \tilde{\xi}_{\phi}^{*} p^f , u \circ \phi^{-1} \right \rangle_{H^s(\Omega)} \\
 &= \left(\tilde{\xi}_{\phi}^{*} p^f | u\circ \phi^{-1} \right) \\
 &= \left(p^f | \tilde{\xi}_{\phi}(u\circ \phi^{-1}) \right) \\
 &= \left(p^f | u \right)
\end{align*}

In conclusion of this section, the Hamiltonian of \eqref{eq:Hamiltonian_image}, the Hamiltonian evolution equations \eqref{eq:Hamiltonian_evol_image} and the conservation law of Theorem \ref{thm:conservation_law} are precisely the ones of image metamorphosis given in section 2 of \cite{Richardson2015} (in the case of Sobolev metrics) which, as expected, can be treated theoretically as a special case of the functional shape setting presented here.  

\subsubsection{The particular case $s=0$}
\label{sec:equations_L2}
We now give a more specific and explicit expression for the evolution equations in the simplest case $s=0$ that corresponds to the continuous form of the discrete $L^2$ metamorphosis equations presented in \cite{Charlier15}. We make the additional regularity assumptions of theorem \ref{thm:conservation_law}, that is $q \in C^2(M,\R^n)$ and $\check{f_1} \in H^1(M)$. We can also identify $p^f$ as the $L^2$ function on $M$ given by Riesz representation theorem. The operator $F_{q}^{0}$ can be then expressed easily since:
\begin{align*}
 \langle p^f , \check{h} \rangle_{L^2(M)} &= \int_{M} p^{f}(m) \check{h}(m) d\mathcal{H}^d(m) \\
 &= \int_{M} |g|^{-1/2}(m) p^{f}(m) \check{h}(m) \vol(g) \\
 &= \langle |g|^{-1/2} p^{f}, \check{h} \rangle_{L^2_q} \ .
\end{align*}
where we write $g$ for the pullback metric induced by $q$ and $|g|^{1/2}$ the corresponding volume density. This gives $F_{q}^{0} p^f = |g|^{-1/2} p^{f}$ and the reduced Hamiltonian
\begin{equation}
 \label{eq:reduced_Hamiltonian_L2}
 H_r(q,\check{f},p,p^f) \doteq \frac{1}{2\gamma_V} (p|K_q p) + \frac{1}{2\gamma_f} \int_{M} (p^{f})^2 |g|^{-1/2} d\mathcal{H}^d
\end{equation}
The two first equations in the Hamiltonian system then write: 
\begin{equation*}
\left\{ \begin{array}[h]{l}
 \dot{q_t} = \frac{1}{\gamma_V} K_{q_t}p_t \\
 \dot{\check{f_t}}=\frac{1}{\gamma_f} |g_t|^{-1/2} p^{f}
\end{array}
\right. 
\end{equation*}
where $p^f$ is a shortcut for $p^f_{1} = p^{f}_{t}$. With the assumptions made, $p^f$ is a $H^1$ function on $M$ and the previous equation implies that for all $t$, $\check{f_t}$ is also in $H^1(M)$. Writing in short $g_t$ for the metric induced by $q_t$, the evolution of geometric momentum $p_t$ is described by:
\begin{equation*}
 (\dot{p_{t}}|\delta q) = -\frac{1}{2\gamma_V} (\partial_{q} (p_t|K_{q_t} p_t) | \delta q) -\frac{1}{2\gamma_f} \large( \partial_q \int_{M} (p^{f})^2 |g_t|^{-1/2} d\mathcal{H}^d \large| \delta q \large) 
\end{equation*}
The previous expression involves the variation of the volume density $|g|^{-1/2}$ with respect to $q$. This is given for example in \cite{Bauer2012} and leads to:
\begin{equation}
\label{eq:pt_dot}
 (\dot{p_{t}}|\delta q) = -\frac{1}{2\gamma_V} (\partial_{q} (p_t|K_{q_t} p_t) | \delta q) +\frac{1}{2\gamma_f}  \int_{M} (p^{f})^2 |g_t|^{-1/2} [\text{div}_{g_t}(\delta q^{\top}) - H_{g_t} \cdot \delta q^{\bot}]  d\mathcal{H}^d
\end{equation}
where $\delta q = \delta q^{\top} + \delta q^{\bot}$ is the decomposition of $\delta q$ in its tangential and normal components to the immersion $q_t$, $\text{div}_{g_t}(\delta q)$ is by definition the tangential divergence of the vector field $\delta q^{\top}$ and $H_{g_t}$ the mean curvature vector for the metric $g_t$. The previous equation involves two terms, the first of which is the same one appearing in Hamiltonian equations of pure geometric shape registration while the second one induces retro action of signal on geometric evolution. 

Momentum $p$ belongs a priori to the very large space of distributions $C^{1}(M,\R^n)^*$. However, with the previous assumptions, its general form can be in fact described more accurately as a vector field on $M$ plus a singular term on the boundary: 
\begin{property}
 For all $t \in [0,1]$, we have:
 $$ p_t =p_t^{in} + p_t^{bo} $$
 where $p_t^{in}$ is a vector field in $L^2(M, \R^n)$ and $p_t^{bo} = \int_{\partial M} \delta_{s} \otimes w_t(s) d\mathcal{H}^{d-1}(s)$ a vector-valued distribution supported on $\partial M$. Moreover, the tangential part of vector field $p_t^{in}$ lies in the vector bundle generated by the vector field $\nabla f_t$. 
\end{property}
\begin{proof}
As already noted before (eq.\eqref{eq:fvarifold_variation}, the boundary condition $p_1 = -\partial_{q} A(q_1,\check{f}_1)$, implies that $p_1$ decomposes as $p_1 = p_1^{in} + p_1^{bo}$ where $p_1^{in} \in L^2(M,\R^n)$ and $p_1^{bo}$ is a singular vector distribution supported on the boundary $\partial M$ of the form $p_1^{bo} = \int_{\partial M} \delta_{s} \otimes w_1(s) d\mathcal{H}^{d-1}(s)$ with $w_1$ a vector field on $\partial M$. In addition, the time derivative of $p_t$ in equation \eqref{eq:pt_dot} can be rewritten using the divergence theorem and regularity of $q$ and $p^f$ as:  
\begin{align}
 (\dot{p_{t}}|\delta q)= &-\frac{1}{2\gamma_V} (\partial_{q} (p_t|K_{q_t} p_t) | \delta q) +\frac{1}{2\gamma_f} \int_{M} \nabla (|g_t|^{-1}(p^{f})^2) \cdot \delta q^{\top} \ |g_t|^{1/2} d\mathcal{H}^d \nonumber \\ 
 &-\frac{1}{2\gamma_f} \int_{M} (p^{f})^2 |g_t|^{-1/2} H_{q_t} \cdot \delta q^{\bot} d\mathcal{H}^d + \frac{1}{2\gamma_f} \int_{\partial M} |g_t|^{-1}(p^{f})^2 \delta q^{\top} \cdot \vec{N_{q_t}} d\mathcal{H}^{d-1}
\end{align}
where once again $\nabla$ denotes the pullback covariant derivative by the embedding $q_t$, $\vec{N_{q_t}}$ the unit outward normal vector field on the boundary. Moreover, for any vector field $p^{in} \in L^2(M,\R^n)$, the expression of $K_q$ in \eqref{eq:expression_Kq} becomes:  
\begin{equation*}
 K_q p^{in} = \int_{M} K_V(q(\cdot),q(m)) p^{in}(m) d\mathcal{H}^d(m)  
\end{equation*}
and therefore
\begin{equation*}
 (p^{in}|K_{q} p^{in}) = \iint_{M \times M} p^{in}(m) \cdot K_V(q(m),q(m')) p^{in}(m') d\mathcal{H}^d(m) d\mathcal{H}^d(m')  
\end{equation*}
leading to a variation in $L^2(M,\R^n)$
\begin{equation*}
 \partial_q (p^{in}|K_{q} p^{in})(m) = \int_{M} p^{in}(m) \cdot \partial_1 K_V(q(m),q(m')) p^{in}(m') d\mathcal{H}^d(m') .
\end{equation*}
On the other hand, if $p^{bo}$ is any singular vector-valued measure on $\partial M$ of the form $p^{bo} = \int_{\partial M} \delta_{s} \otimes w(s) d\mathcal{H}^{d-1}(s)$ with $w(s) \in \R^n$ for all $s$ then:
\begin{equation*}
 K_q p^{bo} = \int_{\partial M} K_V(q(\cdot),q(s)) w(s) d\mathcal{H}^{d-1}(s) 
\end{equation*}
As previously, we obtain that the different terms $\partial_q (p^{in}|K_{q} p^{bo}), \ \partial_q (p^{bo}|K_{q} p^{in}), \ \partial_q (p^{bo}|K_{q} p^{bo})$ can be expressed either as $L^2$ vector fields or vector-valued distributions on $\partial M$. Thus, writing $\dot{p_t} = F(p_t,q_t)$, we see that the application $F$ restricted to distributions of the form $p=p^{in}+p^{bo}$ decomposes as $F(p^{in} + p^{bo},q)=F_1(p^{in},q) + F_2(p^{bo},q)$ where $F_1(\cdot,q)$ and $F_2(\cdot,q)$ are $C^1$ applications respectively from the space of $L^2$ vector fields into itself and the space of singular vector measures on $\partial M$ into itself. With the condition on $p$ at $t=1$, we deduce that at all $t$, $p_t$ is a distribution of the same form.

The last statement in the property follows from the conservation law of Theorem \ref{thm:conservation_law}. Indeed we have, for all vector field $u \in \mathcal{X}(M)$ vanishing on the boundary of $M$, $(p_t^{in}|dq_t(u)) = -(p^f|\nabla \check{f_t} \cdot u)$. We deduce that $(p_t^{in}|dq_t(u))$ vanishes for any $u$ orthogonal to the (vector) $\nabla \check{f_t}$ giving that the component of $p_t^{in}$ tangential to $q_t$ must live in the space generated by $\nabla (\check{f_t} \circ q_t^{-1}) = \nabla f_t$.
\end{proof}

\begin{figure}
\centering
 \includegraphics[width=15cm]{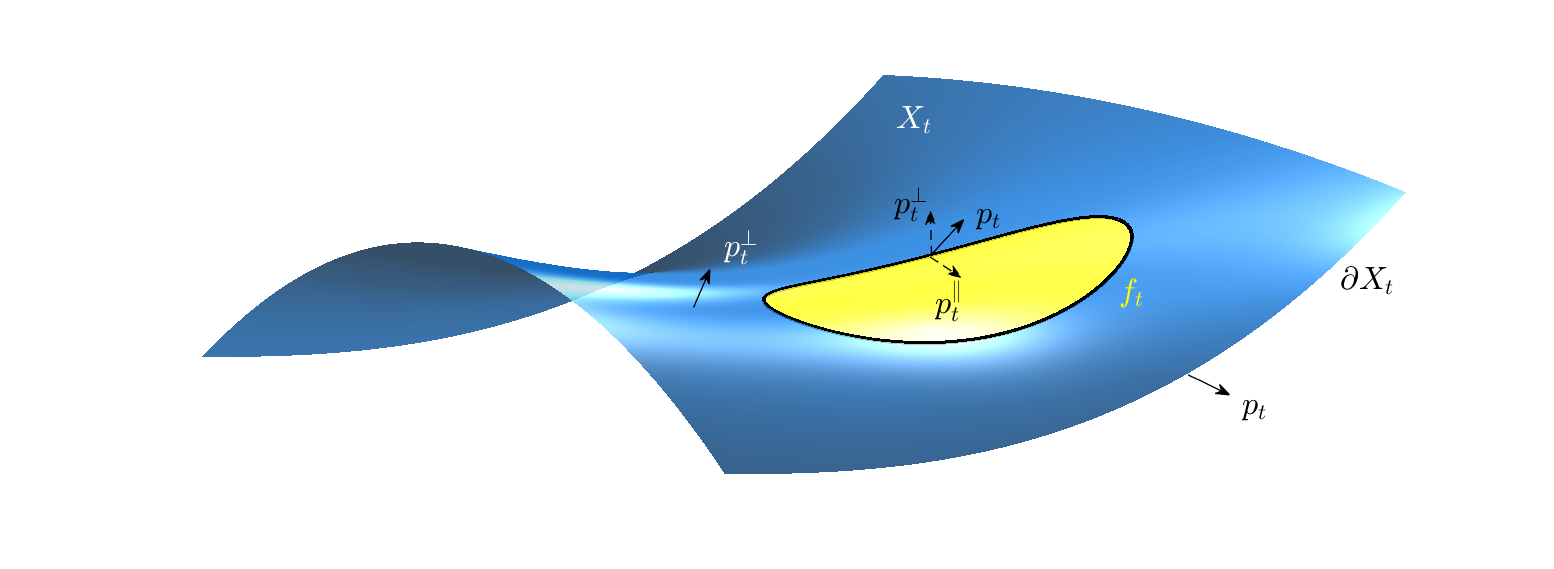}
\caption{General form of the geometric momentum field $p_t$ in metamorphosis: $p_t$ is normal to the surface within each domain of constant signal for $f_t$, components in the tangent space to $X_t$ are concentrated along the level lines and belong to the subspace generated by the gradient of $f_t$.}
\label{fig:momentum_p}
\end{figure}

This property shows in particular that the momentum $p_t$ is orthogonal to the shape at time $t$ at all points located in the interior of a level set of $\check{f_t} \circ q_t^{-1}=f_t$. In other words, tangential components in $p_t$ only appears at boundaries of the level sets of these signals, as illustrated in Figure \ref{fig:momentum_p}.  

\subsubsection{An example of geodesic trajectories}\label{part:evol_sphere_fshape}
As an explicit example of joint evolution of geometry and signal under the previous metamorphosis model in $L^2$ ($s=0$), we consider the very simple case of centered 2-dimensional spheres in $\R^3$ with constant signals. Denote by $q_t: \ \mathbb{S}^2 \rightarrow \R^3$ the parametrization of the sphere of radius $r_t$, i.e $q_t(m) = r_t m$ and with constant signals $f_t$ on $\mathbb{S}^2$. Considering only trajectories governed by constant normal momentum field $p_0 = \rho_0 m$, constant functional momentum $p^f$ and a translation/rotation invariant kernel for deformations of the form $K_V(x,y)= k_{V}(|x-y|^2) \text{I}_{3\times 3}$, it is clear that geodesic trajectories from the metamorphosis equations of previous subsection can only lead to spherical shapes with constant signals and at all times $p_t = \rho_t m$. We can thus describe geodesic trajectories by the evolution of the radius $r_t$ and the signal value $f_t$, which we will deduce from the previous reduced Hamiltonian equations.   

In this specific case, we have $|g_t|^{1/2} = r_t^2$ and consequently:
\begin{equation*}
 \dot{f_t} = \frac{p^f}{\gamma_f r_t^2}
\end{equation*}
Secondly, the velocity field $v_t \circ q_t = \gamma_V^{-1} K_{q_t}p_t$ is such that:
\begin{align*}
 K_{q_t}p_t(m) &= \left(\int_{\mathbb{S}^2} k_V(r_t^2 |m - m'|^2) m' d \Haus^2(m') \right) \rho_t \\
 &= \left(\int_{\mathbb{S}^2} k_V(2 r_t^2[1-\langle m,m' \rangle]) m' d \Haus^2(m') \right) \rho_t
\end{align*}
which leads to the following evolution on the sphere radius $r_t$:
\begin{equation*}
 \dot{r_t} = \frac{1}{\gamma_V} \left(\int_{\mathbb{S}^2} k_V(2 r_t^2[1-\langle m,m' \rangle]) \langle m,m' \rangle d \Haus^2(m') \right) \rho_t
\end{equation*}
Using Funk-Hecke formula, we can rewrite the previous as:
\begin{equation*}
 \dot{r_t} = \frac{1}{\gamma_V} \underbrace{4\pi \left(\int_{-1}^{1} u.k_V(2 r_t^2[1-u]) du \right)}_{\doteq \chi(r_t)} \rho_t
\end{equation*}
Finally, the ODE on $p_t$ translates to the following one on $\rho_t$:
\begin{align*}
 \dot{\rho_t} &= -\frac{1}{2\gamma_V} \chi'(r_t) \rho_t^2 + \frac{(p^f)^2}{\gamma_f r_t^3} 
\end{align*}
Eventually, we have obtained that the time evolution of fshapes in this situation is governed by the following three differential equations:
\begin{equation}
\label{eq:evol_sphere_fshape}
\left\{ \begin{array}[h]{l}
 \dot{f_t} = \frac{p^f}{\gamma_f r_t^2} \\
 \dot{r_t}=\frac{1}{\gamma_V} \chi(r_t) \rho_t \\
 \dot{\rho_t} = -\frac{1}{2\gamma_V} \chi'(r_t) \rho_t^2 + \frac{(p^f)^2}{\gamma_f r_t^3} 
\end{array}
\right. 
\end{equation}

\begin{figure}
\centering
 \begin{tabular}{ccc}
 \includegraphics[width=5cm]{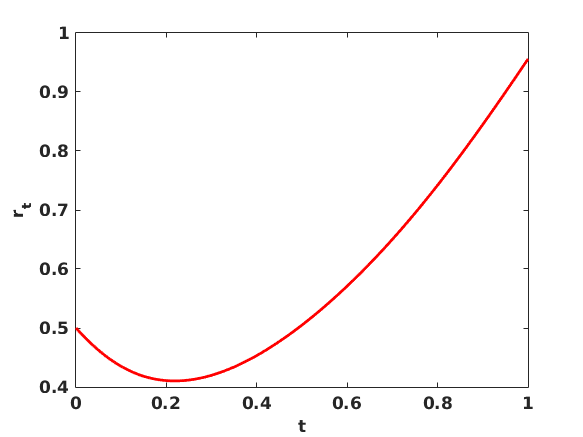} & \includegraphics[width=5cm]{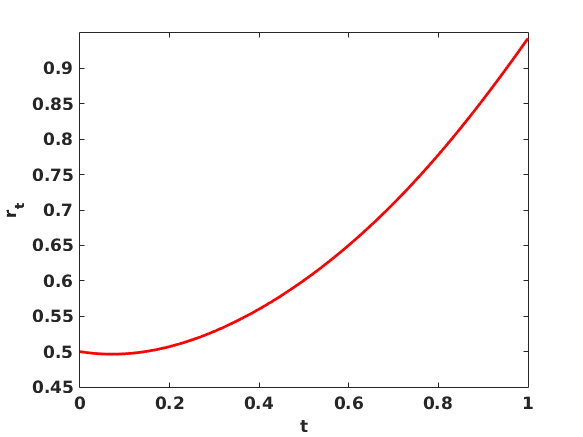} & \includegraphics[width=5cm]{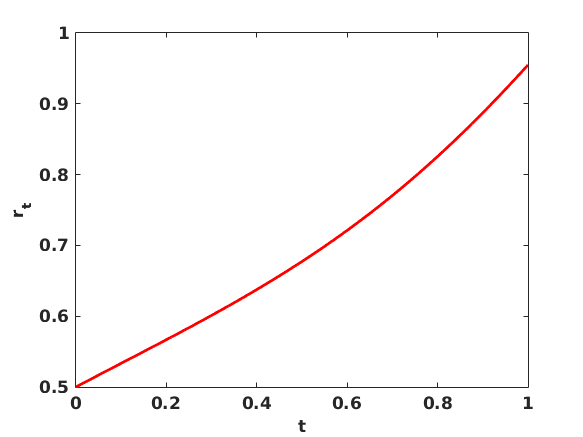} \\
 \includegraphics[width=5cm]{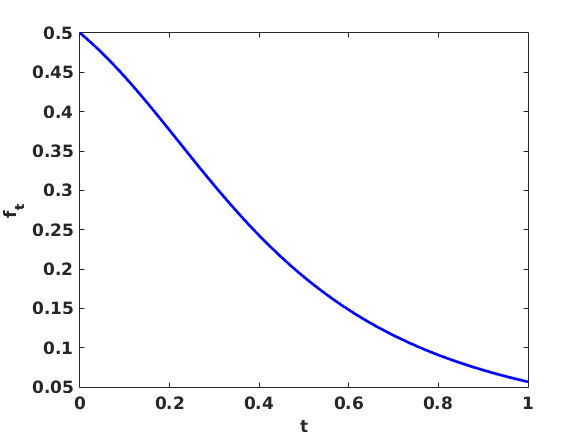} & \includegraphics[width=5cm]{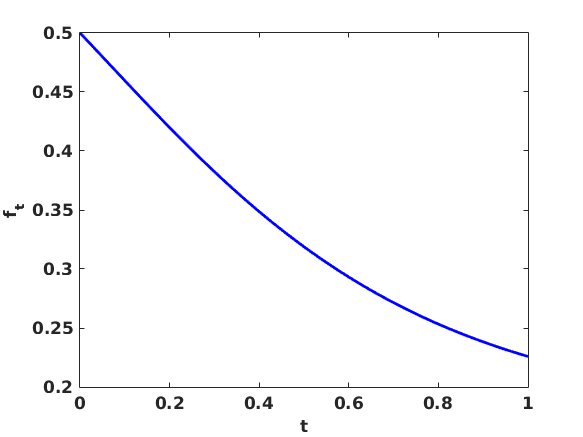} & \includegraphics[width=5cm]{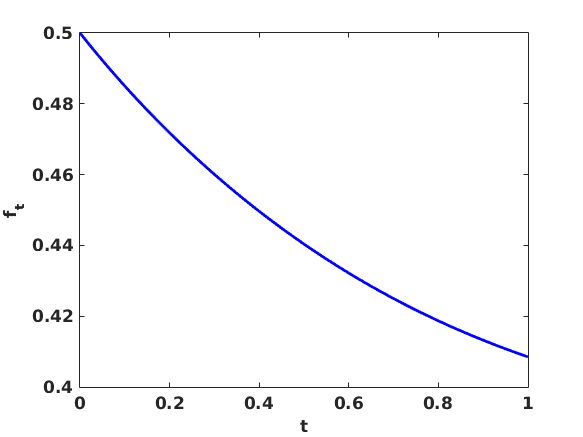} \\
 $\rho_0=-0.25$, $p^f=-0.6$ & $\rho_0=-0.03$, $p^f=-0.5$ & $\rho_0=0.1$, $p^f=-0.2$
 \end{tabular}
\caption{Sphere case: examples of evolutions of the radius and signals along geodesic trajectories for different initial momenta $\rho_0$ and $p^f$ and Gaussian kernel ($\sigma=0.3$, $\gamma_V=1, \gamma_f=5$).}
\label{fig:metam_sphere}
\end{figure}

There are several remarks to be made on the previous equations. First, we see that the speed of signal evolution is proportional to the inverse of the squared radius, thus $f_t$ will vary faster at times when the sphere is smaller in size. Second, the equations governing the radius evolution are identical to the pure LDDMM equations except for the additional recall term $(p^f)^2/(\gamma_f r_t^3)$ in the momentum dynamics. This term may 'bend' the usual trajectories of classical shape evolution as evidenced in the plots of figure \ref{fig:metam_sphere}. These plots show trajectories for $r_t$ and $f_t$ along a few geodesics, calculated from equation \eqref{eq:evol_sphere_fshape} with a Gaussian kernel $k_V(u) = e^{-\frac{u}{2\sigma^2}}$, for which one can verify that:
\begin{equation*}
 \chi(r) = 4\pi \frac{\sigma^2}{r^2} (1+e^{-\frac{2r^2}{\sigma^2}}) \left[1 - \frac{\sigma^2}{r^2} \tanh\left(\frac{r^2}{\sigma^2} \right) \right] 
\end{equation*}
The left hand figures for instance show that under certain combinations of parameters and initial conditions, the sphere may contract (while signal variation accelerates) before expanding, which is a very different behavior compared to the case of pure geometric shapes or to the 'tangential' model for fshapes developed in \cite{Charlier15}.   

\section{Discrete model}

The model of fshape metamorphosis described so far may be rewritten in a totally discrete setting, which is the essential step towards an actual matching algorithm solving numerically the minimization problem of equation \eqref{eq:registration_fshape}. Discretization schemes have already been developed in previous articles for simpler or less general models, in particular \cite{Charlier15} and \cite{Nardi2015}. The latter reference also partly addresses the important issue of $\Gamma$-convergence of the discrete solutions. In the following sections, we will first provide a generic form for the discrete evolution equations along general fshape metamorphoses. The cases of functional surfaces' metamorphoses in $L^2$ and $H^1$ are then treated more specifically with some more details on the chosen finite elements and numerical computations.


\subsection{Discrete Fshapes}

The notations and definitions in the rest of this section closely follow the ones of \cite{Charlier15}. A continuous fshape $(X,f)$ of dimension $d$ embedded in $\R^n$ is only known through a finite set of $P \geq (d+1)$ points with their attached signal and connectivity relations between vertices. An important example to keep in mind is the case of functional surface ($d=2$) coming from 3D medical imaging ($n=3$). This kind of data usually comes from a complex pipeline ranging from image acquisition to segmentation and surface extraction. In this context, the ideal underlying continuous functional surface $(X,f)$ is unknown and is approximated by a textured triangular mesh typically containing several thousands of points ($P \approx 10^4$). 

In the discrete setting, an fshape is therefore described by a triplet of objects $(\x,\f,\C)$ where 
\begin{itemize}
\item $\x = (x_k)_{k=1,\ldots,P}$ is the $P \times n$ matrix of the $P$ vertex coordinates $x_k \in \R^n$. 
\item $\f= (f_k)_{k=1,\ldots,P} \in \R^{P\times 1}$ is a column vector of signal values attached to each vertex (in Lagrangian coordinates). 
\item $\C \in \left\{ 1,\cdots,P \right\}^{T \times (d+1)}$ is a $T \times (d+1)$ connectivity matrix. The mesh is thus composed by $T>0$ simplexes of dimension $d$ so that the $\ell$-th row of $\C$ contains the indices of the $d+1$ vertices of the cell $\ell \in \left\{ 1,\cdots,T \right\}$.
\end{itemize}
In exact translation of the continuous transport equations \eqref{eq:action_fshape}, the transformation of a discrete fshape by a deformation $\phi$ and functional residual $\bzeta \in \R^{P \times 1}$ is the discrete fshape given by $\phi \cdot \x = (\phi(x_k))$, $\f+\bzeta = (f_k+\zeta_k)$ and the same connectivity matrix $\C$. 

\subsection{Discrete functional norm}

At this stage, a continuous fshape $(X,f)$ is approximated by a discrete fshape $(\x,\f,\C)$ which is nothing but a graph with a signal attached at each vertex. From this graph, we define a piecewise polyhedral domain $\mathcal T$ of $\R^n$ made of $d$-dimensional simplices whose vertices and edges are stored in $\x$ and $\C$. Now let $\tilde f: \mathcal T \to \R$ be a function satisfying $\tilde f(x_k) = f_k$. The $H^s$ norm of $\tilde f$ on $\mathcal T$ is denoted $\| \f \|_{H^s(\x)}$  (we drop the dependency of $\tilde f$) and can be written in all generality as 
\begin{equation}\label{eq.l2approx}
	\|\f\|_{H^s(\x)}^2  = \f^{T} D_s(\x) \f 
\end{equation}
where $D_s(\x)$ is a symmetric positive definite $P\times P$ matrix depending on the interpolation formula chosen to define $\tilde f$ on $\mathcal T$. The entry of $D_s(\x)$ may be computed from the matrices $\x$ and $\C$ and 
 $D_s(\x)$ is generally sparse. In the following subsections, we will examine the most useful cases in practice: $d=1$ where $\mathcal T$ is the union of piecewise linear segments and $d=2$ where $\mathcal T$ is the union of piecewise triangular cells.
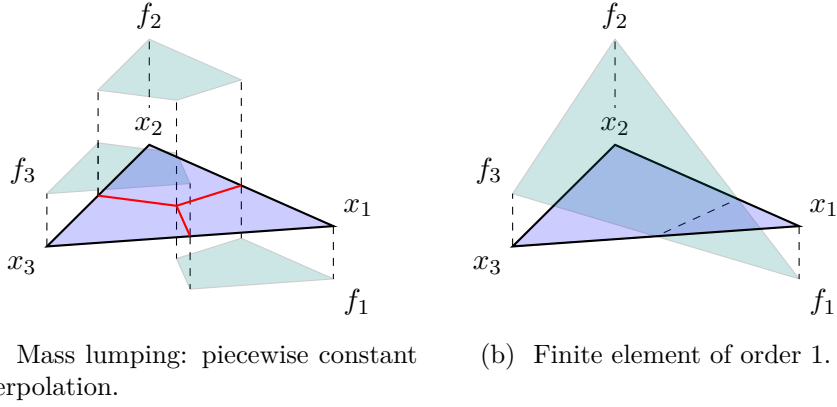
\begin{figure}[h!]
		\centering

		\begin{subfigure}[t]{6cm}
		\centering
			\begin{tikzpicture}[scale=.7]
				\coordinate (X2) at (0,0,0);
				\coordinate (X1) at (5,0,4);
				\coordinate (X3) at (0,0,5);

				\coordinate (X0) at (5/3,0,3);
				\coordinate (X12) at (5/2,0,2);
				\coordinate (X13) at (5/2,0,4.5);
				\coordinate (X23) at (0,0,2.5);

				\coordinate (F2)   at (0  ,2,0);
				\coordinate (FF0)  at (5/3,2,3);
				\coordinate (FF12) at (5/2,2,2);
				\coordinate (FF23) at (0 ,2,2.5);

				\coordinate (F3) at (0,1,5,);
				\coordinate (FFF0) at (5/3,1,3);
				\coordinate (FFF13) at (5/2,1,4.5);
				\coordinate (FFF23) at (0,1,2.5);

				\coordinate (F1) at  (5,-1,4);
				\coordinate (F12) at (5/2,-1,2);
				\coordinate (F13) at (5/2,-1,4.5);
				\coordinate (F0) at (5/3,-1,3);


				\draw[dashed] (X1) node[above right] {$x_{1}$} -- (F1) node[below right] {$f_{1}$} ;
				\draw[dashed] (X0) node[right] {} -- (F0) node[right] {} ;
				\draw[dashed] (X13) node[right] {} -- (F13) node[right] {} ;
				\draw[dashed] (X12) node[right] {} -- (F12) node[right] {} ;

				\draw[dashed] (X2) node[fill=white,above] {$x_{2}$} -- (F2) node[above] {$f_{2}$} ;
				\draw[dashed] (X0) node[left] {} -- (FF0) node[right] {} ;
				\draw[dashed] (X23) node[left] {} -- (FF23) node[right] {} ;
				\draw[dashed] (X12) node[left] {} -- (FF12) node[right] {} ;

				\draw[dashed] (X3)  node[below left] {$x_{3}$}-- (F3) node[above left] {$f_{3}$} ;
				\draw[dashed] (X0) node[left] {} -- (FFF0) node[right] {} ;
				\draw[dashed] (X23) node[left] {} -- (FFF23) node[right] {} ;
				\draw[dashed] (X13) node[left] {} -- (FFF13) node[right] {} ;


				\draw[fill=blue!50!green, opacity=.2] (F1) -- (F13) -- (F0) -- (F12) -- cycle;
				\draw[fill=blue!50!green, opacity=.2] (F2) -- (FF12) -- (FF0) -- (FF23) -- cycle;
				\draw[fill=blue!50!green, opacity=.2] (F3) -- (FFF23) -- (FFF0) -- (FFF13) -- cycle;

				\draw[thick,fill=blue,fill opacity=.2] (X1) -- (X2) -- (X3)  -- cycle;
				\draw[thick,red] (X12)  node[above right] {}-- (X0) node[below left] {};
				\draw[thick,red] (X13)  node[above right] {}-- (X0) node[below left] {};
				\draw[thick,red] (X23)  node[above right] {}-- (X0) node[below left] {};

			\end{tikzpicture}	
		\caption{\label{fig.lump} Mass lumping: piecewise constant interpolation.}
		\end{subfigure}
		\begin{subfigure}[t]{6cm}
		\centering
			\begin{tikzpicture}[scale=.7]
				\coordinate (X2) at (0,0,0);
				\coordinate (X1) at (5,0,4);
				\coordinate (X3) at (0,0,5);

				\coordinate (X0) at (5/3,0,3);
				\coordinate (X12) at (5/2,0,2);
				\coordinate (X13) at (5/2,0,4.5);
				\coordinate (X23) at (0,0,2.5);

				\coordinate (F2)   at (0  ,2,0);
				\coordinate (FF0)  at (5/3,2,3);
				\coordinate (FF12) at (5/2,2,2);
				\coordinate (FF23) at (0 ,2,2.5);

				\coordinate (F3) at (0,1,5,);
				\coordinate (FFF0) at (5/3,1,3);
				\coordinate (FFF13) at (5/2,1,4.5);
				\coordinate (FFF23) at (0,1,2.5);

				\coordinate (F1) at  (5,-1,4);
				\coordinate (F12) at (5/2,-1,2);
				\coordinate (F13) at (5/2,-1,4.5);
				\coordinate (F0) at (5/3,-1,3);

				\coordinate (X4) at (intersection of X1--X2 and F1--F2);
				\coordinate (X5) at (intersection of X1--X3 and F1--F3);

				\draw[dashed] (X1) node[above right] {$x_{1}$} -- (F1) node[below right] {$f_{1}$} ;
				\draw[dashed] (X2) node[fill=white,above] {$x_{2}$} -- (F2) node[above] {$f_{2}$} ;
				\draw[dashed] (X3)  node[below left] {$x_{3}$}-- (F3) node[above left] {$f_{3}$} ;
				\draw[dashed] (X3) -- (X2) -- (X4) -- (X5) -- cycle;
				\draw[thick,fill=blue, fill opacity=.2] (X5) -- (X3) -- (X2) -- (X4) ;

				\draw[fill=blue!50!green, opacity=.2] (F1) -- (F3) -- (F2) -- cycle;
				\draw[thick,fill=blue,fill opacity=.2] (X4) -- (X1) -- (X5);
			\end{tikzpicture}
		\caption{\label{fig.l2p2} Finite element of order 1.}
		\end{subfigure}
		\caption{\label{fig.fem} An illustration of two different interpolations to defined $\tilde f$ on $\mathcal T$ composed by a single triangle. The graph of $\tilde f$ is in green.}
	\end{figure}

\subsubsection{Mass lumping}
\label{ssec:mass_lumping}

This formula is used to compute the $L^2$ norm (i.e. the $H^s$ norm with $s=0$) of a piecewise constant $\tilde f$ on $\mathcal T$ as in Figure \ref{fig.lump}. The idea is to choose an interpolation scheme with a diagonal weight matrix. We let 
\begin{equation}
	D_0(\x)= \operatorname{Diag} \left (\frac{1}{d+1}\sum_{\tau\ni x_k} r_\tau \right)_{k=1}^{P}
	\label{eq.lump}
\end{equation}
where $r_{\tau}$ is the $d$-volume of simplex $\tau$. If  $\mathcal T$ is triangular mesh ($d=2$), it means that the $k$-th diagonal entry of $D_0(\x)$ is computed by performing a sum of the areas of all triangles $\tau \in \mathcal T$ containing the $k$-th vertex (of coordinate vector $x_k$).

\subsubsection{Exact formula for P1 finite elements}

Let $(\psi^{(0)}_\tau)_{\tau\in\mathcal T}$ be the canonical basis for the finite elements of order 0 (i.e. $\psi_\tau^{(0)}:\mathcal T \to \R$ is equal to 1 on the cell $\tau$ and 0 everywhere else) and  $(\psi^{(1)}_k)_{k=1}^P$ be the canonical basis for the finite elements of order 1 (i.e. $\psi_k^{(1)}:\mathcal T \to \R$ is continuous piecewise linear such that $\psi_k^{(1)}(x_\ell) = 1$ if $k=\ell$ and $0$ if $k\neq \ell$). 

Let $\tilde f = \sum_{k=1}^P f_k \psi_k^{(1)}$ be the function defined on $\mathcal T$ by piecewise linear interpolation of the $f_k$'s with P1 finite elements. Using standard numerical integration formula 
as in \cite{Allaire} page 178 
we have 
\[
	\|\tilde f\|^2_{L^2(\mathcal T)} = \|\f\|^2_{L^2(\x)} \doteq  \begin{cases}\displaystyle \sum_{\tau\in \mathcal T} \frac{r_\tau}{6} \Big( \big(f^{(1)}_\tau\big)^2 + 4\big(f^{(12)}_\tau\big)^2 + \big(f^{(2)}_\tau\big)^2 \Big), & \text{ if $d=1$}\\\displaystyle  \sum_{\tau\in \mathcal T} \frac{r_\tau}{3} \Big( \big(f^{(12)}_\tau\big)^2 + \big(f^{(23)}_\tau\big)^2 + \big(f^{(31)}_\tau\big)^2 \Big), & \text{ if $d=2$}
	  \end{cases}
\]
where $f^{(ij)}_\tau = \frac 1 2 (f^{(i)}_\tau + f^{(j)}_\tau)$ is the value of $\tilde f$ at the center of the edge linking vertices $i$ and $j$ in cell $\tau$. We can now define the matrix $D_0(\x)\in\R^{P\times P}$ as the (symmetric) matrix of the following quadratic form
\begin{equation}\label{eq.l2p2}
	\f \mapsto	\f^T D_0(\x) \f \doteq \|\f\|^2_{L^2(\x)}.
\end{equation}
Formula \eqref{eq.l2p2} may be used as an alternative to equation \eqref{eq.lump} to compute $L^2$ norm. We emphasis that matrix $D_0$ in equation \eqref{eq.l2p2} is sparse but no longer diagonal and that the computation is exact on finite elements of order 1. 

For the computation of the $H^1$ norm of $f$, note that the gradient of $\tilde f$ is defined almost everywhere on $\mathcal T$ and is constant on the interior of each cell. We thus introduce the function $g =  \sum_{\tau\in\mathcal T} c_\tau \psi^{(0)}$ with $c_\tau = \| \nabla \tilde f_\tau \|_{\R^d}$ and we use the simple integration formula exact on finite elements of order 0 to get 
\[
	\|\nabla \tilde f\|^2_{L^2(\mathcal T)} = \|g \|_{L^2(\x)}^2  \doteq \sum_{\tau\in\mathcal T} r_\tau  c_\tau^2.
\]
Finally, $D_1(\x)\in\R^{P\times P}$ is the symmetric matrix of the quadratic form defined by
\[
\f \mapsto	\f^T D_1(\x) \f \doteq  \|\f\|^2_{L^2(\x)}\ + \|\nabla \f\|^2_{L^2(\x)}.
\]

\subsection{Deformation on discrete fshapes}

\subsubsection{Discrete Hamiltonian equations}

We can now derive a discrete fshape metamorphosis model along the same lines as the continuous one of previous sections. If we fix $M$ as the template polyhedral manifold $X_0$ itself and consider signals that are obtained with a given finite element interpolation of the values at the vertices, then the state variables in this discrete setting are the two vectors $\x$ and $\f$ and a metamorphosis is determined by a couple $(v_t,\h_t)$ with $v \in L^2([0,1],V)$ and $\h_t=(h_{k,t}) \in \R^{P \times 1}$ such that we have the finite-dimensional evolution equations:
\begin{equation}
 \label{eq:discrete_evol}
   \left\{
  \begin{array}[h]{rcl}
    \dot{x}_{k,t} & = & v_t(x_{k,t})\\
    \dot{f}_{k,t} & = & h_{k,t}
  \end{array}\right.,
\end{equation}
The energy \eqref{eq:energy_def} becomes:
\begin{equation}
\label{eq:discrete_energy_def}
 E_{\x}(v,\h) = \frac{\gamma_{V}}{2} \int_{0}^{1} \|v_t\|_{V}^{2} dt + \frac{\gamma_{f}}{2} \int_{0}^{1} \h_t^{T} D_s(\x_t) \h_t dt 
\end{equation}
The Hamiltonian corresponding to the minimization problem with this discrete energy also takes the form:
\begin{align}
\label{eq:discrete_Hamiltonian}
 H(\x_t,\f_t,\p_t,\p^{f}_t,v,\h_t) &\doteq (\p_t|v_t(\x)) + (\p^{f}_t|\h_t) - \frac{\gamma_V}{2} \|v_t\|_V^2 - \frac{\gamma_f}{2} \h^{T}_t D_s(\x_t) \h_t \nonumber \\
 &= \langle \sum_{\ell=1}^{P} p_{\ell,t}^{T} K_V(x_{\ell,t},\cdot),v_t\rangle_V + \h_t^{T} \p_{t}^{f} - \frac{\gamma_V}{2} \|v_t\|_V^2 - \frac{\gamma_f}{2} \h_t^{T} D_s(\x_t) \h_t
\end{align}
where $\p \in \R^{P \times n}$ and $\p^f \in \R^{P \times 1}$ are the discrete co-state variables. Denoting $K_V$ the vector kernel associated to the RKHS $V$, the optimality conditions along geodesics $\partial_{v} H(\x_t,\f_t,\p_t,\p^{f}_t,v_t,\h_t) = 0$ and $\partial_{\h} H(\x_t,\f_t,\p_t,\p^{f}_t,v_t,\h_t) = 0$ from the PMP lead to the following expressions of the optimal controls:
\begin{equation*}
	\begin{cases}
		\displaystyle v_t  =  \frac{1}{\gamma_V} \sum_{\ell=1}^{P} K_V(x_{\ell,t},\cdot) p_{\ell,t} \\
		\displaystyle \h_{t}  =  \frac{1}{\gamma_f} D_s^{-1}(\x_t)  \p^f_t
	\end{cases}
\end{equation*}
As usual for the LDDMM model, optimal velocity fields $v_t$ are entirely parametrized by the finite dimensional momenta vectors $p_{k,t}$ attached to each vertex position. It results in the following discrete reduced Hamiltonian:
\begin{equation}
\label{eq:discrete_reduced_Hamiltonian}
H_r(\x_t,\f_t,\p_t,\p^{f}_t) = \frac{1}{2\gamma_V}  \p_t^T K_{\x_t,\x_t} \p_t + \frac{1}{2\gamma_f}  (\p^{f}_t)^T D_s^{-1}(\x_t)  \p^f_t
\end{equation}
where $\p_t^T K_{\x_t,\x_t} \p_t \doteq \sum_{k,\ell=1}^P p_{k,t}^T K_V(x_{k,t},x_{\ell,t}) p_{\ell,t}$.

\subsubsection{Forward equations}\label{seq.forward}

From equation \eqref{eq:reduced_Hamiltonian_evol_fshape}, we obtain the discrete equivalent of the Hamiltonian evolution equations:
\begin{equation}
 \label{eq:discrete_reduced_Hamiltonian_evol_fshape}
\begin{pmatrix}
	\dot \x_t \\ \dot \f_t\\ \dot \p_t \\ \dot \p^f_t
\end{pmatrix} = 
\begin{pmatrix}
	\phantom{-}\partial_{\p} H_r(\x,\f,\p,\p^f) \\
	\phantom{-}\partial_{\p^f} H_r(\x,\f,\p,\p^f) \\
	-\partial_{\x} H_r(\x,\f,\p,\p^f) \\
	-\partial_{\f} H_r(\x,\f,\p,\p^f)
\end{pmatrix} = 
	 F(\x,\f,\p,\p^f).
\end{equation}
It may be written in an explicit way by using formula \eqref{eq:discrete_reduced_Hamiltonian} and we have
\begin{equation}
	F(\x,\f,\p,\p^f) =
 \begin{pmatrix}
	\frac{1}{\gamma_V} K_{\x_t,\x_t} \p_{t} \\
	\frac{1}{\gamma_f} D^{-1}_s(\x_t) \p^f \\
	-\frac{1}{2\gamma_V}  \p_{t}^T \partial_{\x_t} K_{\x_{t},\x_{t}} \p_{t} + \frac{1}{2\gamma_f} (D^{-1}_s(\x_t)\p^{f})^T \partial_{\x_{t}} D_s(\x_t)  (D_s^{-1}(\x_t)\p^f)\\
	\displaystyle  0
 \end{pmatrix}.
\end{equation}
Some remarks can be made about the system of forward equations \eqref{eq:discrete_reduced_Hamiltonian_evol_fshape}. First, we recover the fact that the momentum $\p^f_t$ is constant over the time (see Theorem \ref{theo_Hamiltonian_eq_metam}) and for that reason we have dropped the subscript $t$ in writing $\p^f \doteq \p^f_t$. We also point out that formula \eqref{eq:discrete_reduced_Hamiltonian_evol_fshape} contains new  terms (\ie compared to the 'tangential' algorithm of \cite{Charlier15}) related to the evolution of the signal. In particular, $\dot \p_t$ now depends on the functional momentum $\p^f$ meaning that a variation in the signal induces a variation in the geometry (see Section \ref{part:evol_sphere_fshape} for an illustration). Finally, these new terms involve in particular the inverse of the sparse matrix $D_s(\x_t) \in \R^{P\times P}$ used in the computation of the functional norms (see equation \eqref{eq.l2approx}). Each time step thus requires solving the sparse (but still large) linear system $D_s(\x_t) h = \p^f$ which may be numerically costly. We use MATLAB linear sparse solver to perform that operation. Yet, this can result in a typically 5 to 10 times slower algorithm compared to the 'tangential' one for fshapes having in the range of ten thousand vertices.

\subsubsection{Geodesic shooting algorithm}
\label{ssec:algo_matching}

Along the same lines, data attachment term $g(\x_1,\f_1)$ and its derivatives with respect to $\x_1$ and $\f_1$ are discretized from the continuous version of equation \eqref{eq:fvarifold}: we refer to \cite{Charlier15} for the detailed expressions. 

The discrete equivalent of fshape registration equation \eqref{eq:registration_fshape} can be then cast as a finite dimensional optimization problem on the initial momenta variables $\p_0\doteq \rst{\p_t}{t=0} \in \R^{P \times 3}$ and $\p^f \in \R^{P \times 1}$ that writes:
\begin{equation}
 \label{eq:registration_fshape_discrete}
 \min_{\p_0,\p^f} J(\p_0,\p^f) \doteq \tfrac{1}{2\gamma_V} \p_0^T K_{\x,\x} \p_0 + \tfrac{1}{2\gamma_f} (\p^{f})^T D_s^{-1}(\x_0) \p^f + \gamma_W g(\x_1,\f_1)
\end{equation}
subject to the dynamics described by equation \eqref{eq:discrete_reduced_Hamiltonian_evol_fshape}. Due to the intricate dependency of final states $\x_1$ and $\f_1$ in the variables $\p_0$ and $\p^f$ as well as the possible non-convexity of $g$, this is typically a non-convex problem and thus, at best, we aim at finding a (not necessarily unique) minimum. The formulation of equation \eqref{eq:registration_fshape_discrete} suggests a \textbf{geodesic shooting} scheme for solving the minimization generalizing widely used similar frameworks in diffeomorphic shape matching, as the ones presented for example in \cite{Allassonniere2005,Vialard2012b}. 

In the case of our problem, this amounts essentially in a gradient descent on the initial momenta variables $(\p_0,\p^f)$. The gradients of the two first terms in equation \eqref{eq:registration_fshape_discrete} are easily computed, only the last term $g(\x_1,\f_1)$ that involves final states is slightly more involved. It may be tackled by integrating backward the so called adjoint linearized system of equations: 
\begin{align}
\label{eq:discrete_adjoint_equation}
\begin{pmatrix}
\dot X_t \\ \dot F_t\\ \dot P_t \\ \dot P^f_t	\end{pmatrix} &= (-dF(\x_t,\f_t,\p_t,\p^f_t))^T\begin{pmatrix}
X_t \\ F_t\\ P_t \\ P^f_t	\end{pmatrix}
\end{align}
with the adjoint variables $X_t\in\R^{P\times n}$, $F_t\in \R^{P\times 1}$, $P_t\in\R^{P\times n}$, $P^f_t\in \R^{P\times 1}$ and the endpoint conditions $X_1=\partial_{\x} g(\x_1,\f_1)$, $F_1 = \partial_{\f} g(\x_1,\f_1)$, $P_1=\partial_{\p} g(\x_1,\f_1)=0$ and $P^f_1=\partial_{\p^f} g(\x_1,\f_1)=0$. In practice, the system of equations \eqref{eq:discrete_adjoint_equation} is tedious to implement and we use instead the finite difference trick presented in \cite{arguillere14:_shape} (Section 4.1 just before Proposition 9). To integrate the adjoint system \eqref{eq:discrete_adjoint_equation}, rather than explicitly compute each term in the matrix $dF^T$, we only need to compute a single directional derivative at each time step with a finite difference method. This has has several advantages: it is rather general, it greatly simplifies the implementation and in the end amounts in about twice the computational cost of the forward system of equations \eqref{eq:discrete_reduced_Hamiltonian_evol_fshape}.

\medskip

In summary, the gradient of the objective functional with respect to $\p_0$ and $\p^f$ is obtained by the following forward-backward scheme:
\begin{itemize}
 \item[$(1)$] Compute $(\x_t,\f_t,\p_t,\p_t^f)$ by integrating equation \eqref{eq:discrete_reduced_Hamiltonian_evol_fshape} forward with initial conditions $(\x_0,\f_0,\p_0,\p_0^f)$.
 
 \item[$(2)$] Compute the gradients of $g(\x_1,\f_1)$ with respect to $\f$ and $\x$. 
 
 \item[$(3)$] Transport the gradients to $t=0$ by integrating backward equation \eqref{eq:discrete_adjoint_equation} with final conditions $X_1=\partial_{\x} g(\x_1,\f_1)$, $F_1 = \partial_{\f} g(\x_1,\f_1)$, $P_1=0$, $P^f_1=0$.
 
 \item[$(4)$] Set $\nabla_{\p_0} J = \frac{1}{\gamma_V} K_{\x,\x} \p_0 + P_0$ and $\nabla_{\p^f} J = D_0(\x_0)(\frac{1}{\gamma_f}  D_s^{-1}(\x_0) \p^f + \gamma_W P^f_0)$
 
\end{itemize}

We point out that the gradient with respect to the functional momentum $\p^f$ at step (4) is computed with respect to the $L^2$ metric on $X_0$ instead of the Euclidean metric, which adds the extra weight matrix $D_0(\x_0)$. This can be crucial for example when the mesh $X_0$ is not regular but contains triangles of very different areas. The updates on $\p^f$ obtained from the gradient computed with respect to this metric ensures that the signal variations $\dot \f = D^{-1}_s (\x) \p^f$ will not be too much affected by the quality of the initial mesh.

The rest of the fshape matching algorithm consists in an adaptive step gradient descent simultaneously on $\p_0$ and $\p^f$. The architecture of the code is in MATLAB with time-consuming segments (computation of kernel sums for the most part) externalized in CUDA. The whole code is available within the \texttt{FshapesTk} software \cite{fstk}.

\section{Results and discussion}
In this section, we show a few results of the fshape matching algorithm presented in section \ref{ssec:algo_matching}. We will first focus on some simple examples to illustrate certain aspects of the method in particular the influence of the norm regularity. Following these, we evaluate qualitatively the output of the algorithm on a few examples of functional shapes originating from medical imaging. All experiments were performed on a server machine equipped with a Nvidia GTX 555 graphics card.      

\subsection{Synthetic data}

\paragraph{Digits.} We first evaluate the algorithm on an example mimicking the situation of gray level images as in Section \ref{ssec:image_metamorphosis}. Here, the geometrical part of both the source and target fshapes is the flat square  $[-1,1]\times [-1,1]\times \left\{ 0 \right\} \subset \R^3$. These two distinct triangular meshes were created with a standard Delaunay triangulation method and contain 4900 vertices each as shown in Figure \ref{fig:metam_digits_st}.  The signal part represents two  handwritten digits with value ranging from 0 (red) to 0.6 (blue). 

Figure \ref{fig:metam_digits} shows an example of metamorphoses in $L^2$ with varying penalty coefficients on the functional momentum part of the energy $\gamma_f$ and $\gamma_V$. Results are consistent with the expected behavior: the smaller $\gamma_f$ the more the transformation is performed in the photometric component instead of deforming the image by the diffeomorphism. We chose for the kernel $K_V$ defining $\|\cdot\|_{V}$ a sum of two radial scalar Gaussian \cite{Vialard2012} with (small) widths $0.2$ and $0.1$ (the square having an edge of size 2). The optimization is performed with a coarse to fine strategy (as described in \cite{Charlier15}) and the final kernels $k_p$ and $k_f$ are taken Gaussian as well with respectively $\sigma_p = 0.05$ and $\sigma_f=0.7$.

\begin{figure}
\centering
 \begin{tabular}{cc}
  \includegraphics[width=3cm]{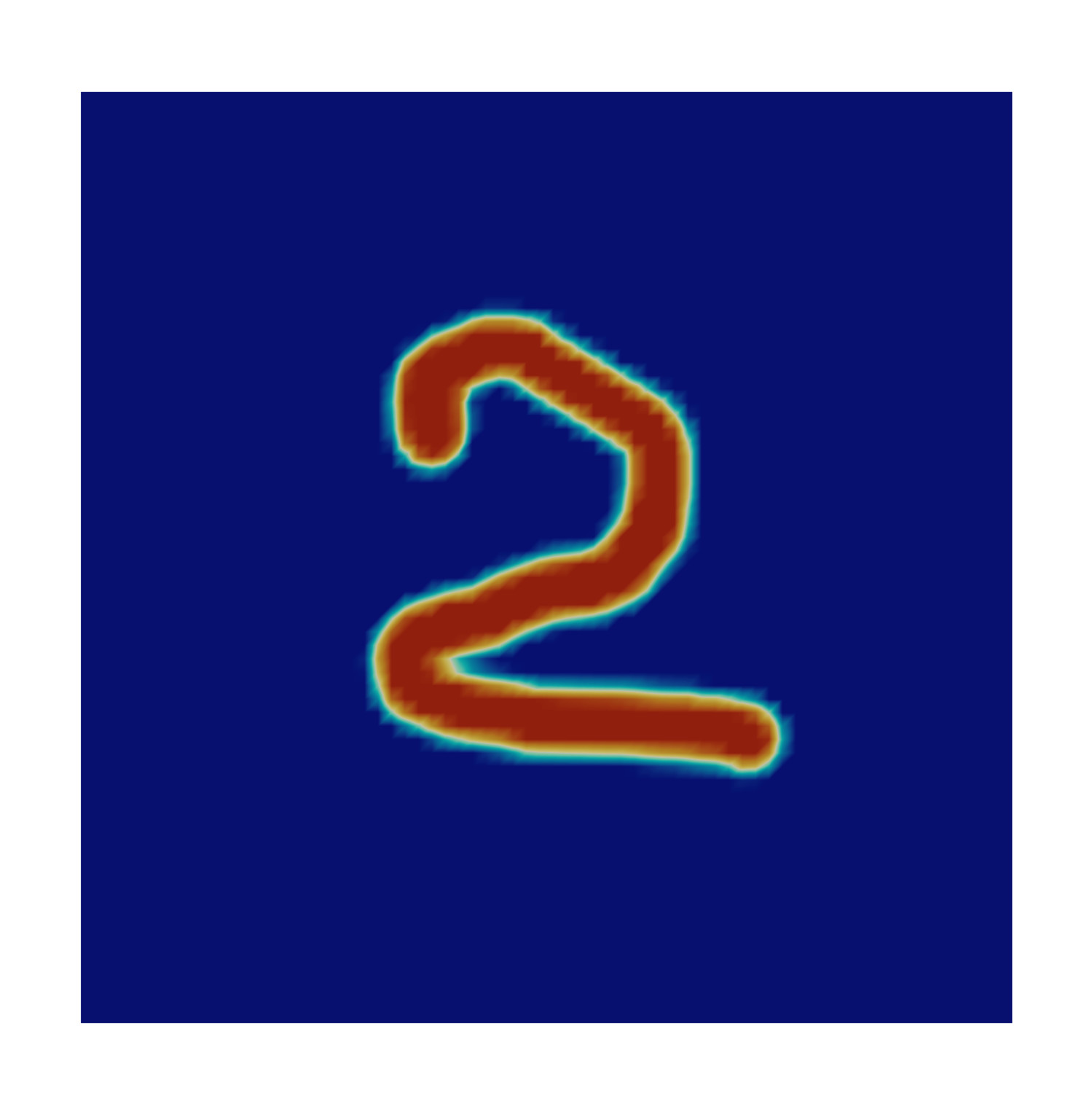}& \includegraphics[width=3cm]{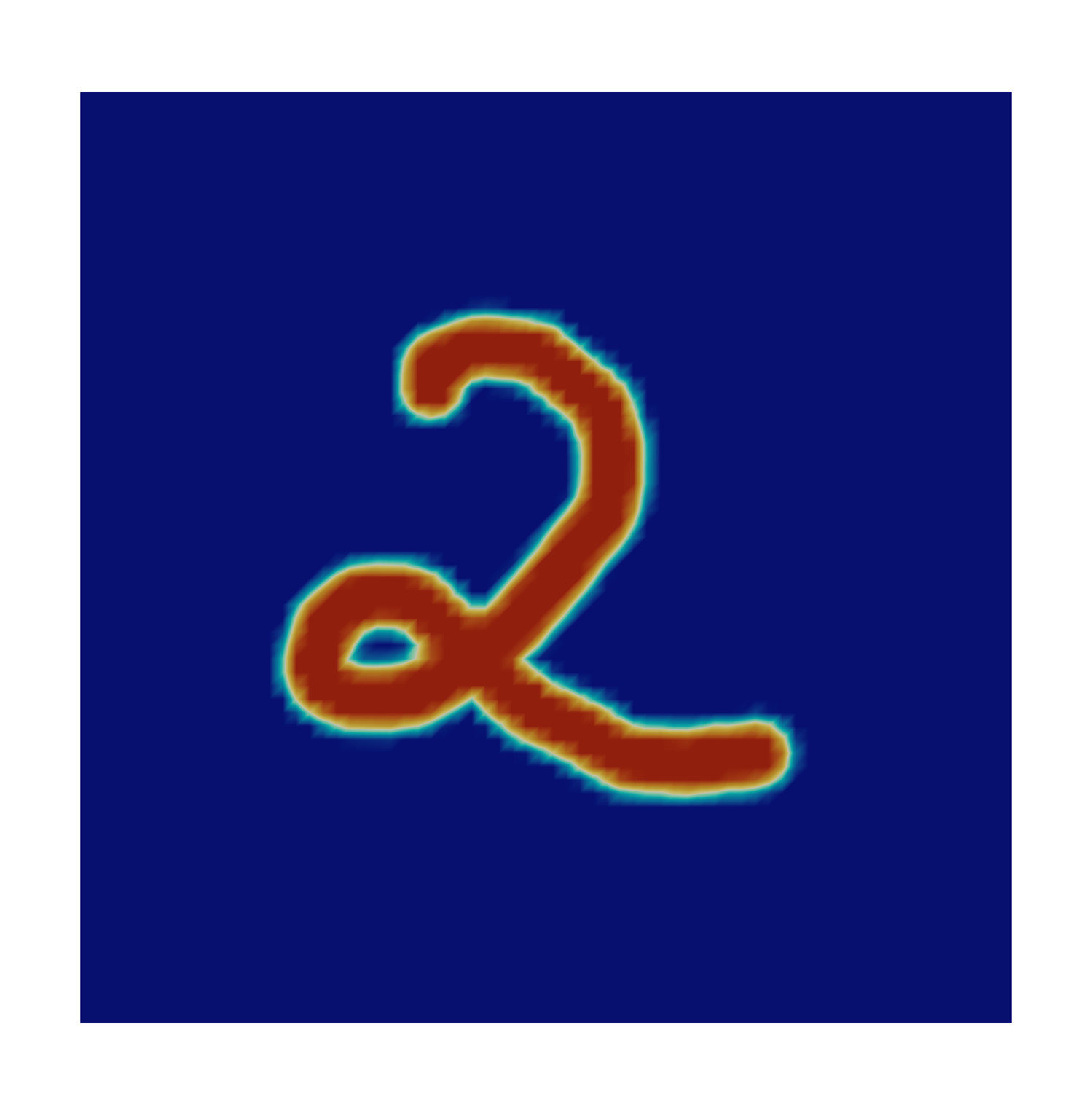} 
 \end{tabular}\\ 
 \caption{Digits: source fshape (left) and target fshape (right)}
\label{fig:metam_digits_st}
\end{figure}

\begin{figure}
\centering
 \begin{tabular}{ccccc}
  \includegraphics[width=3cm]{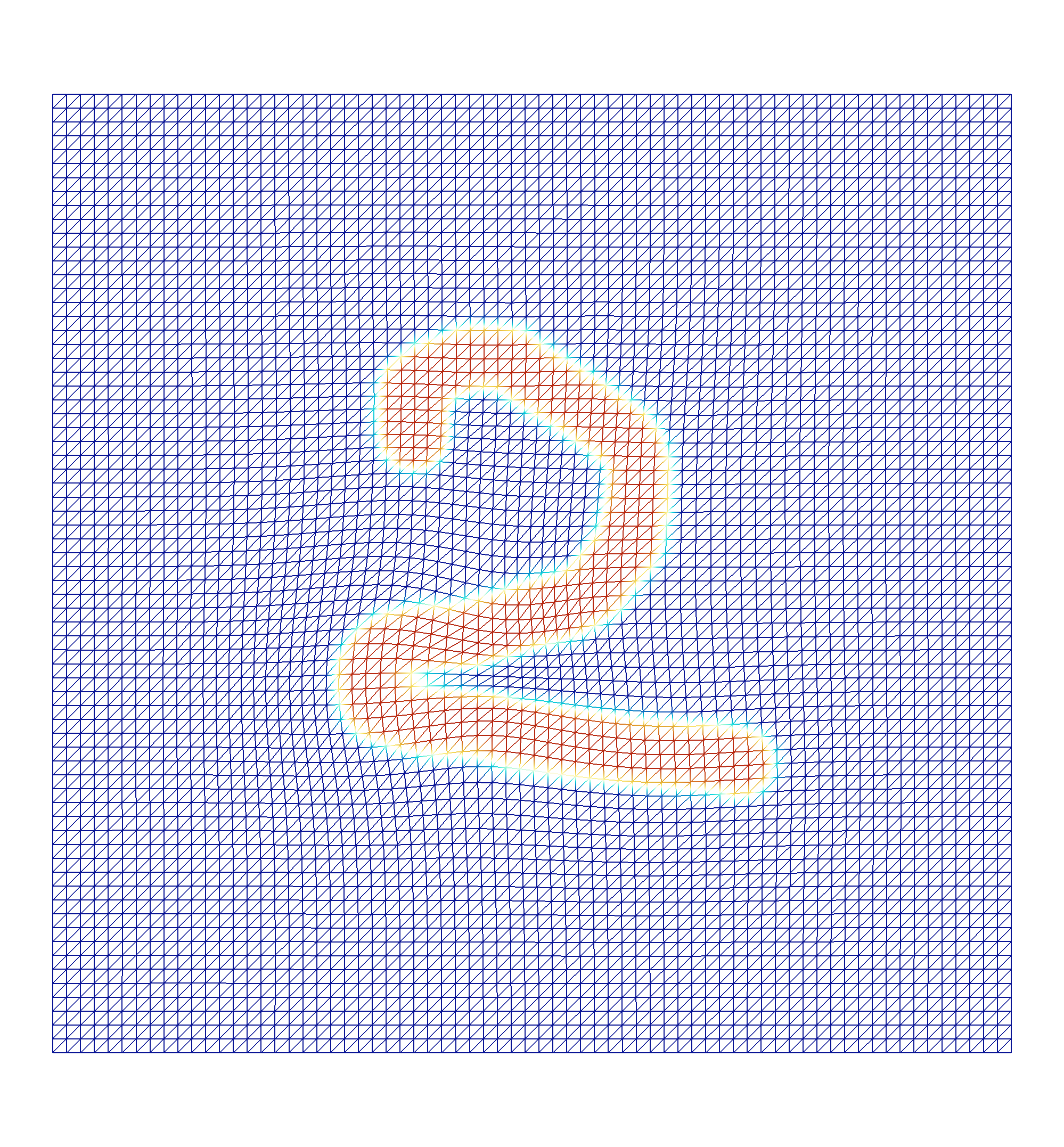} & \includegraphics[width=3cm]{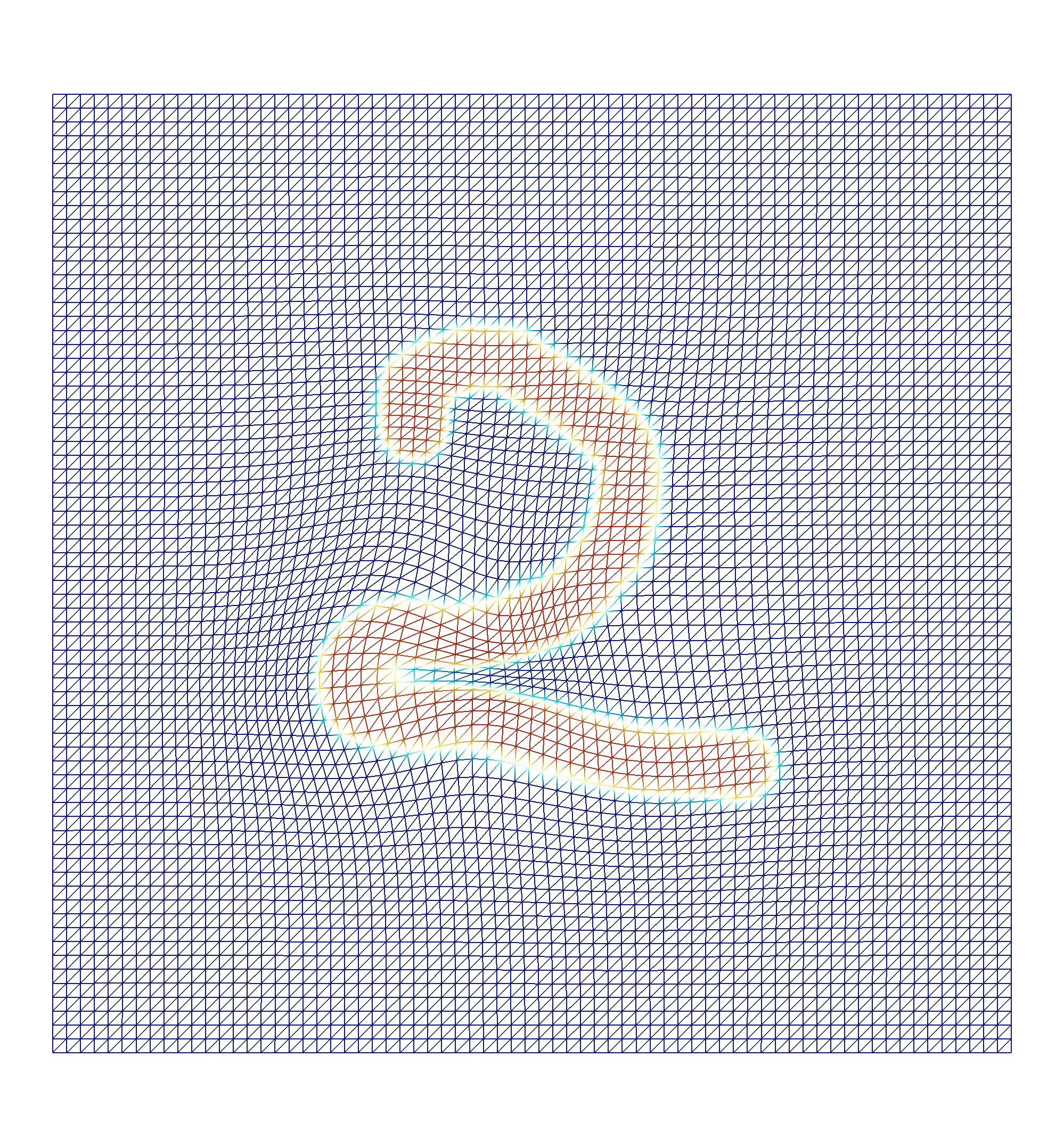} & \includegraphics[width=3cm]{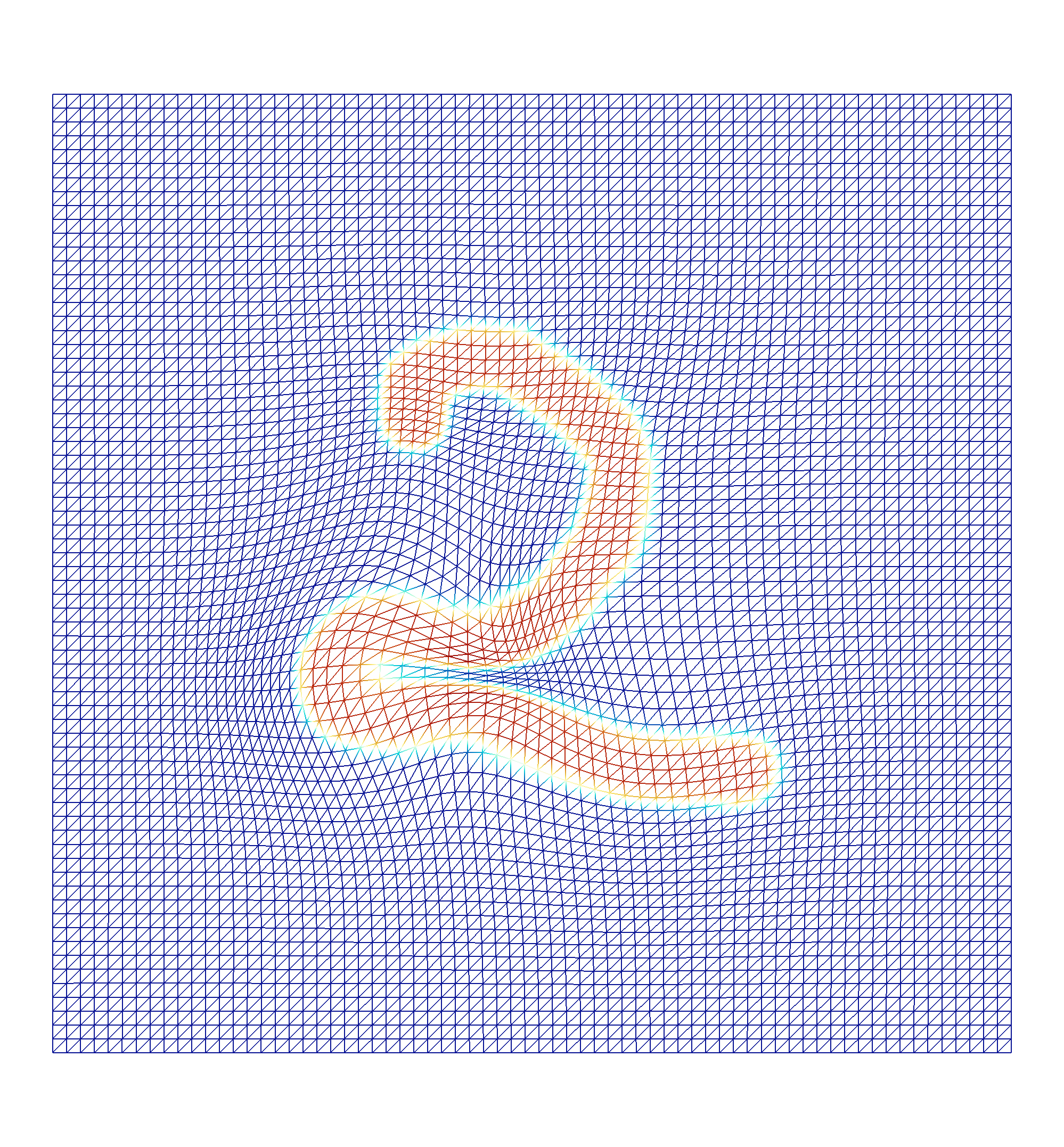} & \includegraphics[width=3cm]{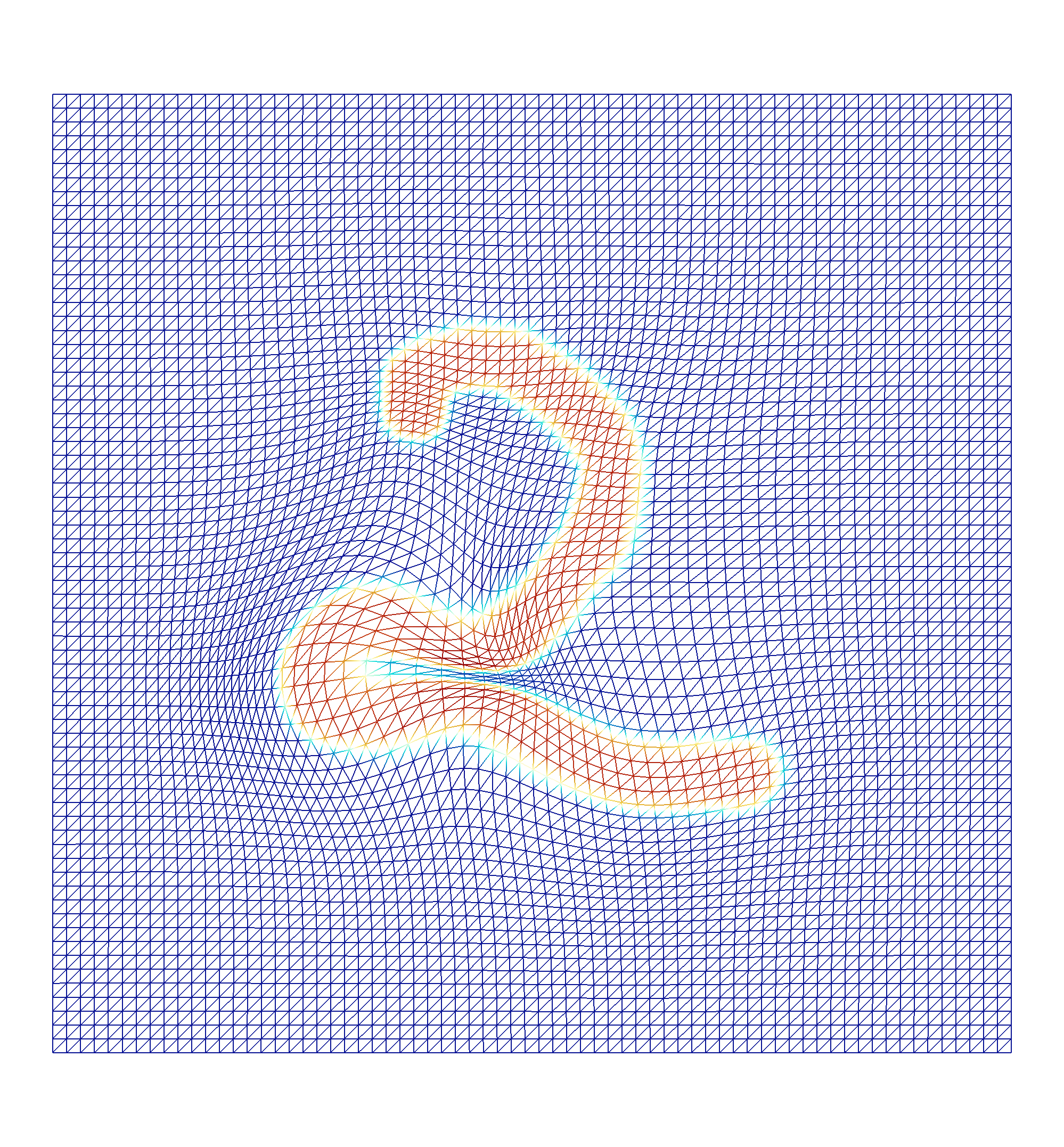}& \includegraphics[width=3cm]{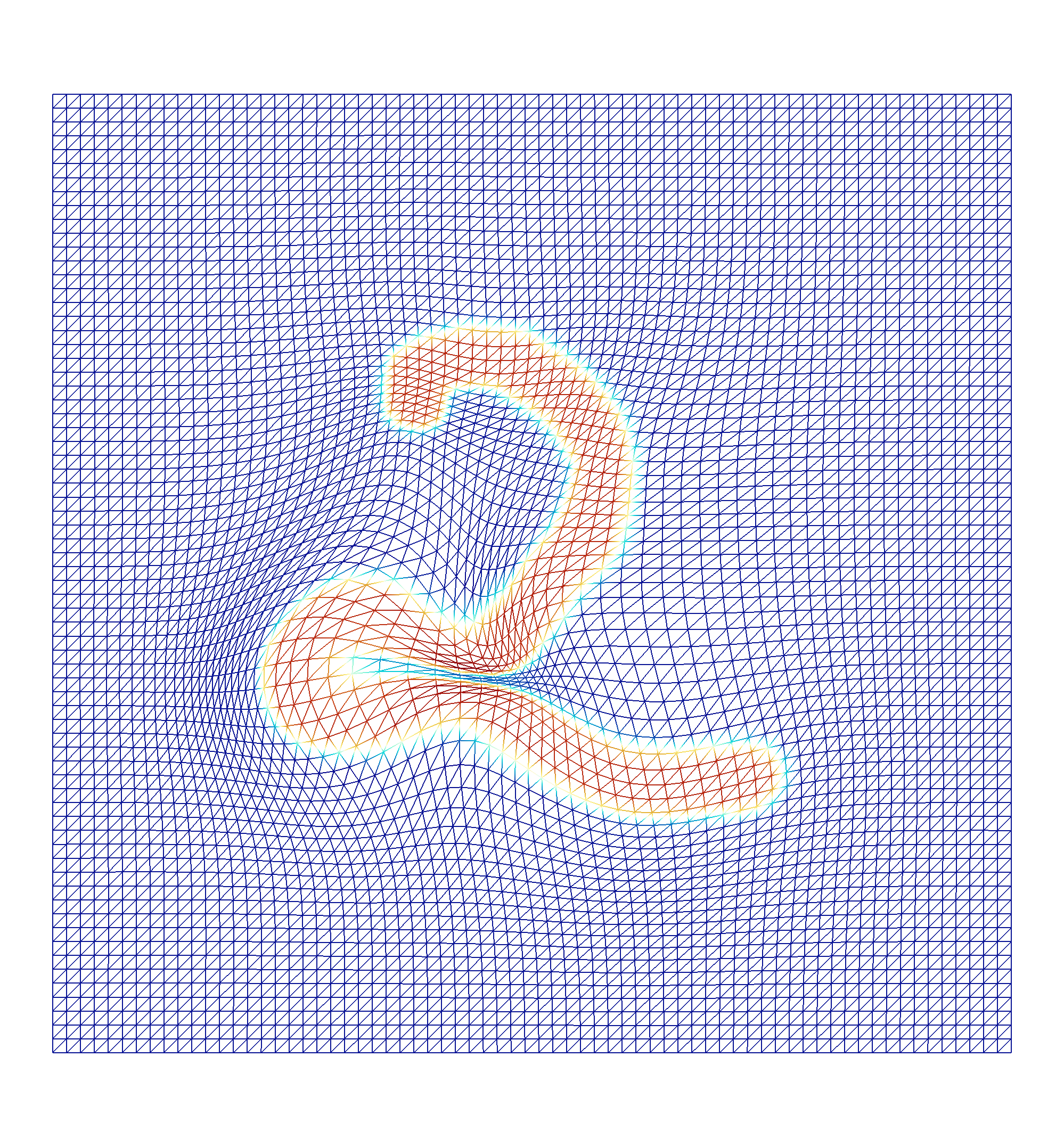}  \\
 \phantom{a} & \phantom{a} & \phantom{a} & \phantom{a} & \phantom{a} \\
 \includegraphics[width=3cm]{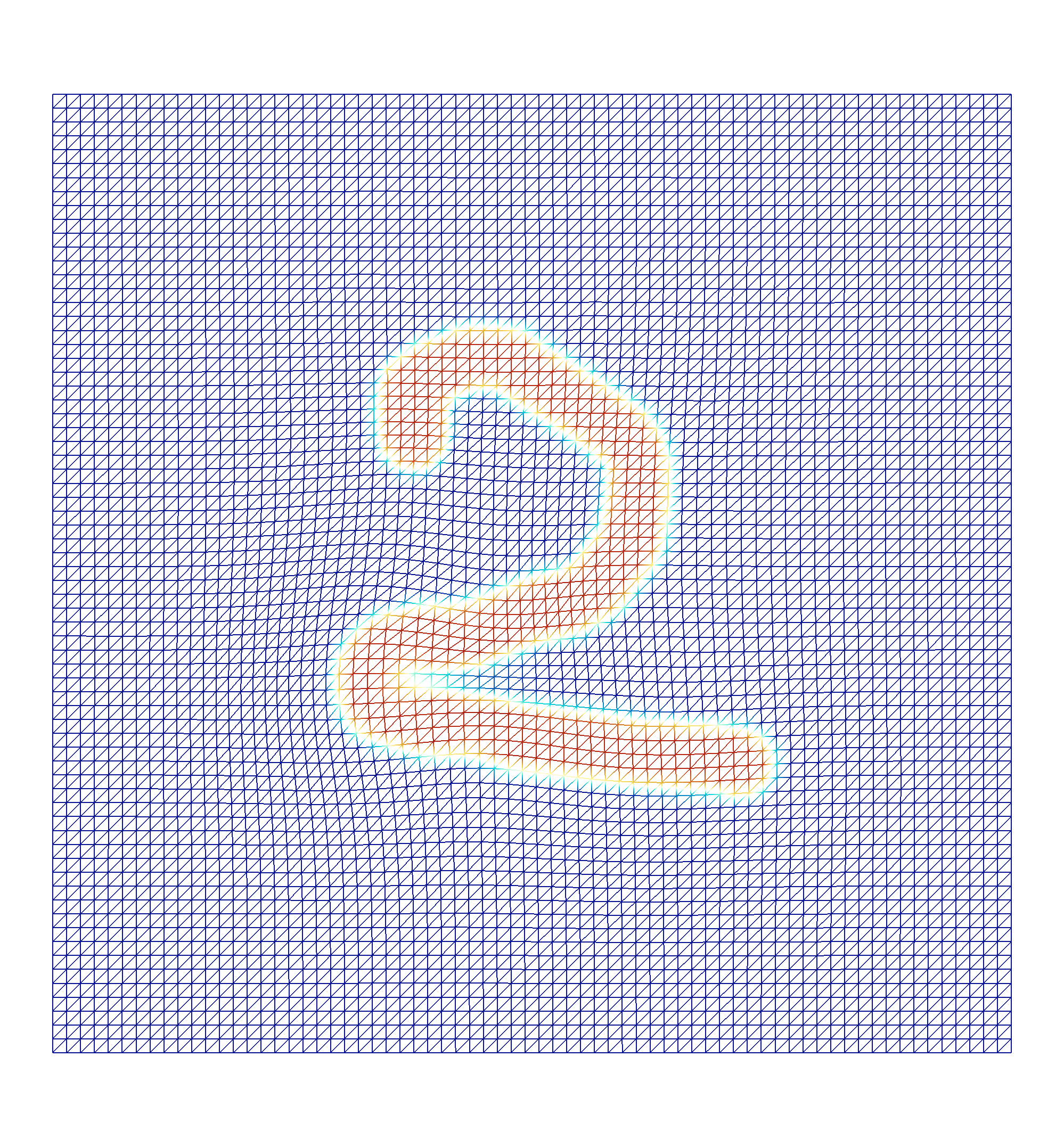} & \includegraphics[width=3cm]{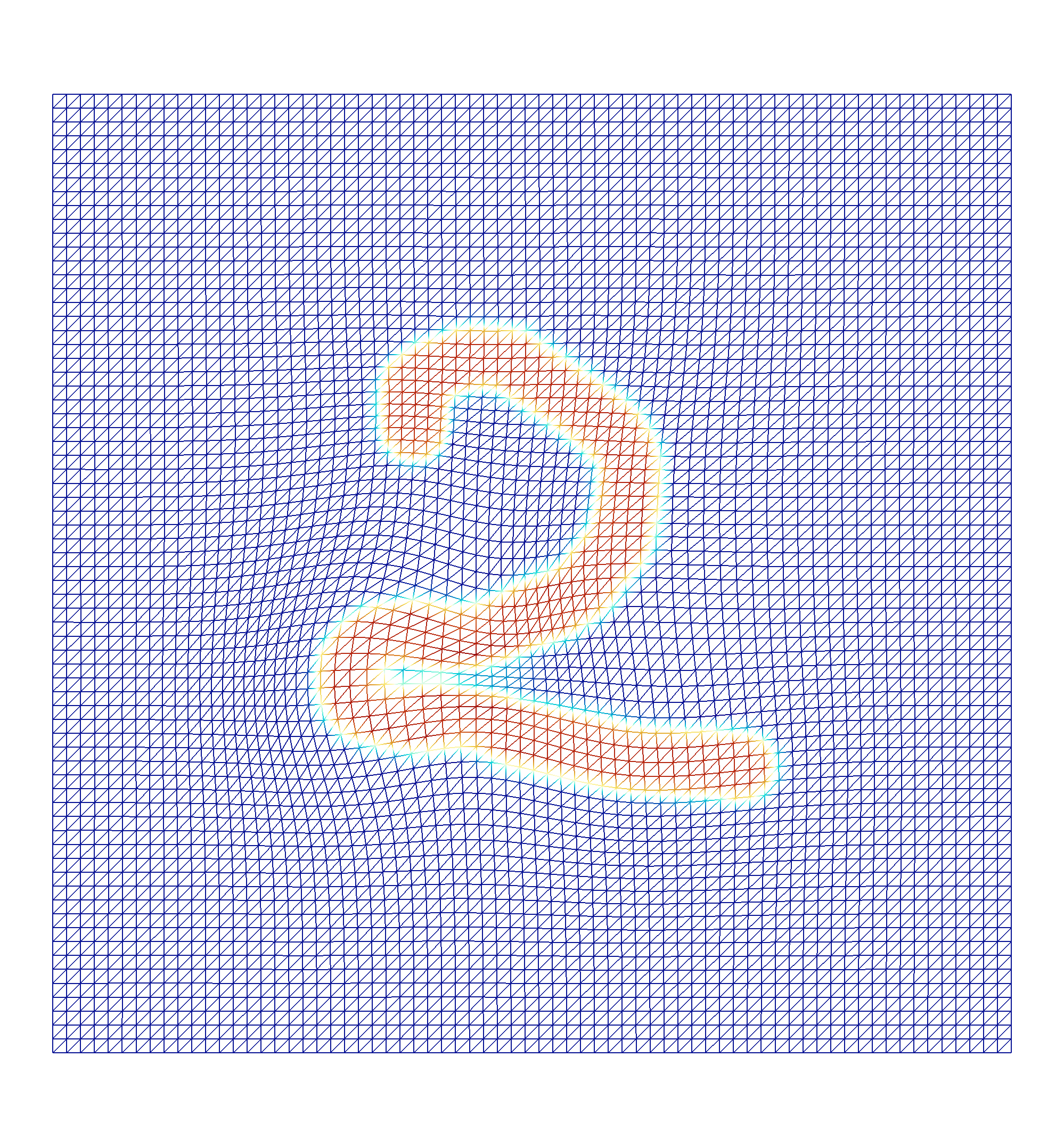} & \includegraphics[width=3cm]{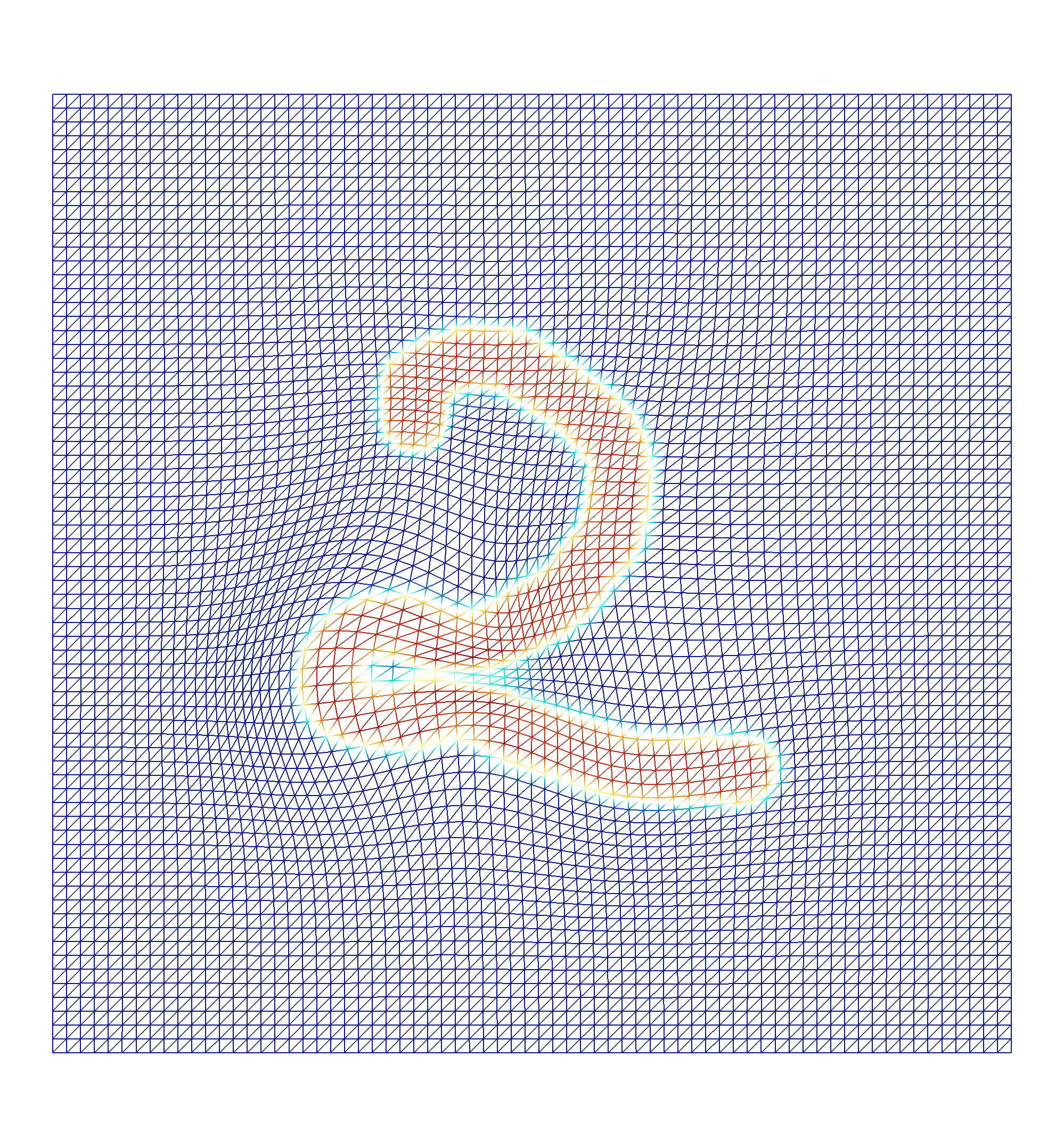} & \includegraphics[width=3cm]{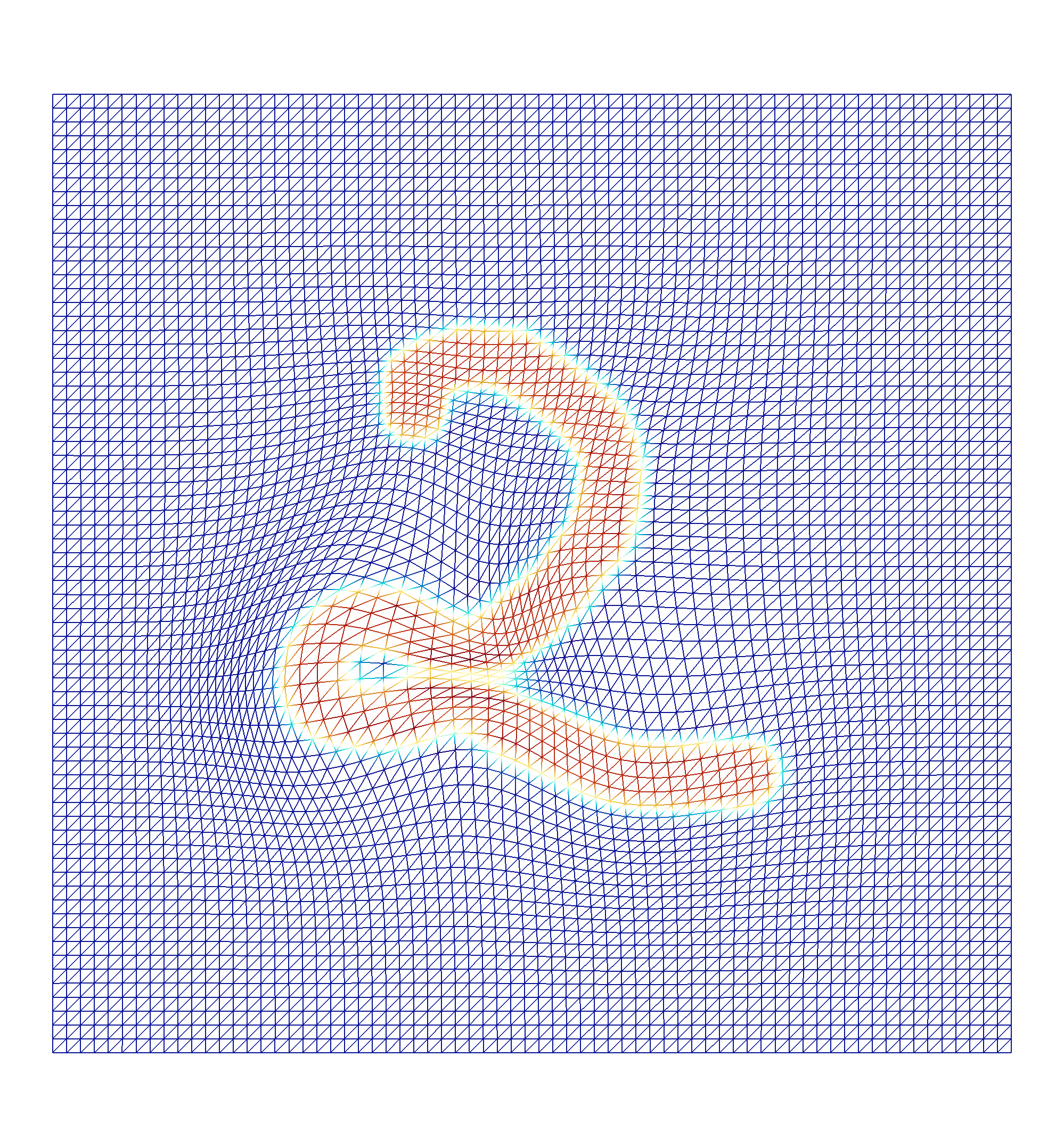}& \includegraphics[width=3cm]{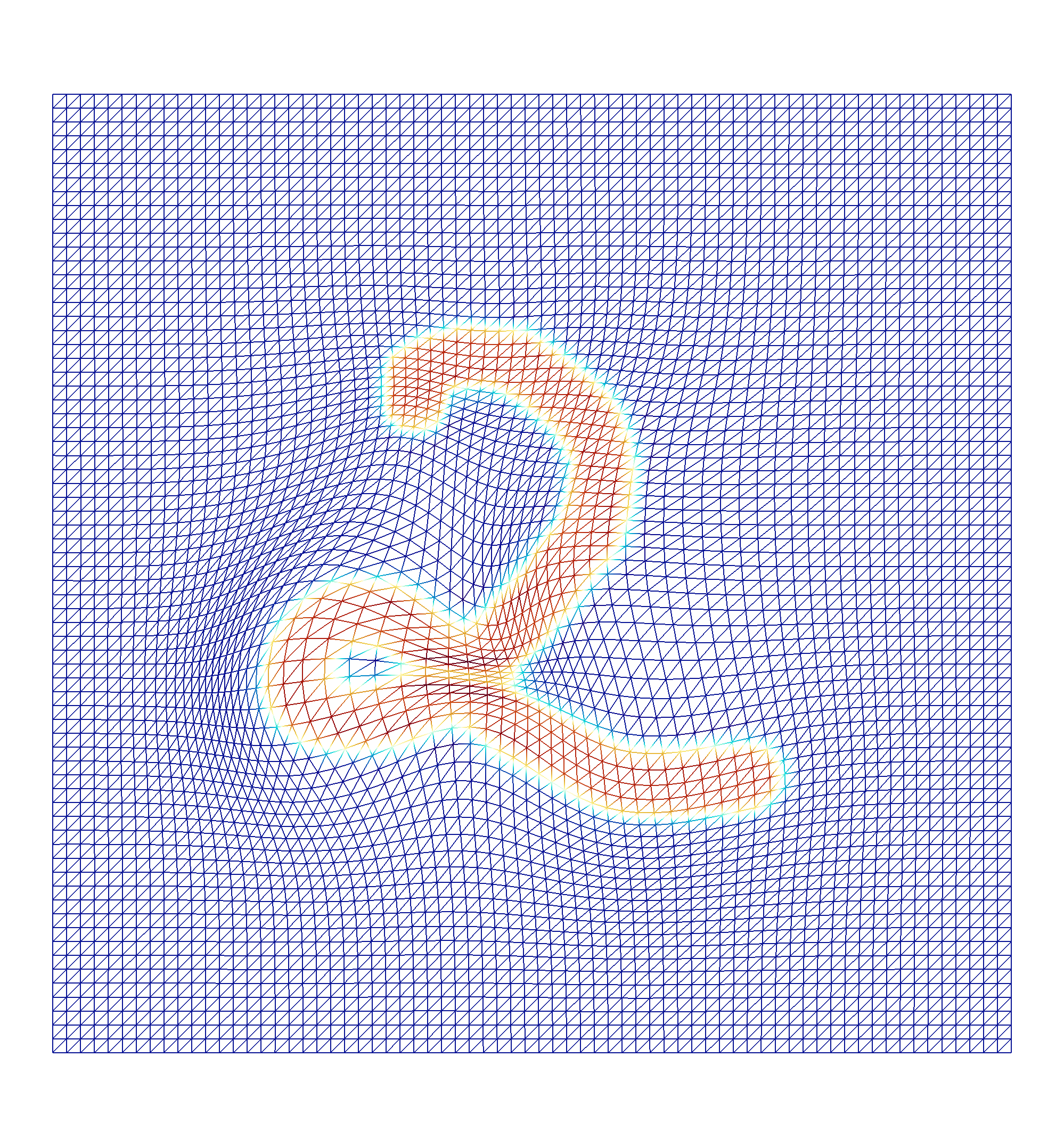}  \\
 \phantom{a} & \phantom{a} & \phantom{a} & \phantom{a} & \phantom{a} \\
\includegraphics[width=3cm]{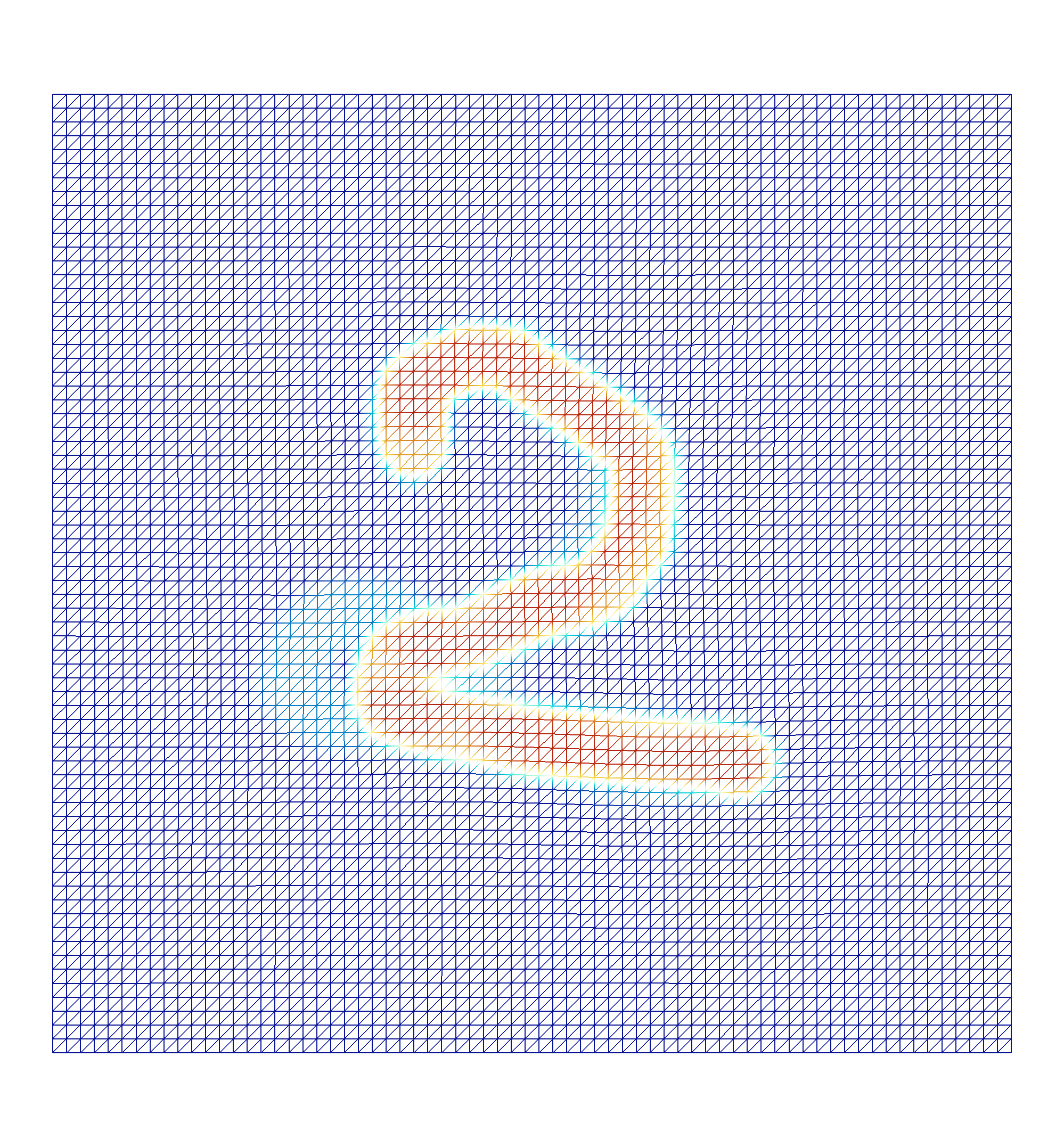} & \includegraphics[width=3cm]{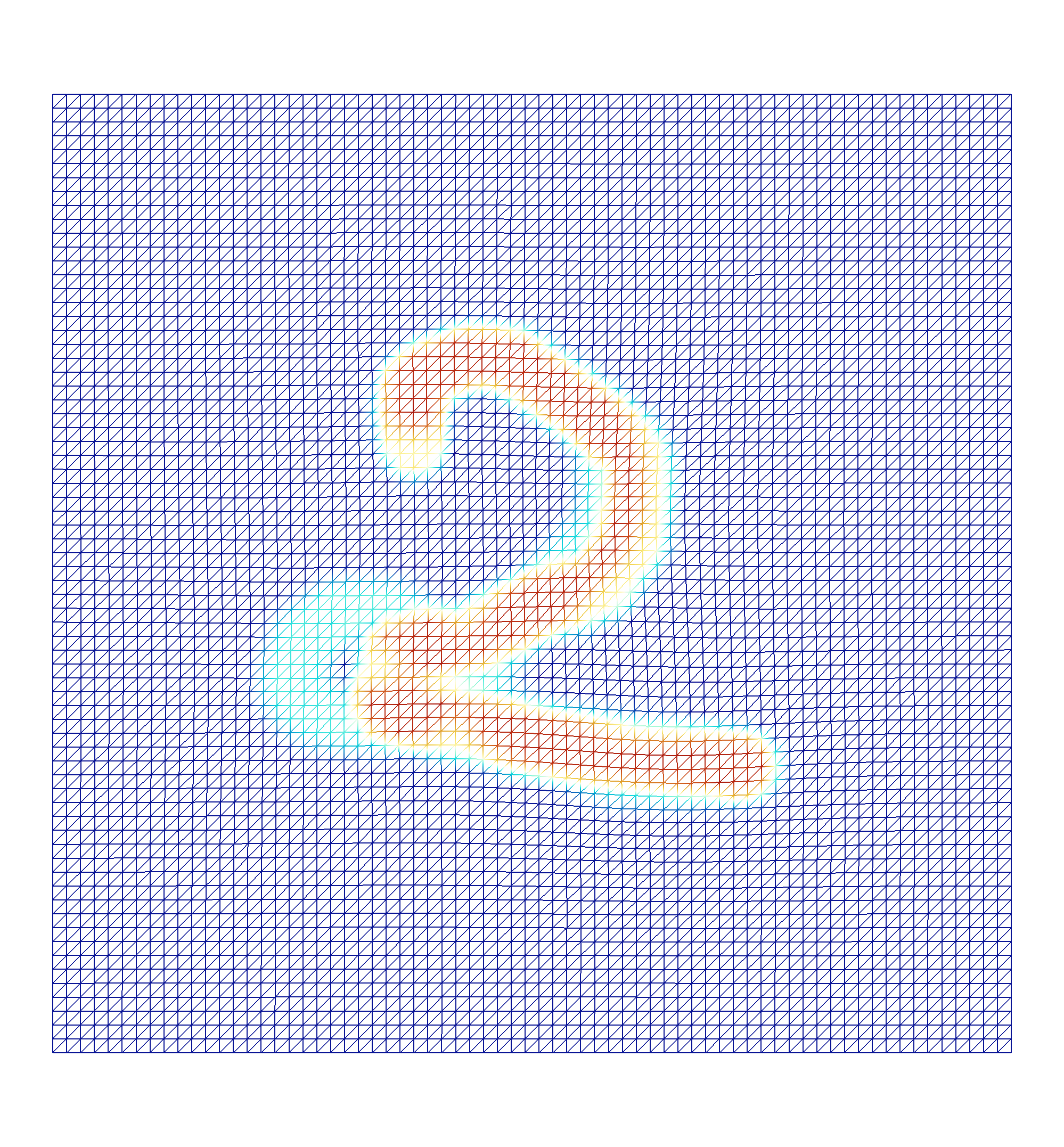} & \includegraphics[width=3cm]{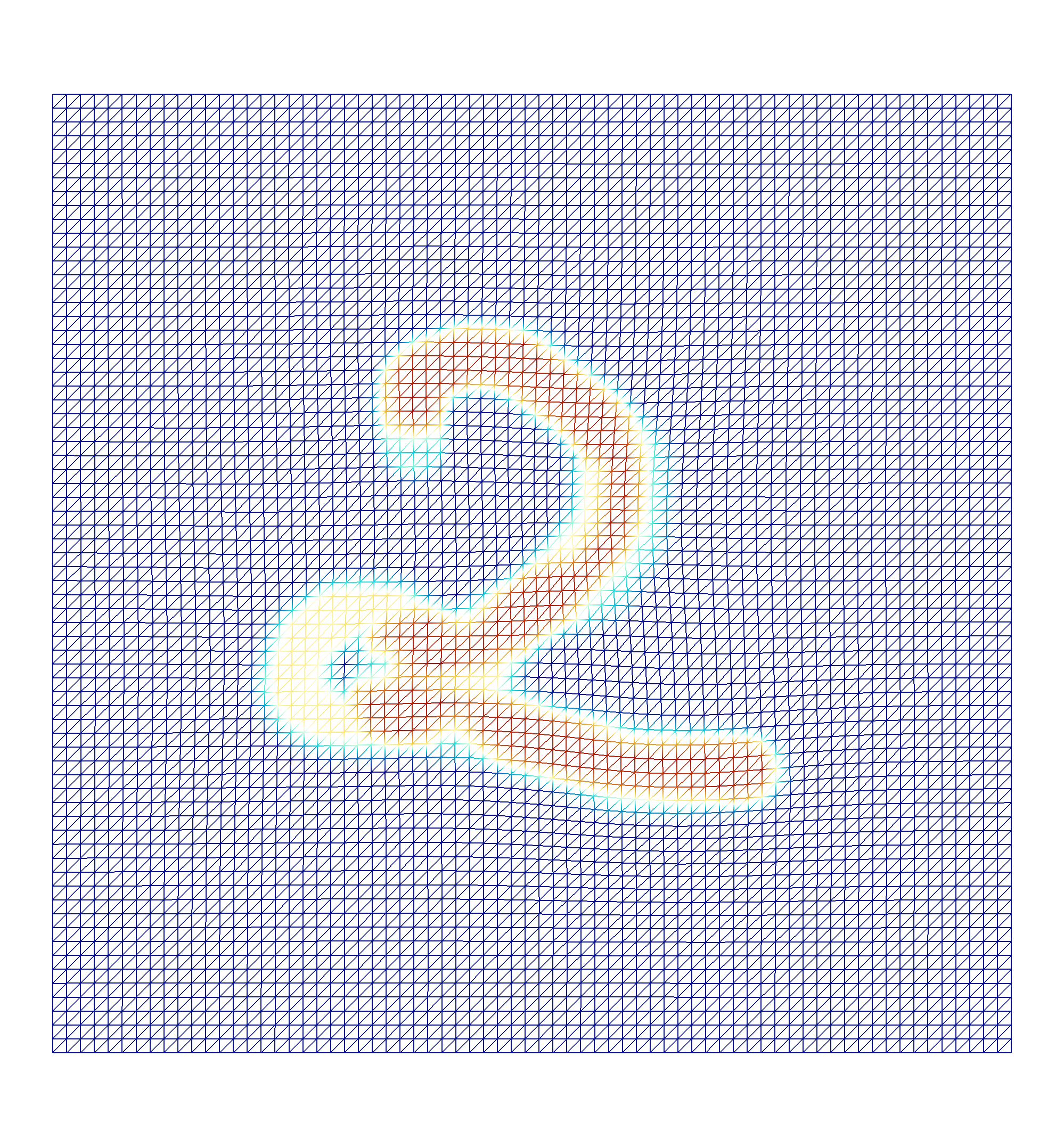} & \includegraphics[width=3cm]{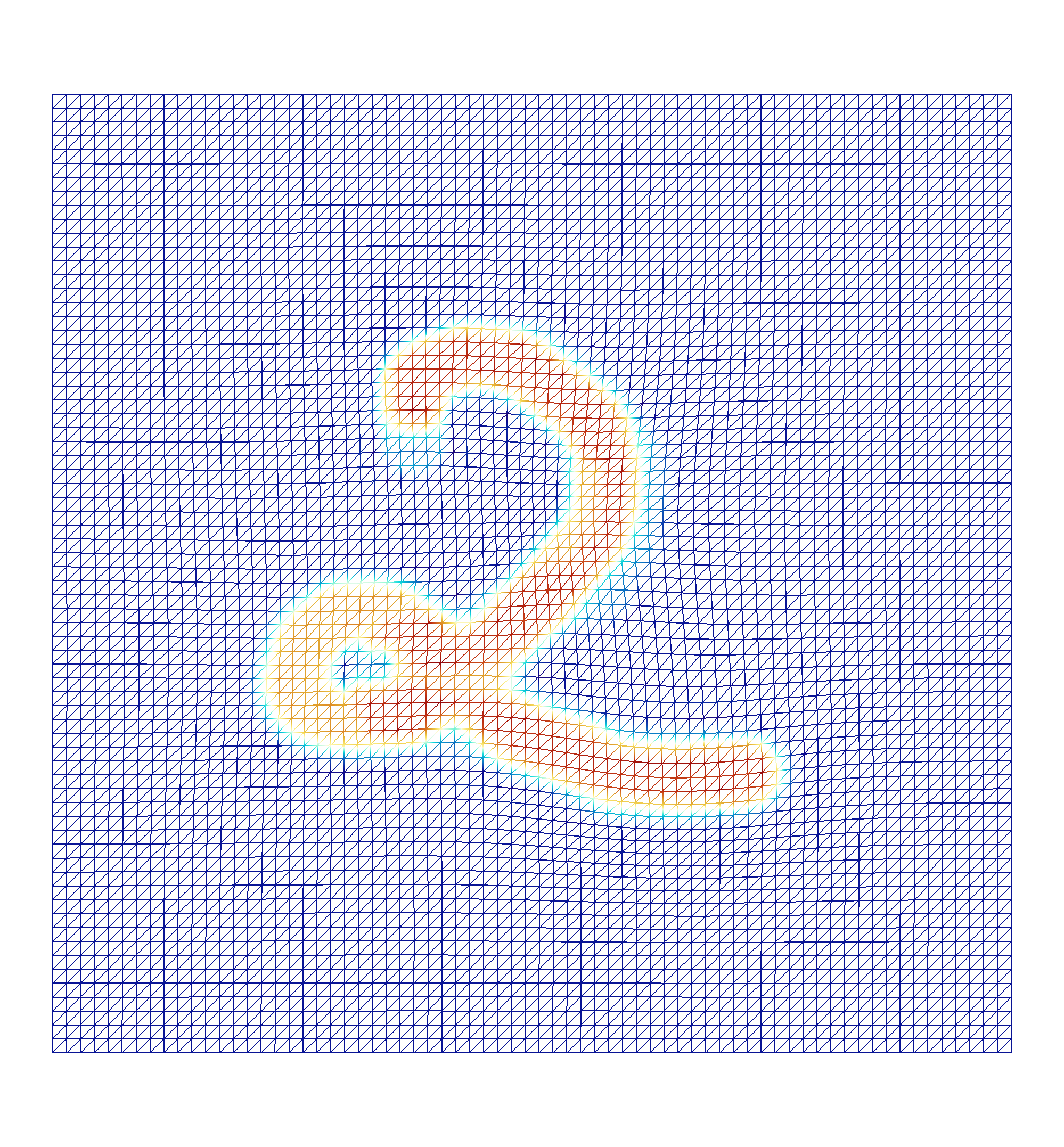}& \includegraphics[width=3cm]{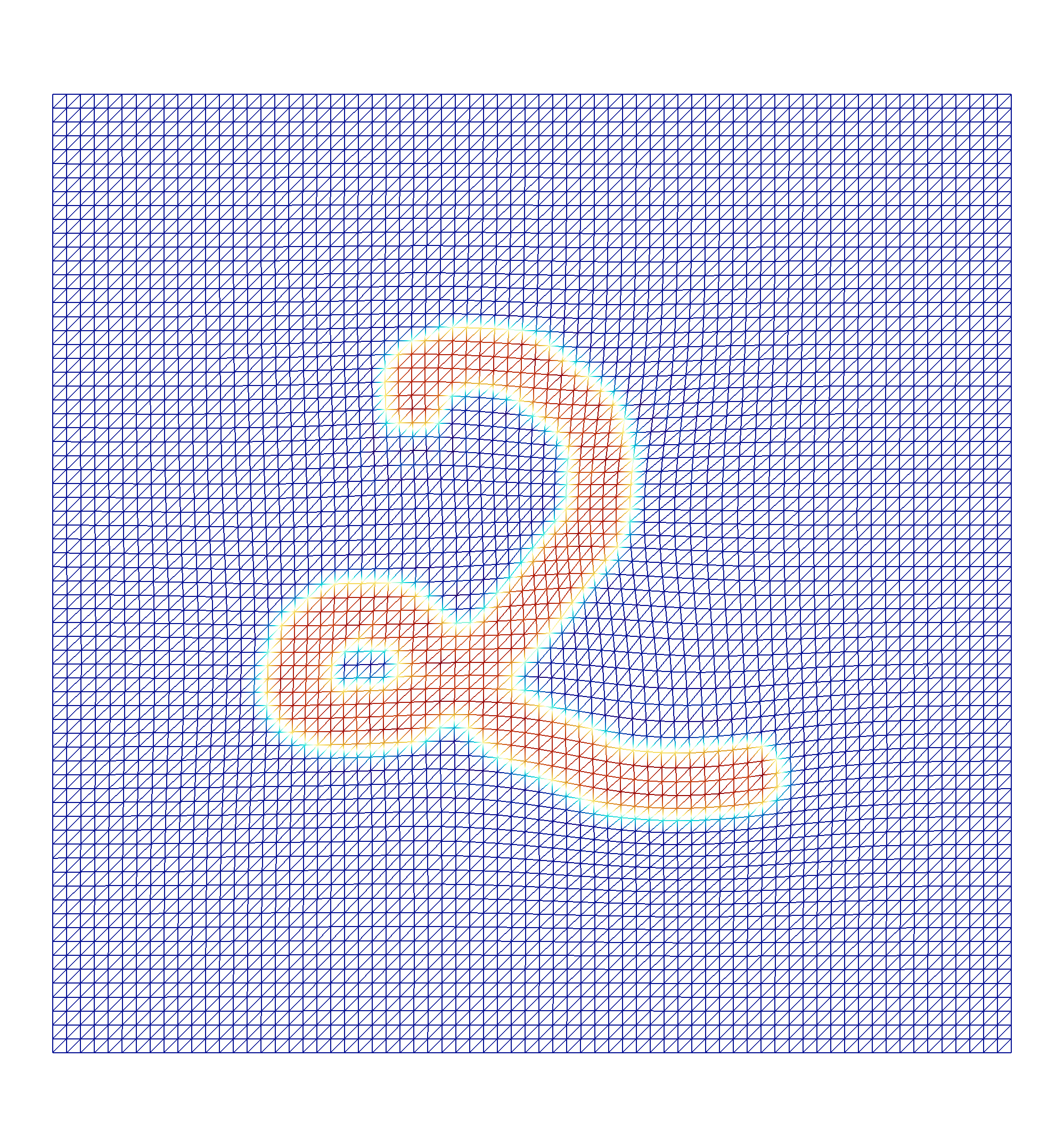}  \\
 \phantom{a} & \phantom{a} & \phantom{a} & \phantom{a} & \phantom{a} \\
  $t=0.2$ & $t=0.4$ & $t=0.6$ & $t=0.8$ & $t=1$
 \end{tabular}
\caption{Estimated metamorphoses of fshapes with a signal representing handwritten digits. The meshes are plotted in wireframe representation to clearly see the deformation. Three experiments for different parameters in the energy are showm: $\gamma_V/\gamma_f = 1$ (first row), $\gamma_V/\gamma_f = 20$ (second row) and $\gamma_V/\gamma_f = 100$ (third row). 
}
\label{fig:metam_digits}
\end{figure}

\paragraph{Stanford bunny.} Secondly, we examine the effect of increasing the metric regularity in the functional dynamics' penalty. The example in Figure \ref{fig:metam_bunny} is a metamorphosis of a sphere (with 10242 vertices) onto the Stanford bunny surface (with 2581 vertices) with a fairly smooth signal function. Results from metamorphosis in $H^1$ display nice regular evolution throughout time and a resulting transformation very consistent with the target despite the difference of sampling between the two meshes. On the other hand, the equivalent result in $L^2$ (with the same parameters) shows some residual oscillatory patterns in the recovered signal unlike the target one, appearing mostly in areas where the transformation is not as close to the target. The qualitative comparison is shown in Figure \ref{fig:metam_bunny_h1_l2} with several views. This effect is particularly obvious on the below part of the mesh where some holes are present in the target. Such oscillations had been noticed already and studied in simpler settings as in \cite{Nardi2015}. They are in a sense numerical manifestations of the conditions on the existence of solutions to the problem with $L^2$ and the absence of weak continuity in $L^2$ of the fvarifold terms. Note that oscillations may be still alleviated if one increases the penalty weight $\gamma_f$; however this would also result in less overall accuracy in the signal matching. Another classical advantage of $H^1$ metamorphosis over $L^2$ is the robustness to signal noise: resulting metamorphoses in $L^2$ are much more affected by the presence of noise or outliers in signal values than higher regularity metrics. In terms of running time however, the $L^2$ metamorphosis scheme with the mass lumping discretization described in \ref{ssec:mass_lumping} only involves inversion of diagonal linear systems in the signal dynamics, resulting in an algorithm running in 45 minutes which is about 6 times faster compared to the finite elements scheme of the $H^1$ case.                   
\begin{figure}
\centering
 \begin{tabular}{cccc}
 \includegraphics[width=3.2cm]{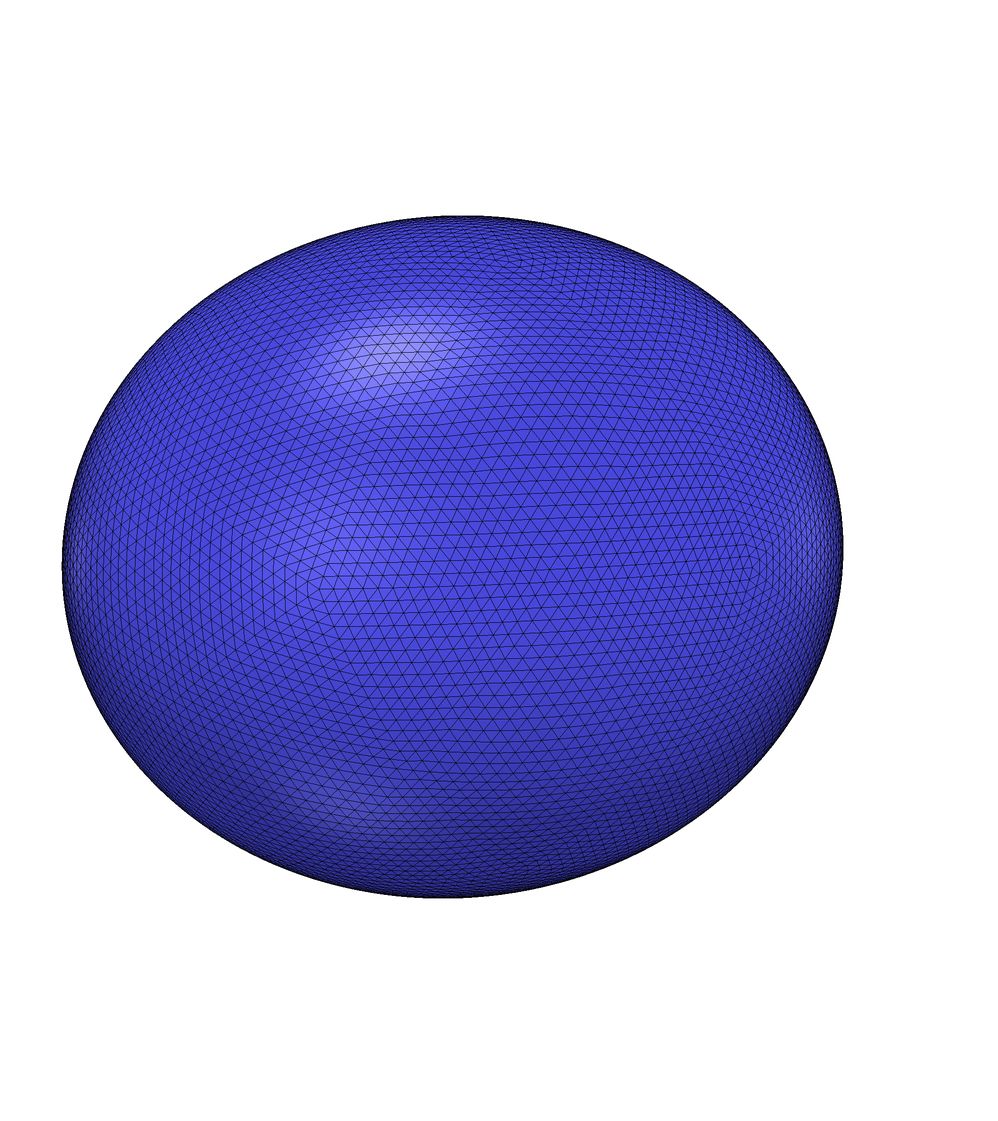} & \includegraphics[width=3.2cm]{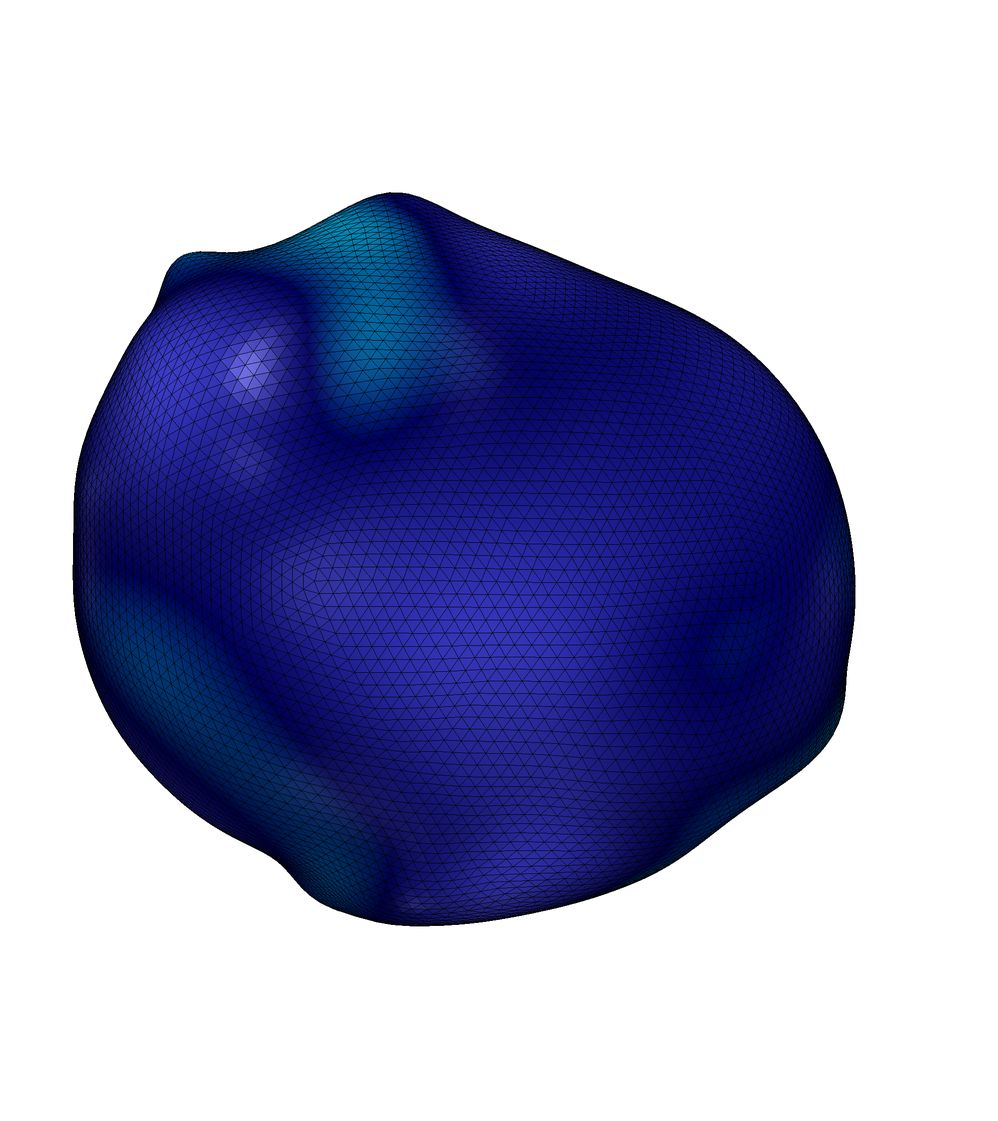} & \includegraphics[width=3.2cm]{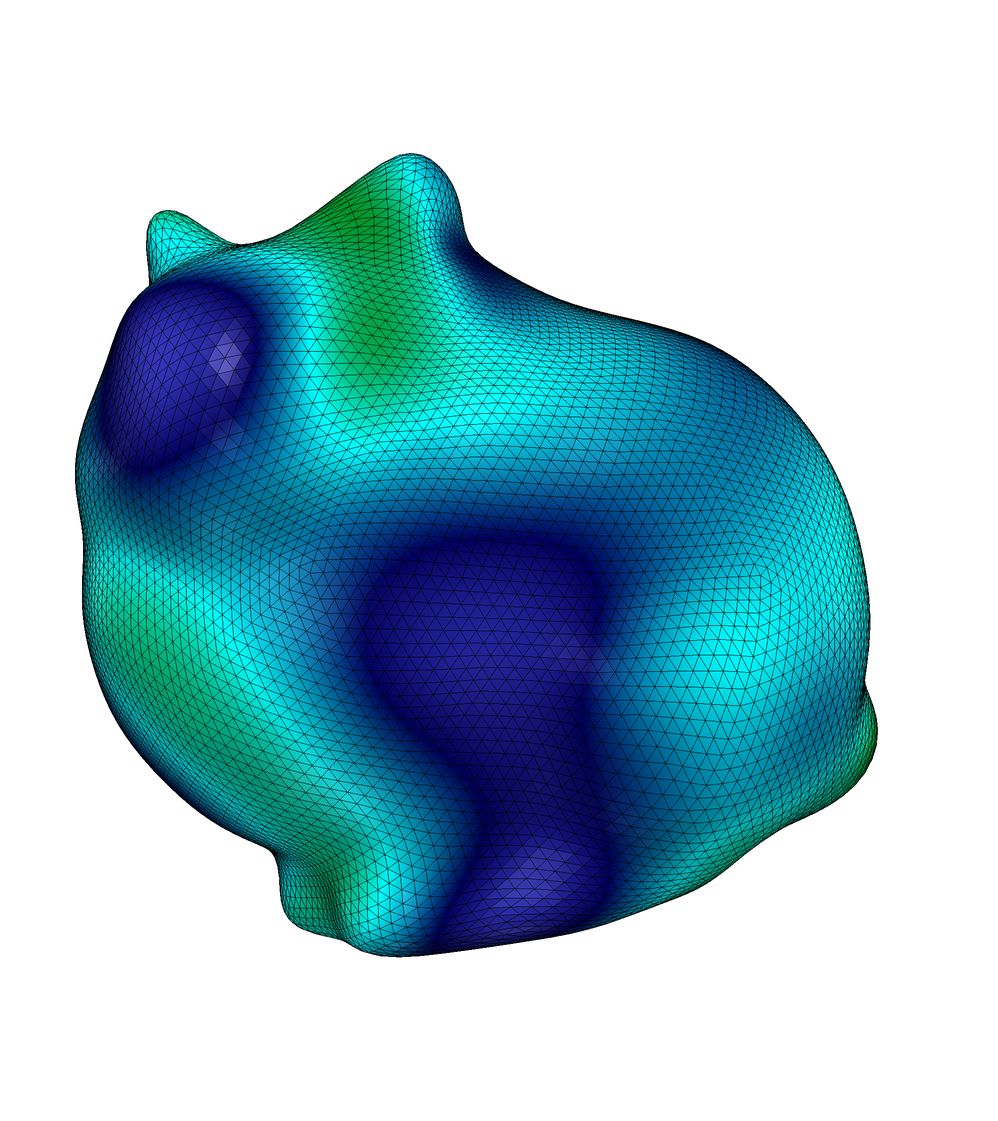} & \includegraphics[width=3.2cm]{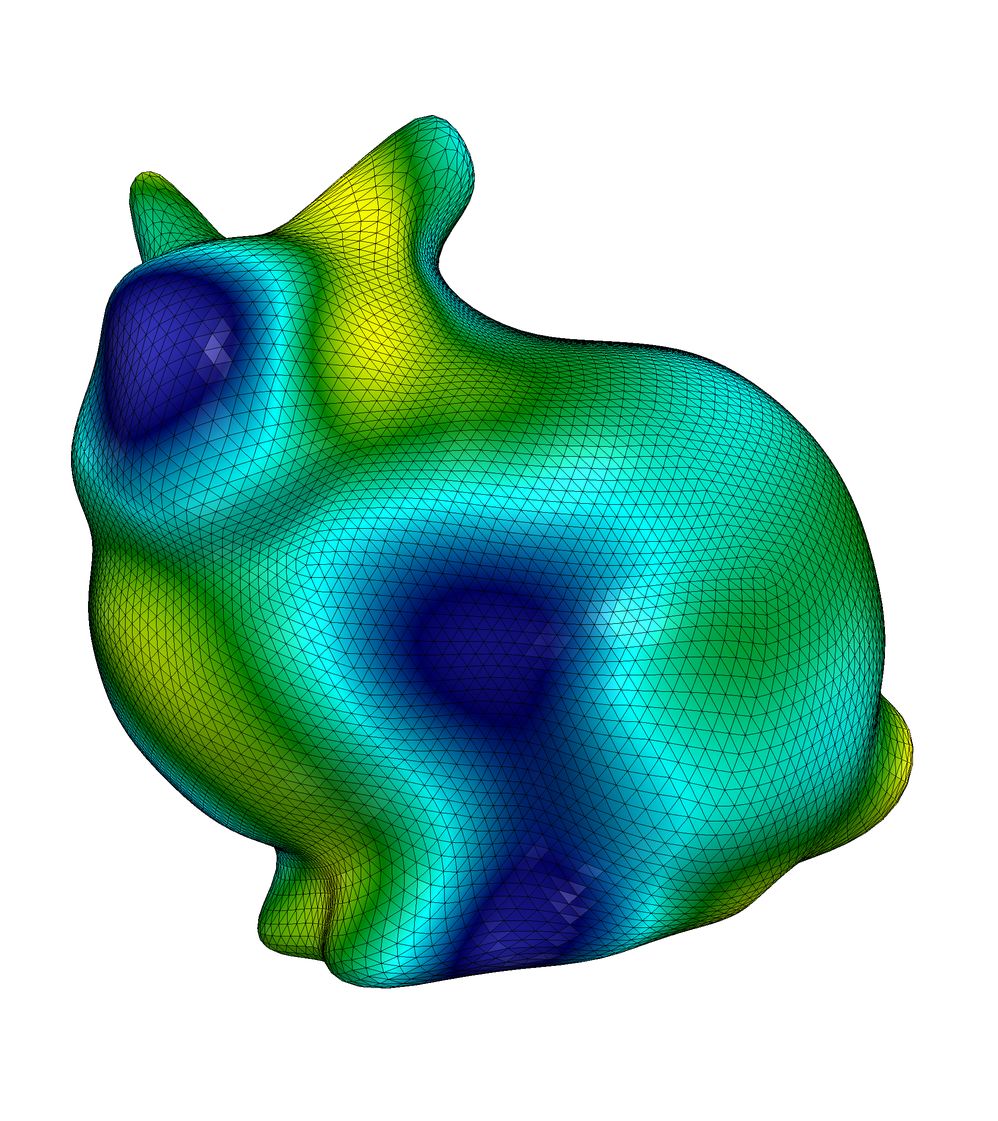} \\
 Source & $t=0.2$ & $t=0.4$ & $t=0.6$ 
 \end{tabular}
 \begin{tabular}{ccc}
 \includegraphics[width=3.2cm]{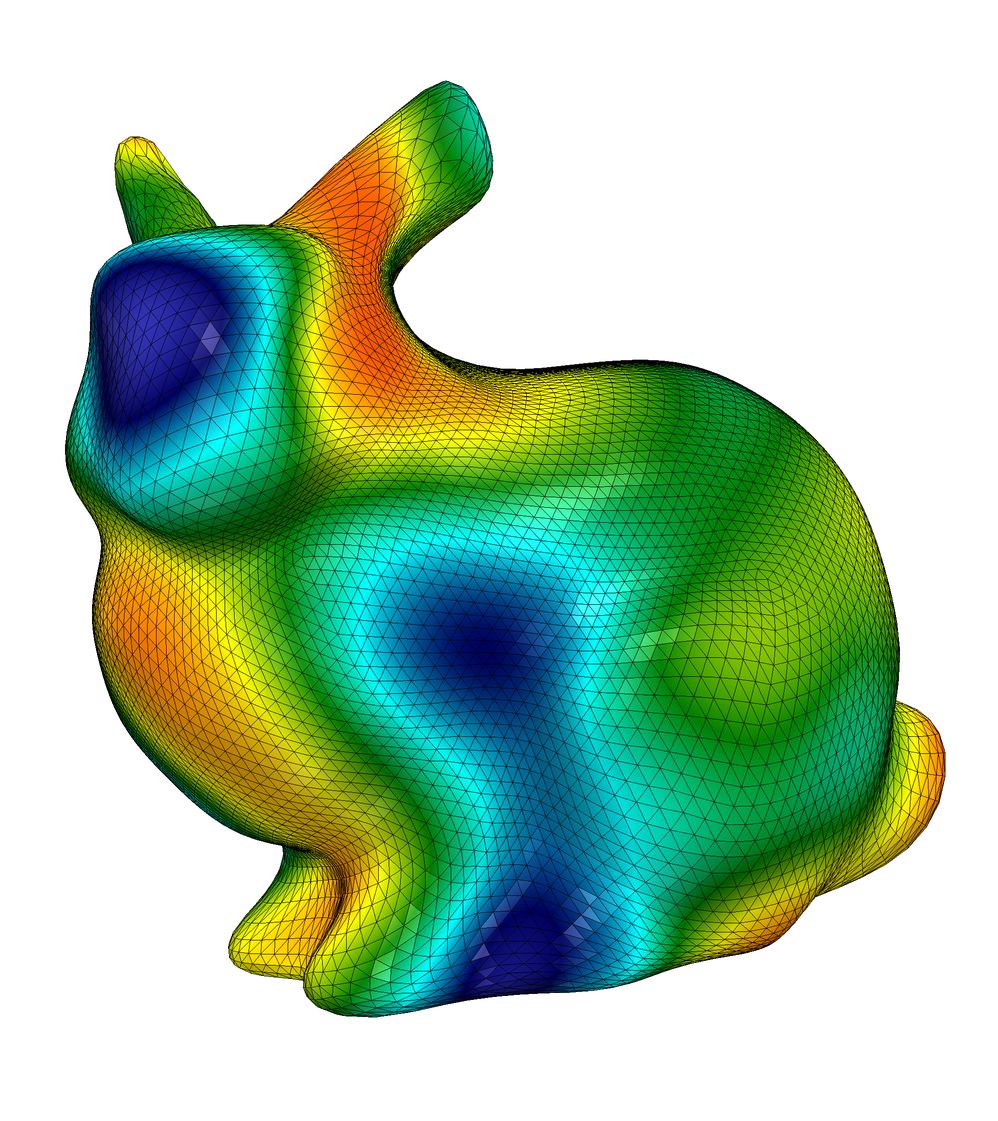} & \includegraphics[width=3.2cm]{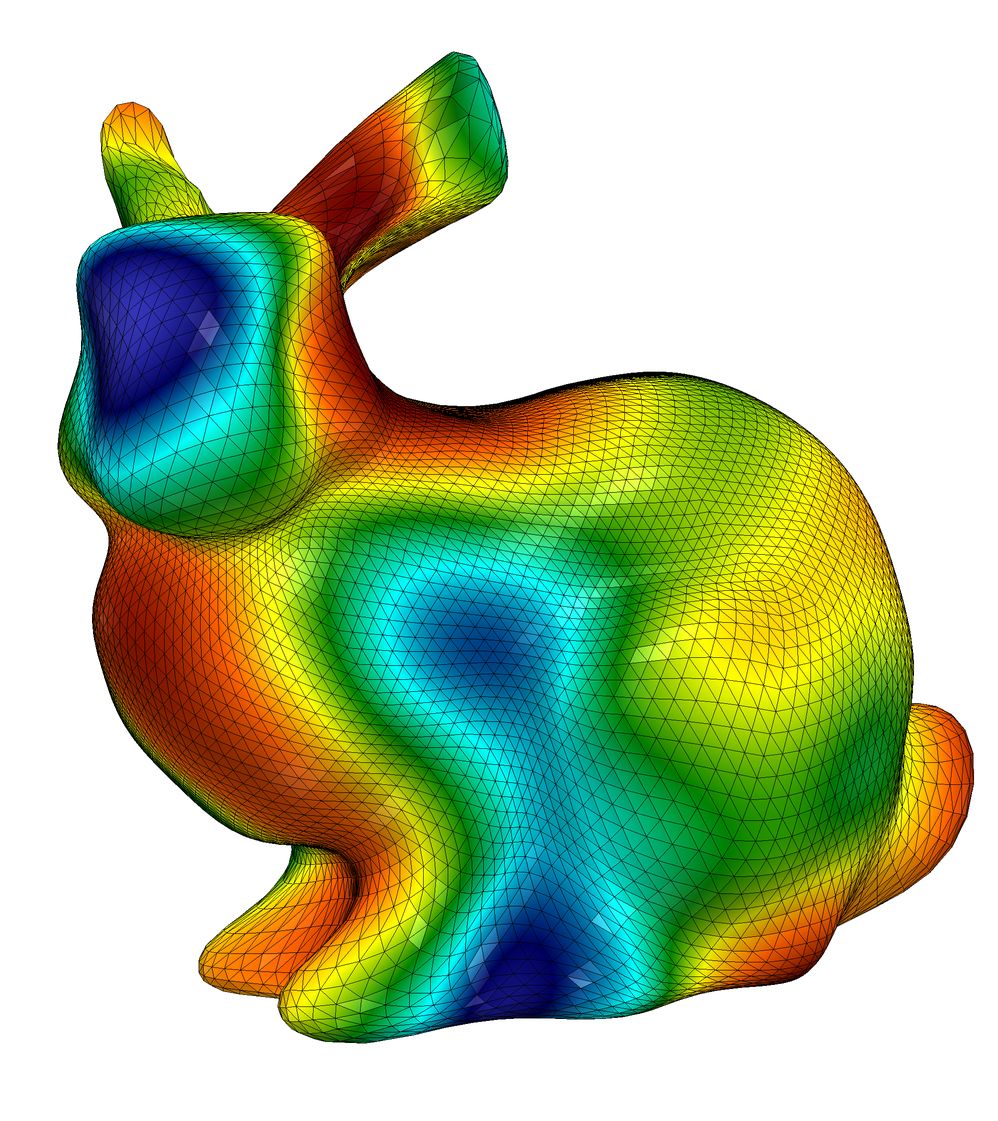} &  \includegraphics[width=3.2cm]{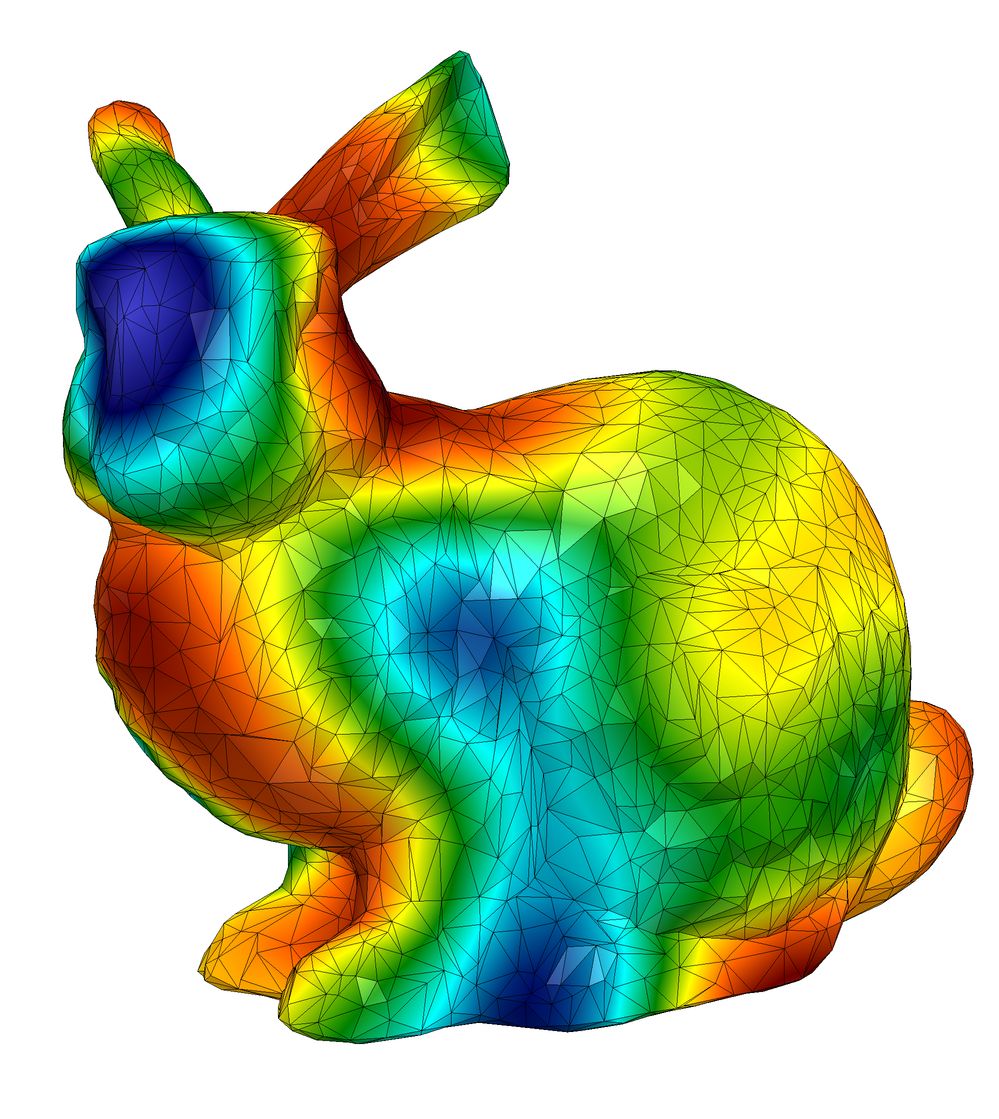}\\
  $t=0.8$ & $t=1$ & target 
 \end{tabular}
 
 \caption{$H^1$ metamorphosis of a sphere with constant signal onto a textured Stanford bunny.}
\label{fig:metam_bunny}
\end{figure}

\begin{figure}
\centering
  \begin{tabular}{ccc}
 \includegraphics[width=3.2cm]{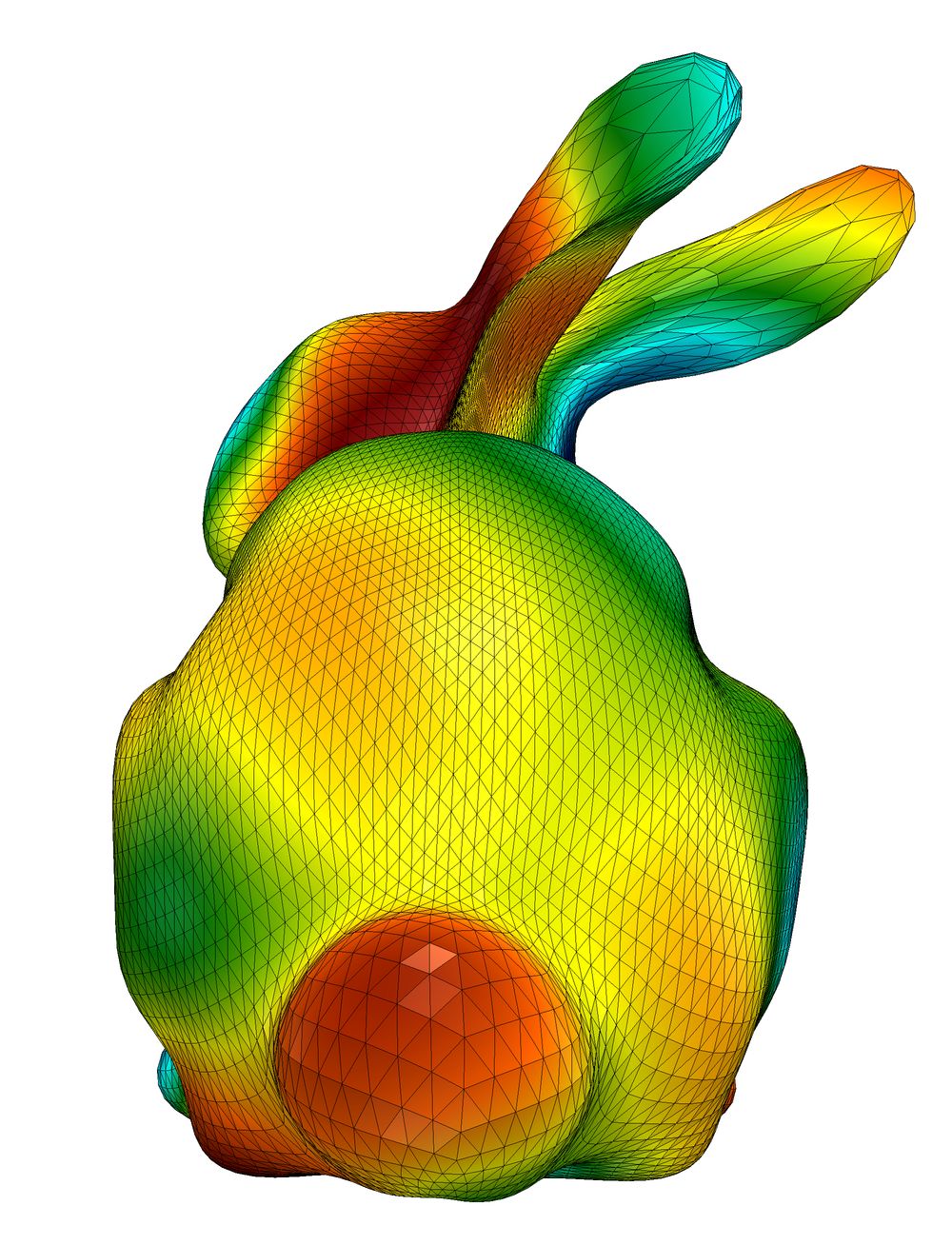}   & \includegraphics[width=3.2cm]{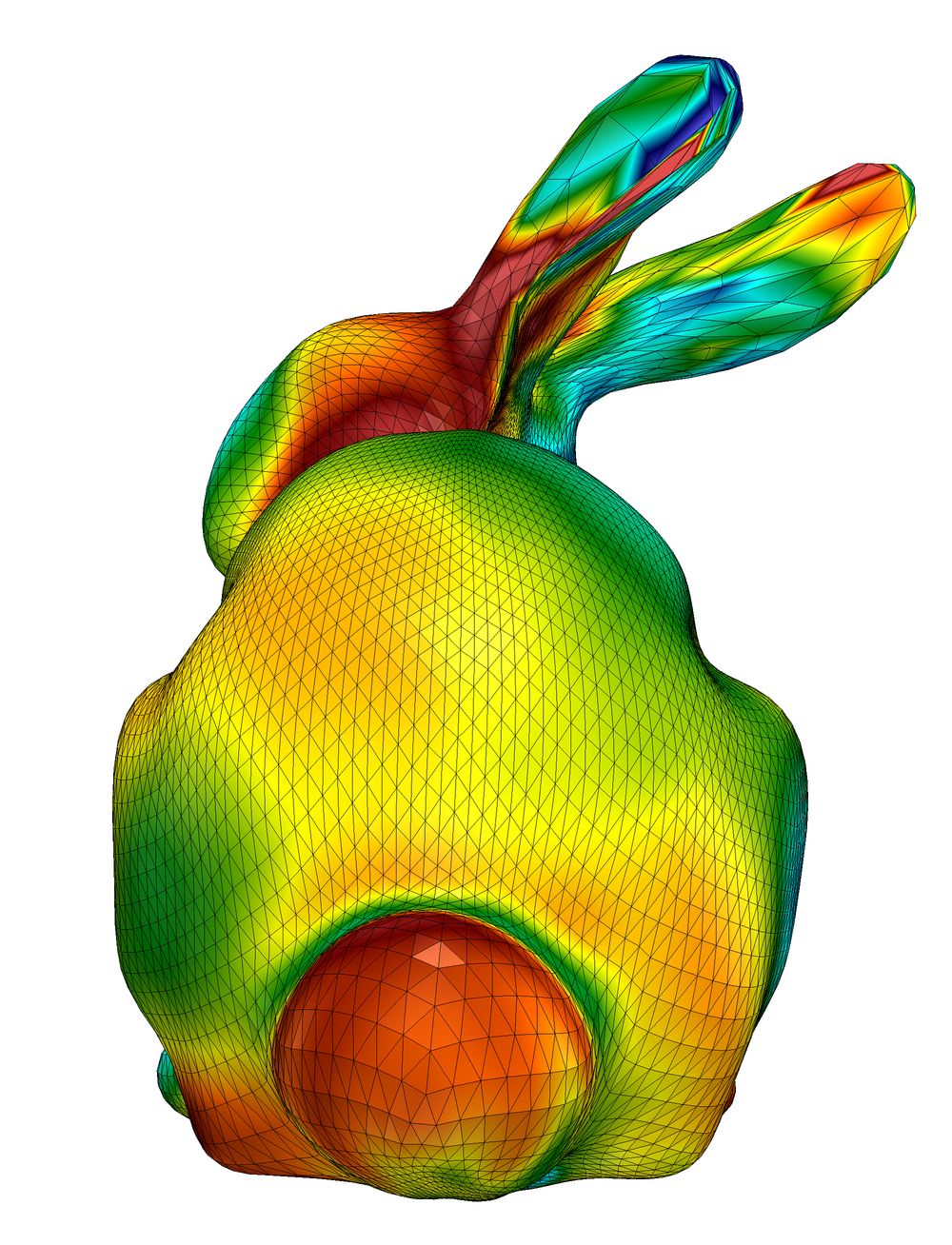}   & \includegraphics[width=3.2cm]{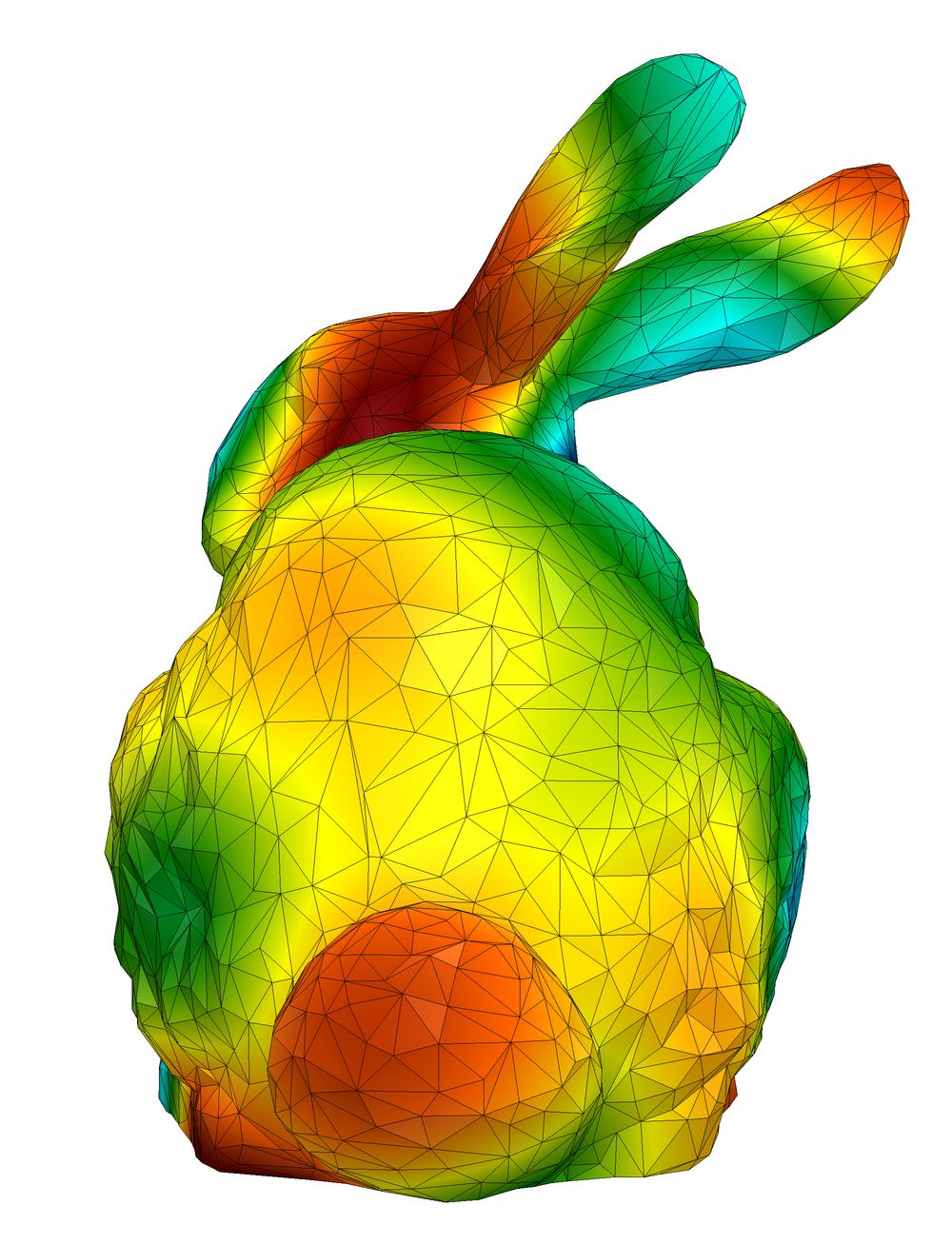}  \\
 \includegraphics[width=4cm]{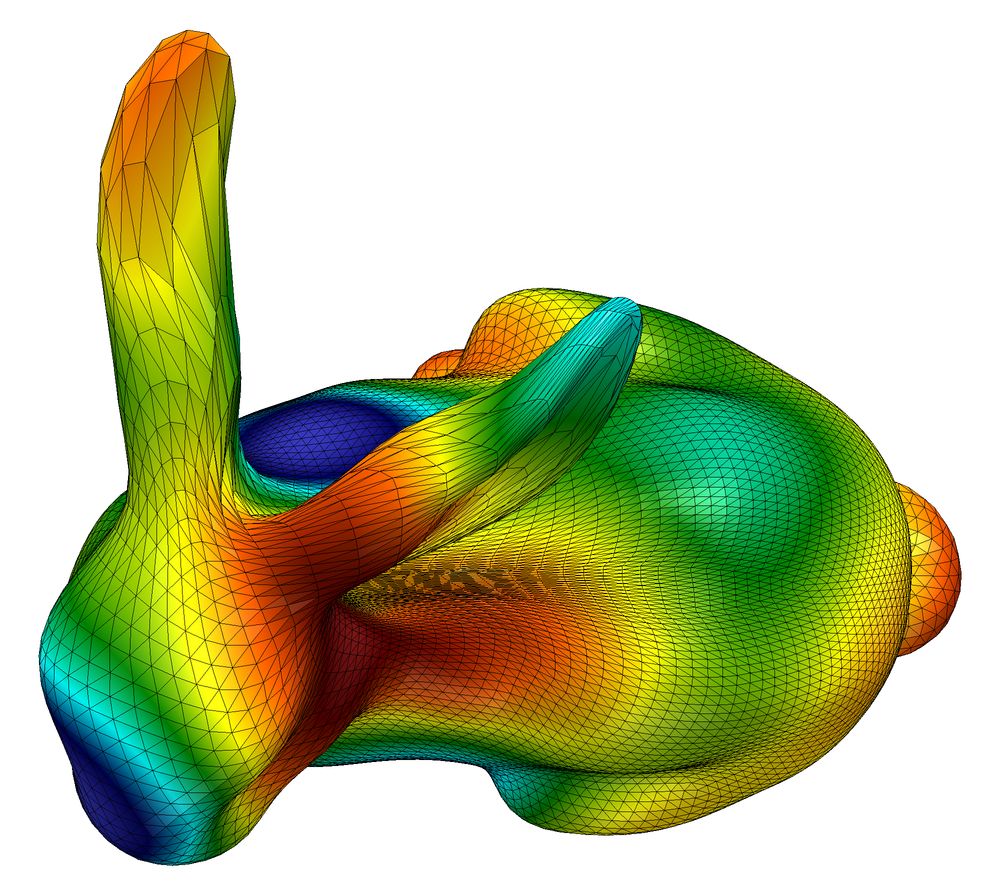} & \includegraphics[width=4cm]{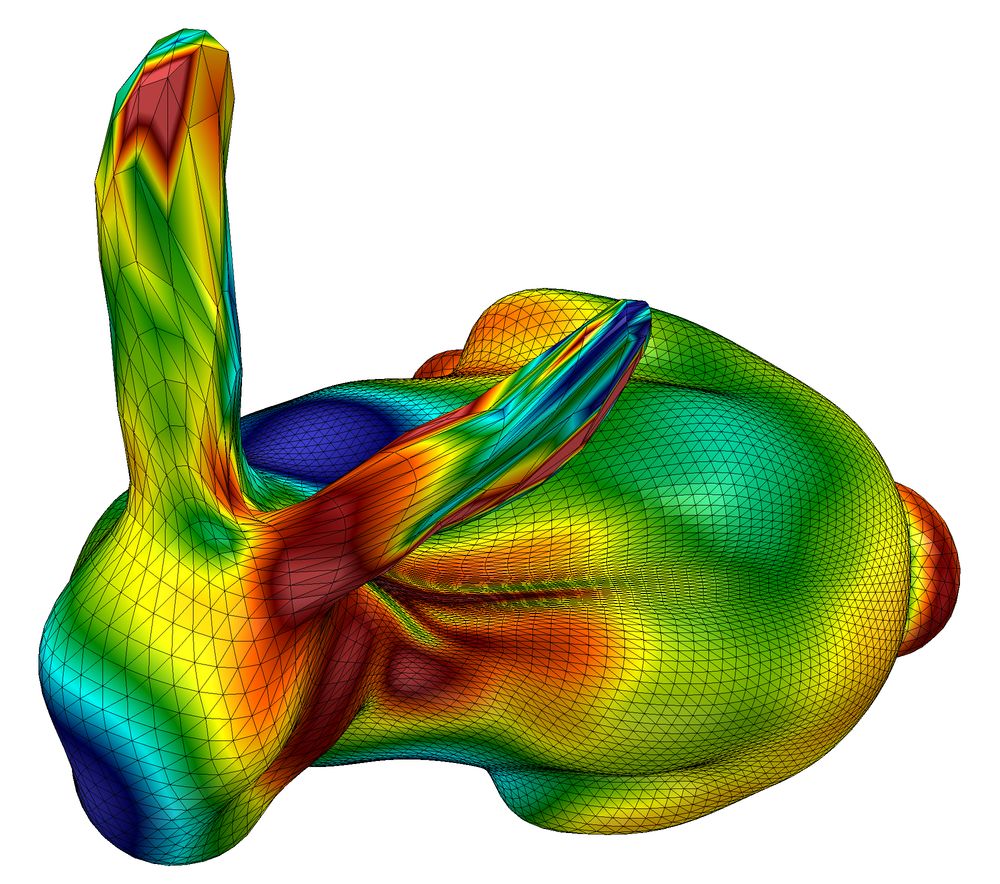} & \includegraphics[width=4cm]{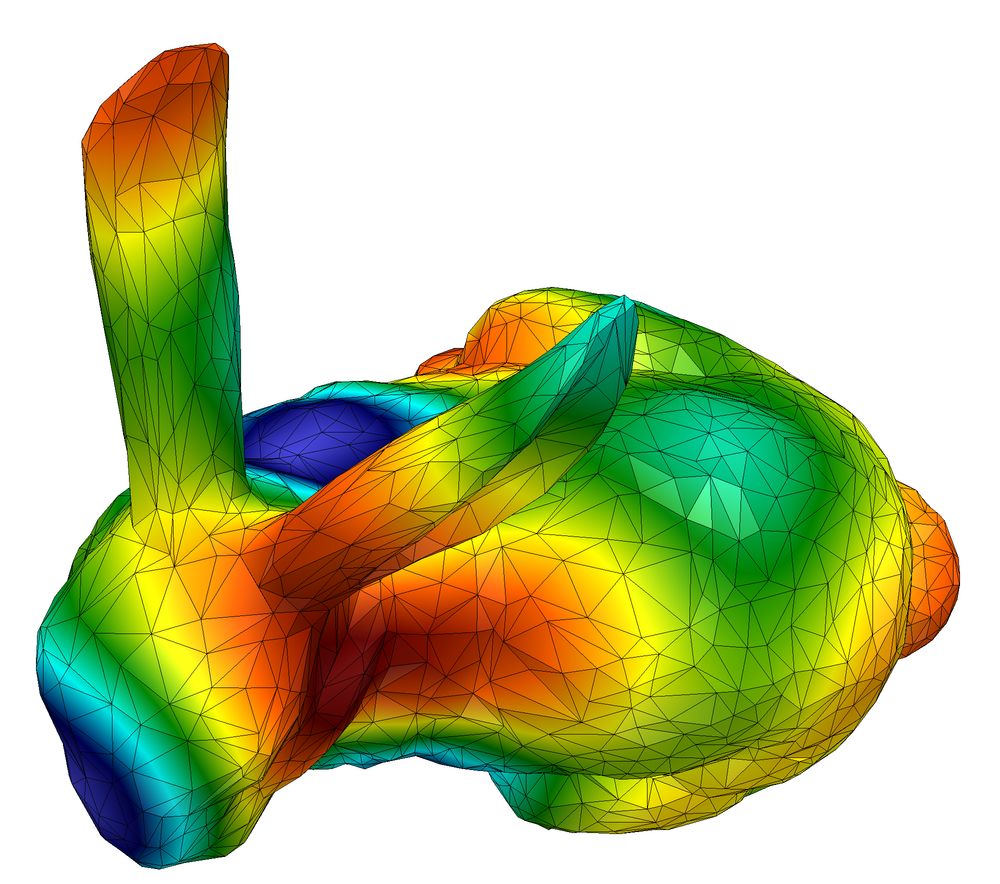}\\
 \includegraphics[width=4cm]{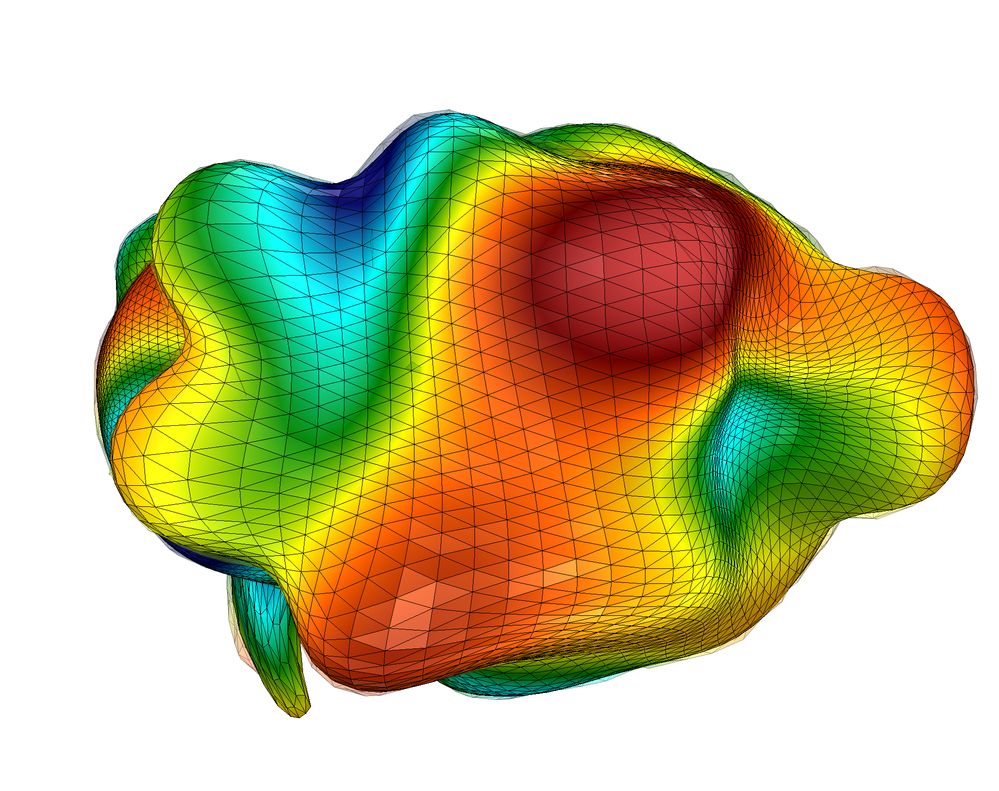} & \includegraphics[width=4cm]{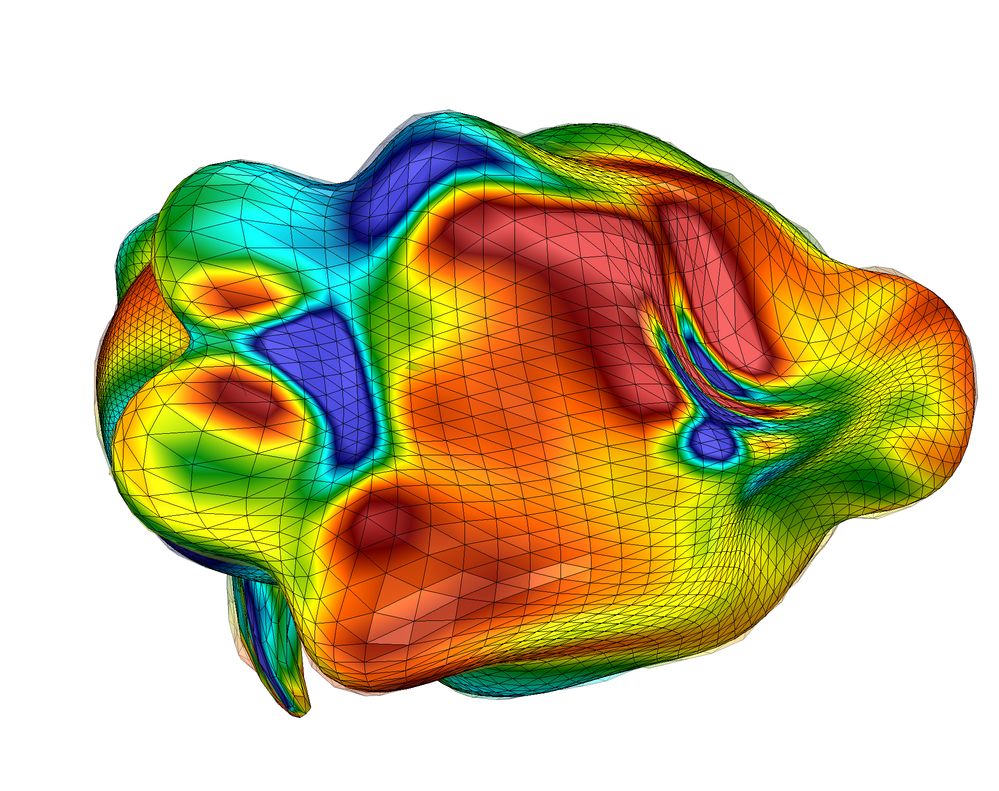} & \includegraphics[width=4cm]{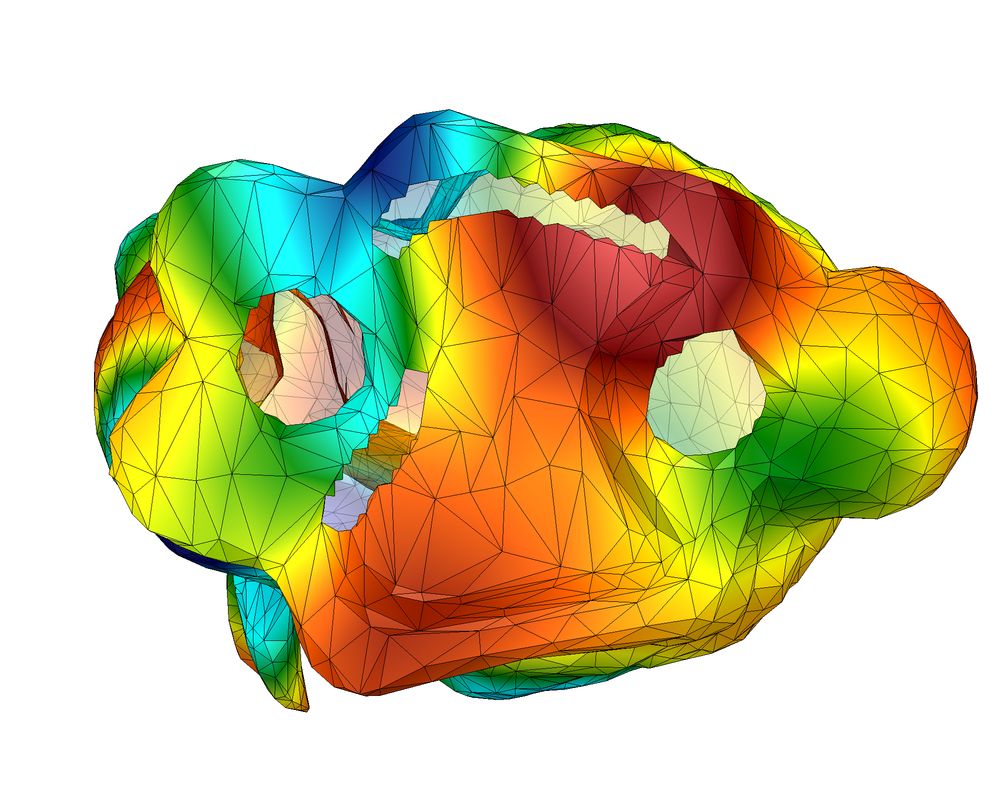} \\
 $H^1$ & $L^2$ & Target  
 \end{tabular}
 
 \caption{Comparison between $L^2$ and $H^1$ metamorphoses from different views. Last row, notice the holes underneath the target mesh: $H^1$ metamorphosis completes with smooth signal whereas $L^2$ metamorphosis creates oscillating signal.}
\label{fig:metam_bunny_h1_l2}
\end{figure}

\subsection{Real data}
The algorithm was also tested on some functional shapes occurring in medical imaging. In the following, we present a couple of qualitative results on these datasets mostly to try the behavior and robustness of the method on potentially more involved situations than the previous synthetic cases.

\paragraph{Thickness maps.} We first examine the output of metamorphosis matching (in $H^1$) on anatomical surfaces with estimation of the membrane thickness at each vertex. The first example in Figure \ref{fig:metam_OCT} is from a dataset of Nerve Fiber Layer (NFL) membranes in the retina with estimated measurements of thickness. The example corresponds to two age-matched subjects, one control and one affected by glaucoma. Each surface has 5000 vertices and the algorithm is run for 220 iterations in a total time of about 3.3 hours. We show the output metamorphosis together with the magnitude of the geometric momentum and the functional momentum. The deformation is mostly concentrated along the optical nerve opening while the functional momentum shows the overall decrease in thickness, particularly in a typical crescent region around the opening. Although illustrated here on two particular subjects, such anatomical effects have been analyzed and confirmed statistically in \cite{Lee15}.    

\begin{figure}
\centering
 \begin{tabular}{ccccc}
 \includegraphics[width=3.2cm]{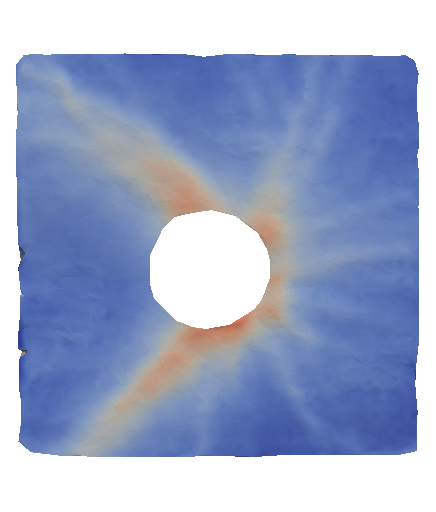} & \includegraphics[width=3.2cm]{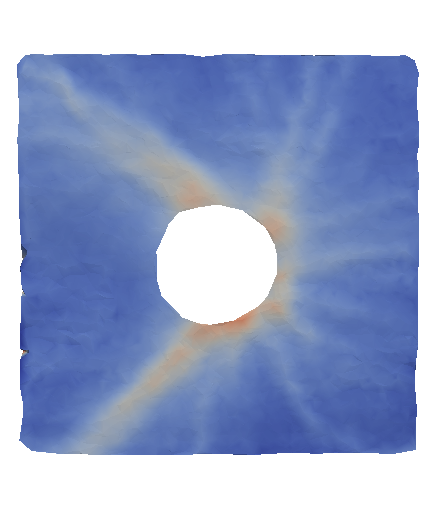} & \includegraphics[width=3.2cm]{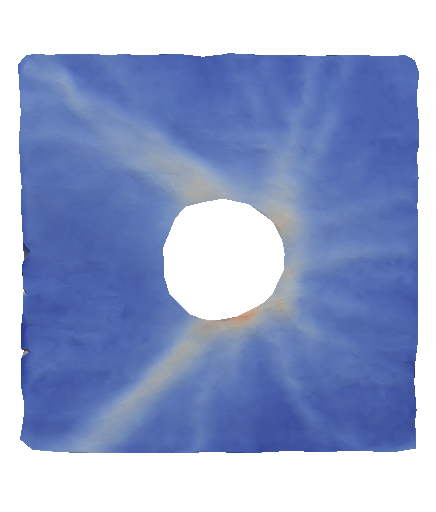} & \includegraphics[width=3.2cm]{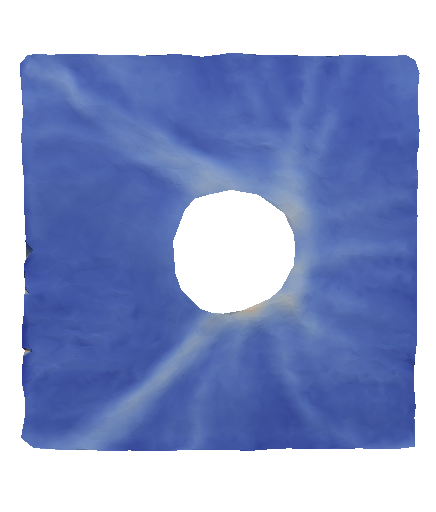} & \includegraphics[width=3.2cm]{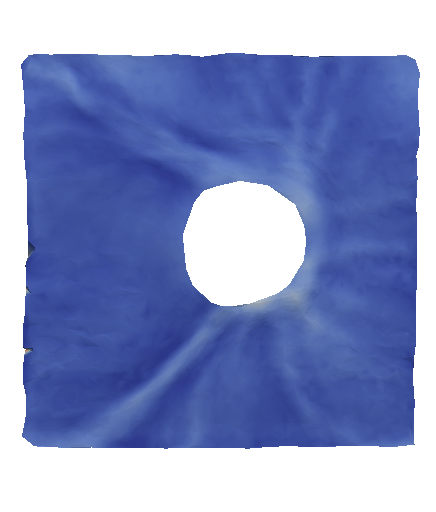}\\
 $t=0$ & $t=0.2$ & $t=0.4$ & $t=0.7$ & $t=1$   
 \end{tabular}
  \begin{tabular}{ccc}
 \includegraphics[width=3.5cm]{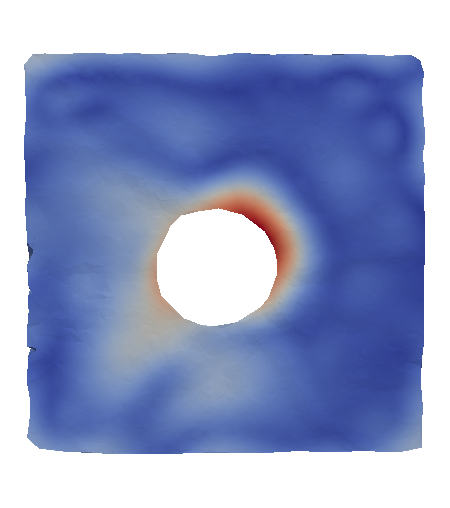} & & \includegraphics[width=3.5cm]{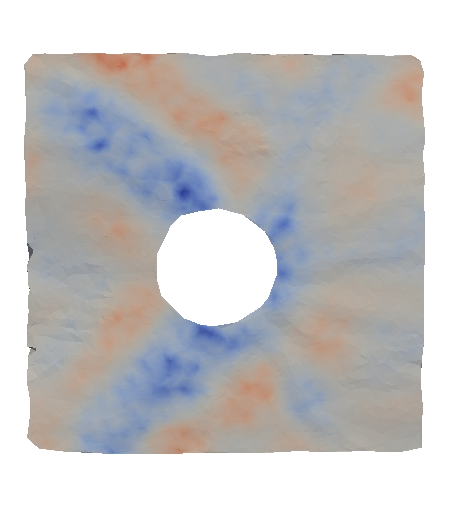}\\
 Geometric momentum (magnitude) & & Functional momentum 
 \end{tabular}
 
 \caption{Metamorphosis between two subjects of the Nerve Fiber Layer dataset (data courtesy of S. Lee, M. Sarunic, F. Beg, Simon Fraser University).}
\label{fig:metam_OCT}
\end{figure}

\begin{figure}[t]
\centering
 \begin{tabular}{ccc}
 \includegraphics[width=4.5cm]{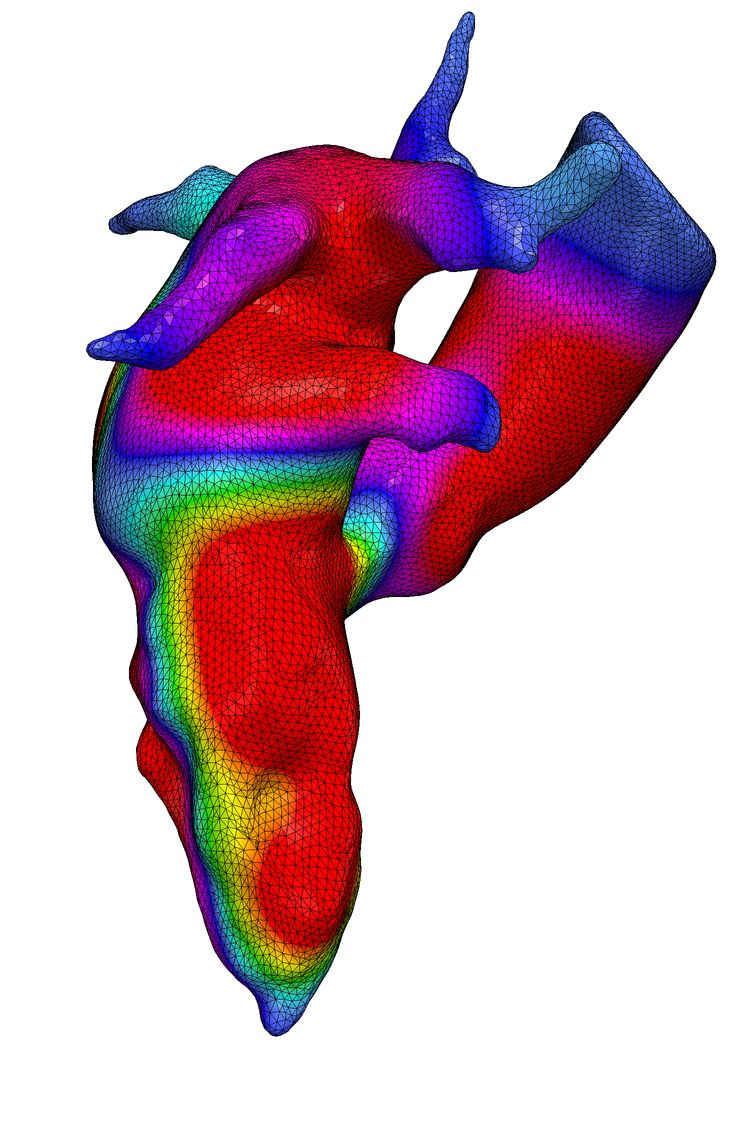} &\includegraphics[width=4.5cm]{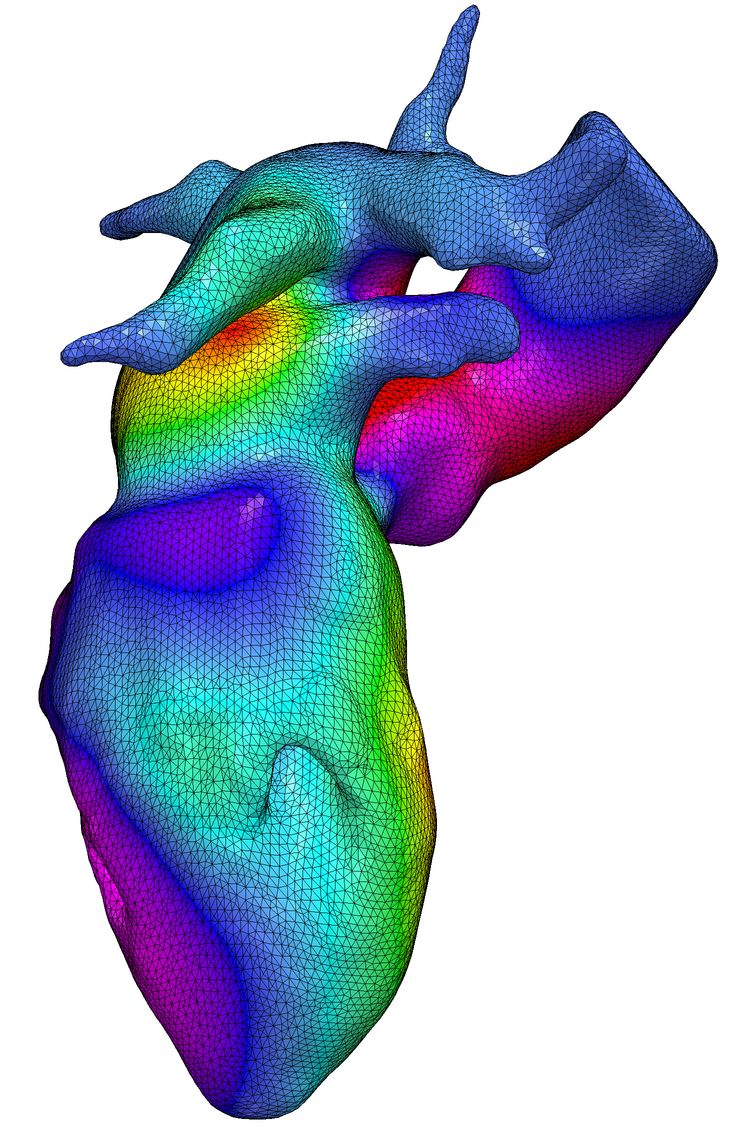}\\
  Source & Target &
 \end{tabular}
 \caption{Valves dataset: heart surfaces with pressure (data courtesy of C. Chnafa, S. Mendez and F. Nicoud, University of Montpellier)}
\label{fig.valvee}
\end{figure}

\paragraph{Heart pressure.} As a last example, we consider a surface of heart with signals corresponding to simulated pressure maps on the membrane (see Figure \ref{fig.valves}). We show the time evolution obtained from the metamorphosis matching algorithm between the initial and final states of the cardiac cycle in Figure \ref{fig:metam_heart}. Surfaces have approximately 26000 vertices, and the algorithm took on the order of 6 hours to reach convergence. It is also interesting to compare the resulting fshape evolution to the output of another model and algorithm for fshape matching (cf Figure \ref{fig:tan_heart}): the 'tangential' model studied in \cite{Charlier15}. In the latter, the penalty on signal variations is measured with the metric of the reference template only as opposed to evolving the metric with the shape in metamorphosis. The dynamics of signal evolution is then always a simple affine interpolation in $t$ between the initial and final values whereas the metamorphosis evolution tends to show an early acceleration of signal decrease on the lower valve that is inflating.

\begin{figure}[t]
\centering
 \begin{tabular}{ccc}
 \includegraphics[width=4.5cm]{valves_met_00.jpg} &\includegraphics[width=4.5cm]{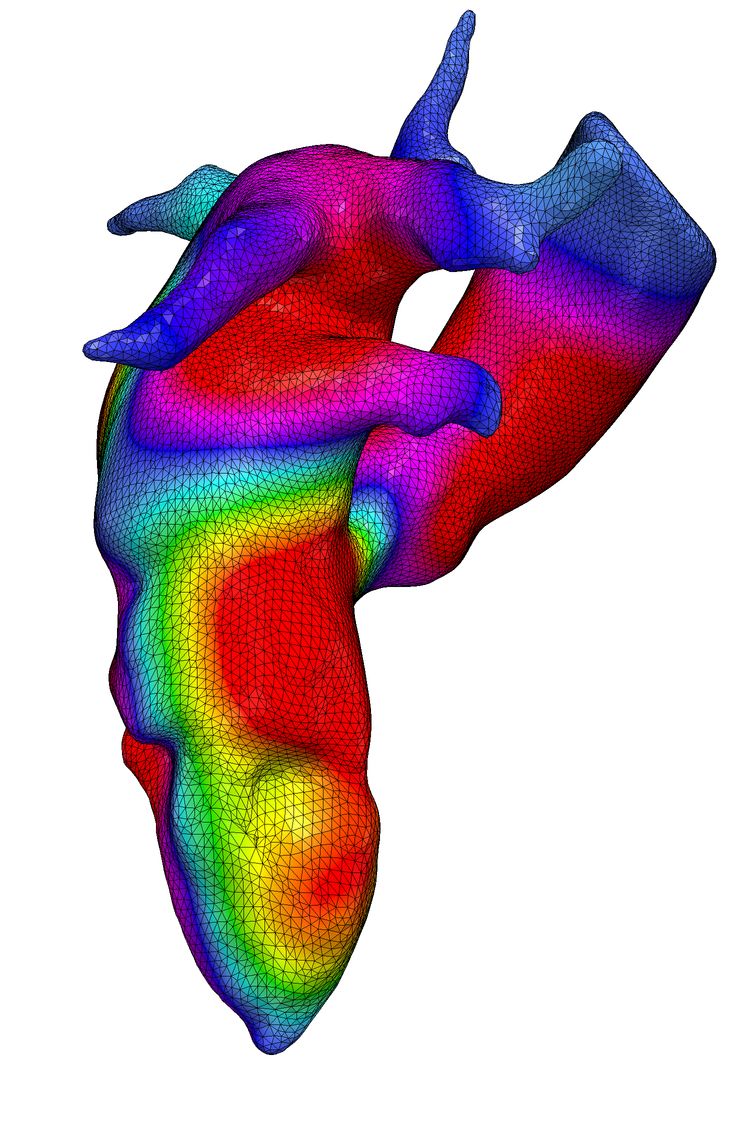} & \includegraphics[width=4.5cm]{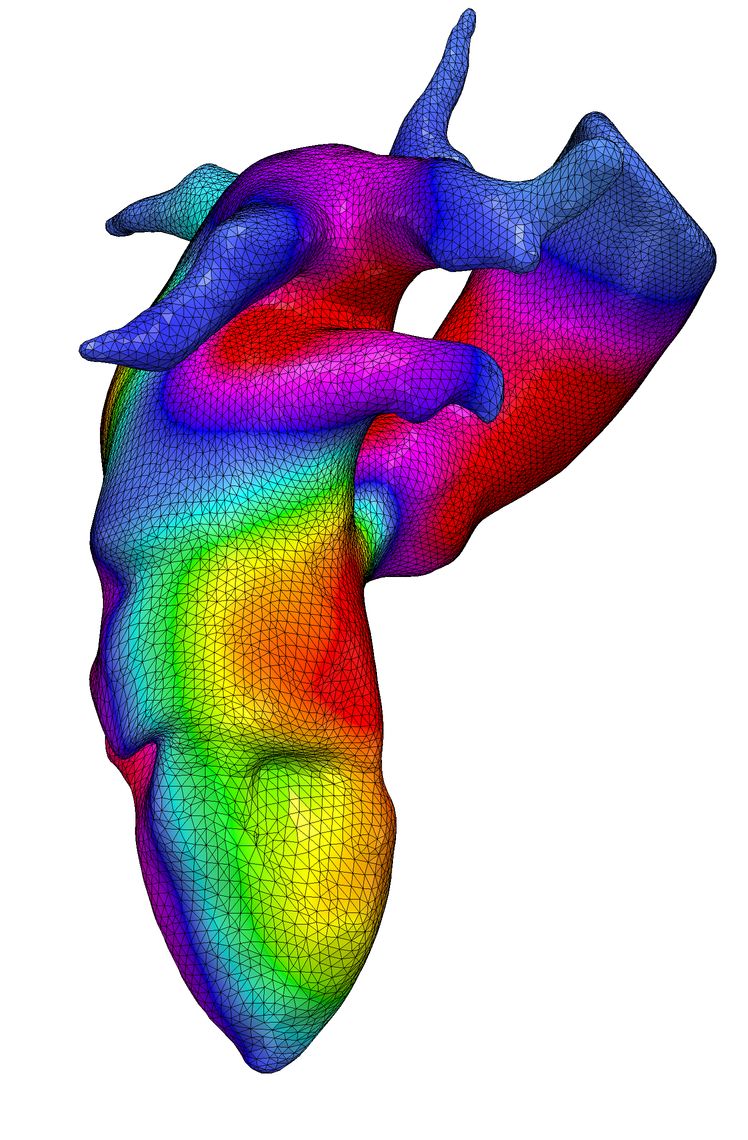} \\ Source & $t=0.2$ & $t=0.4$ \\
 \includegraphics[width=4.5cm]{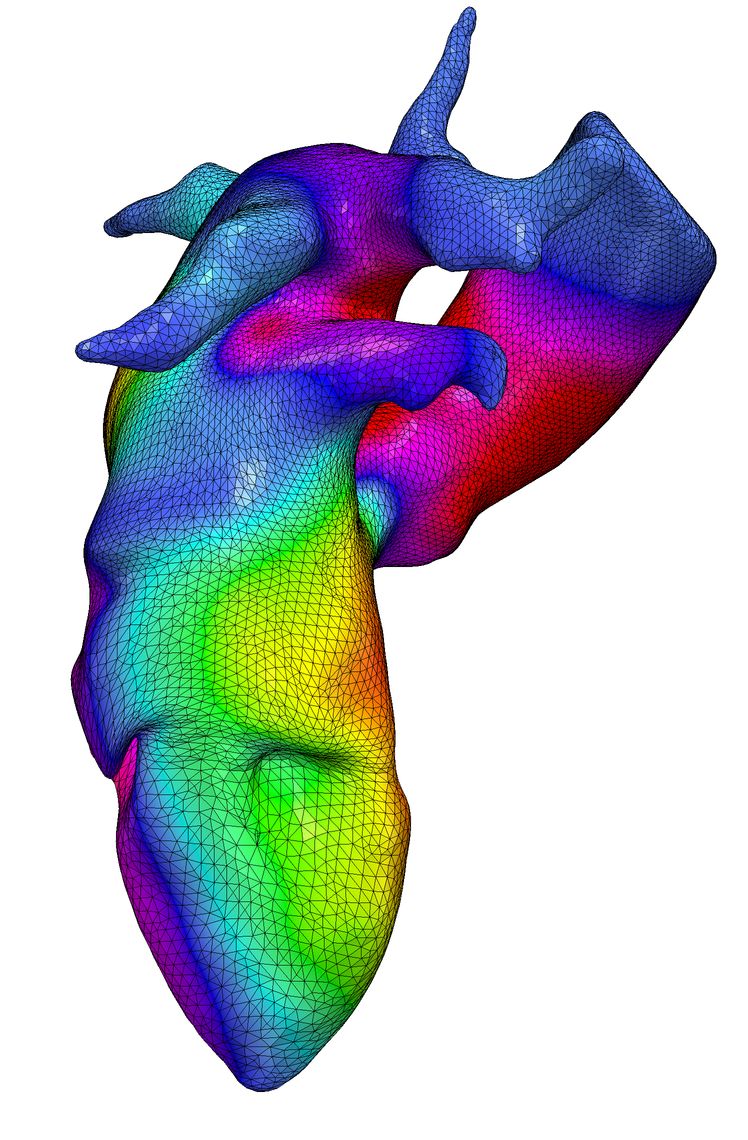} & \includegraphics[width=4.5cm]{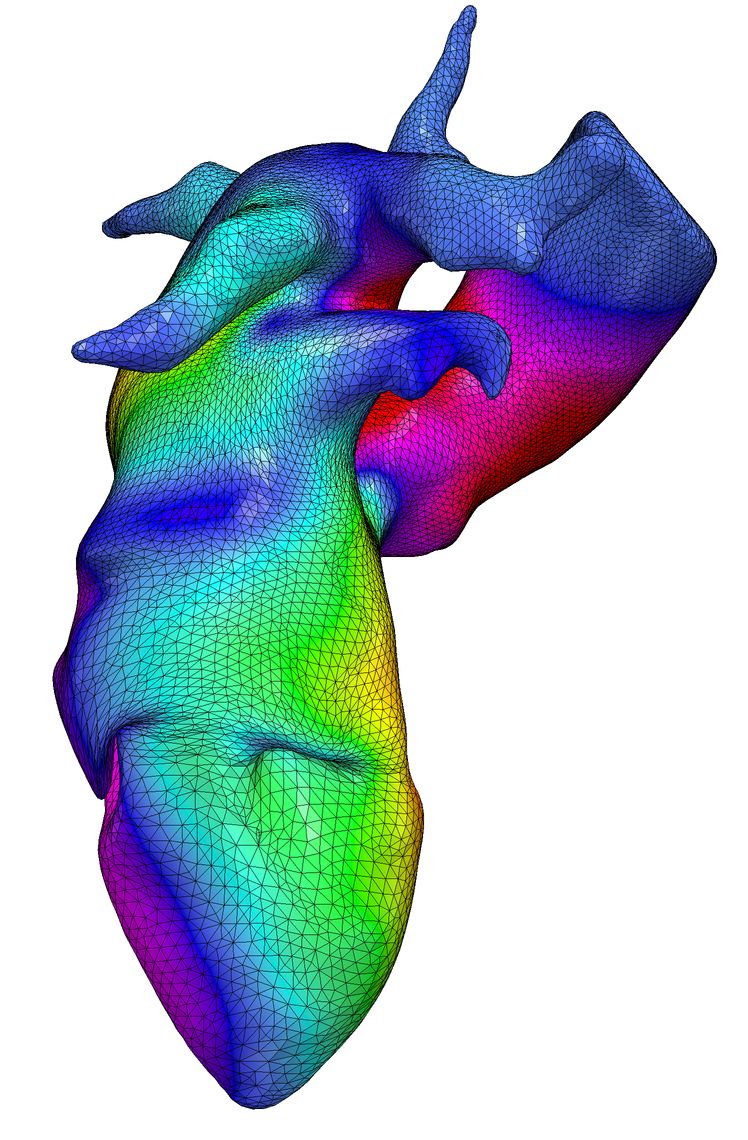} & \includegraphics[width=4.5cm]{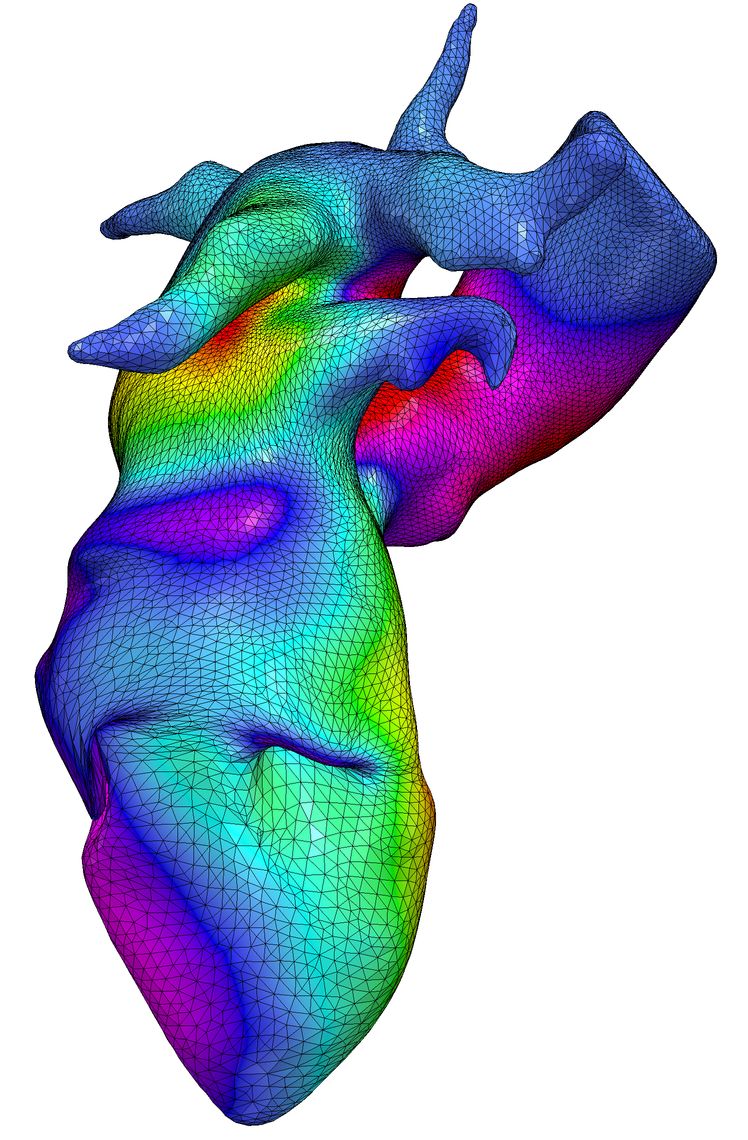}\\ $t=0.6$ & $t=0.8$ & $t=1$ 
 \end{tabular}
 \caption{Metamorphosis with $H^1$ regularity}
\label{fig:metam_heart}
\end{figure}

\begin{figure}[t]
\centering
 \begin{tabular}{ccc}
 \includegraphics[width=4.5cm]{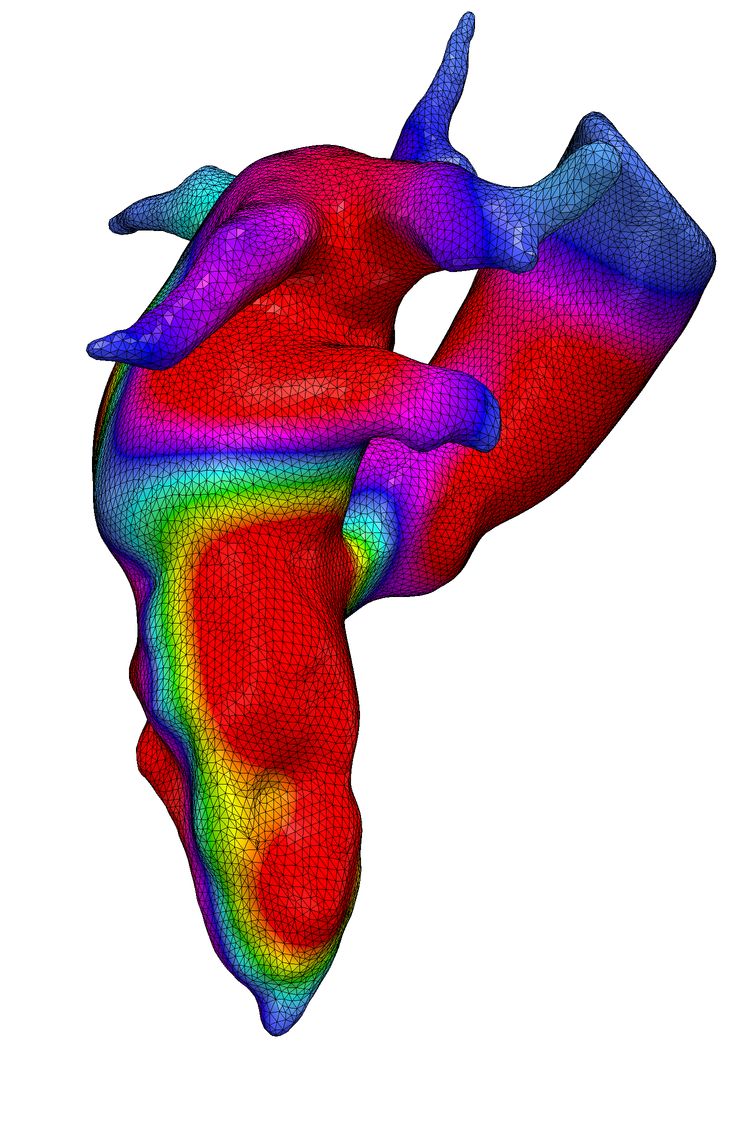} &\includegraphics[width=4.5cm]{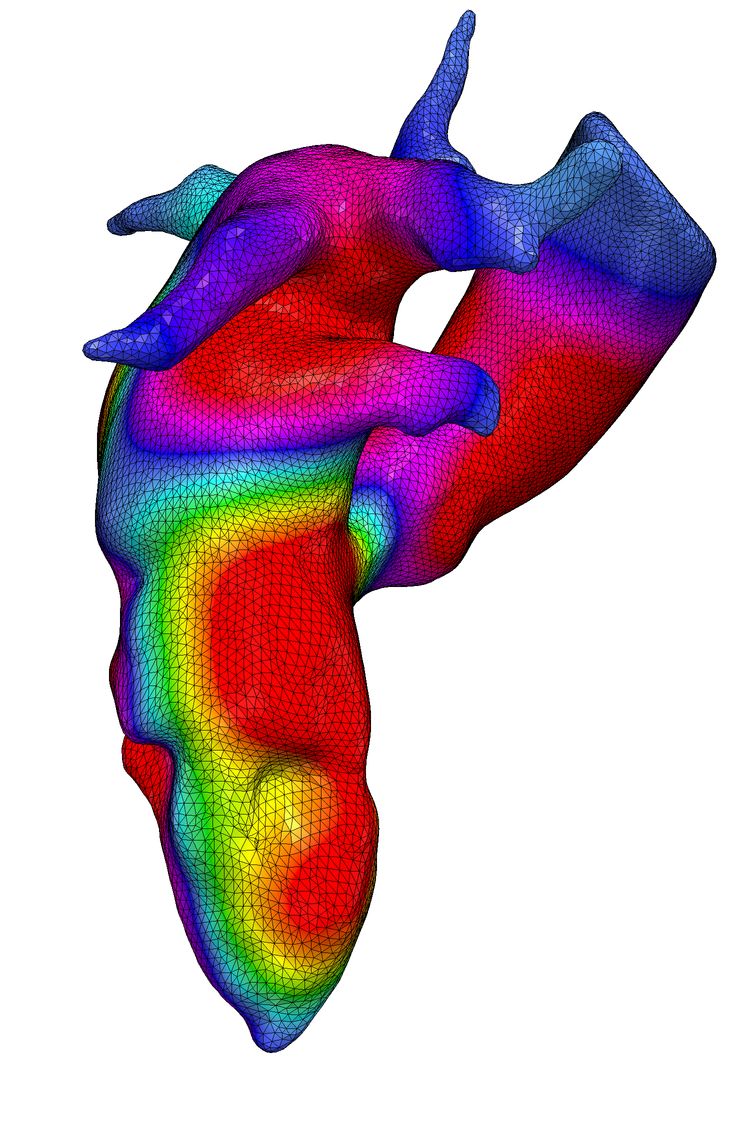} & \includegraphics[width=4.5cm]{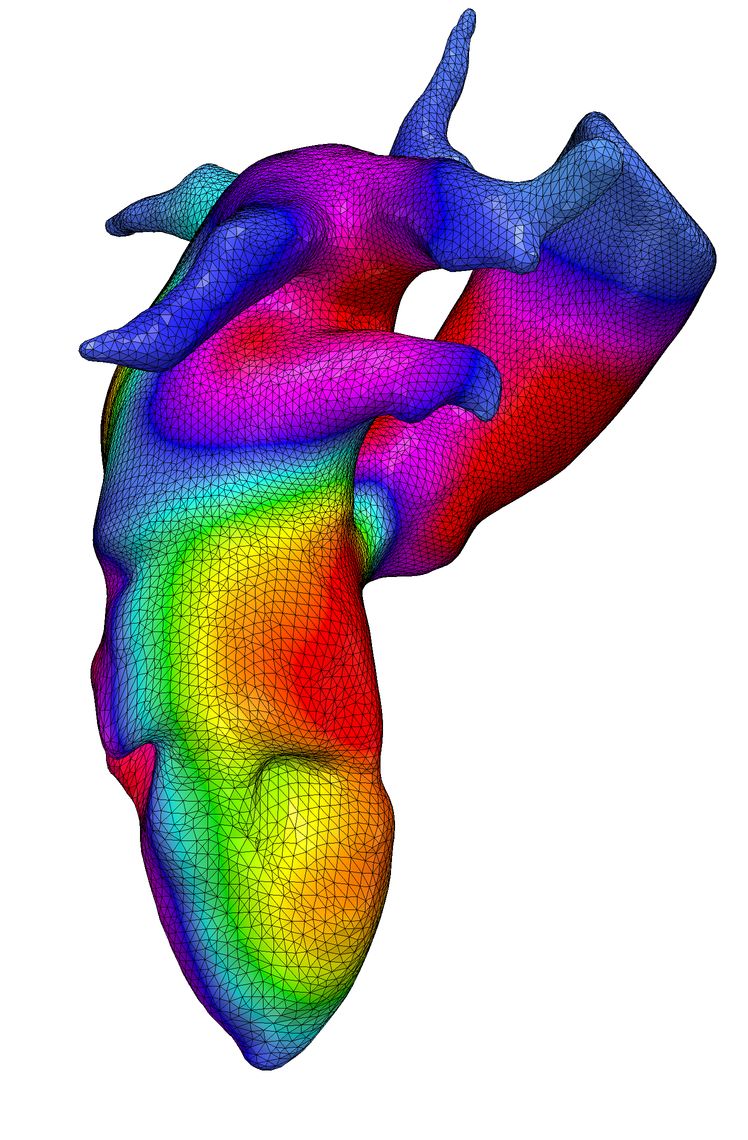} \\ Source & $t=0.2$ & $t=0.4$ \\
 \includegraphics[width=4.5cm]{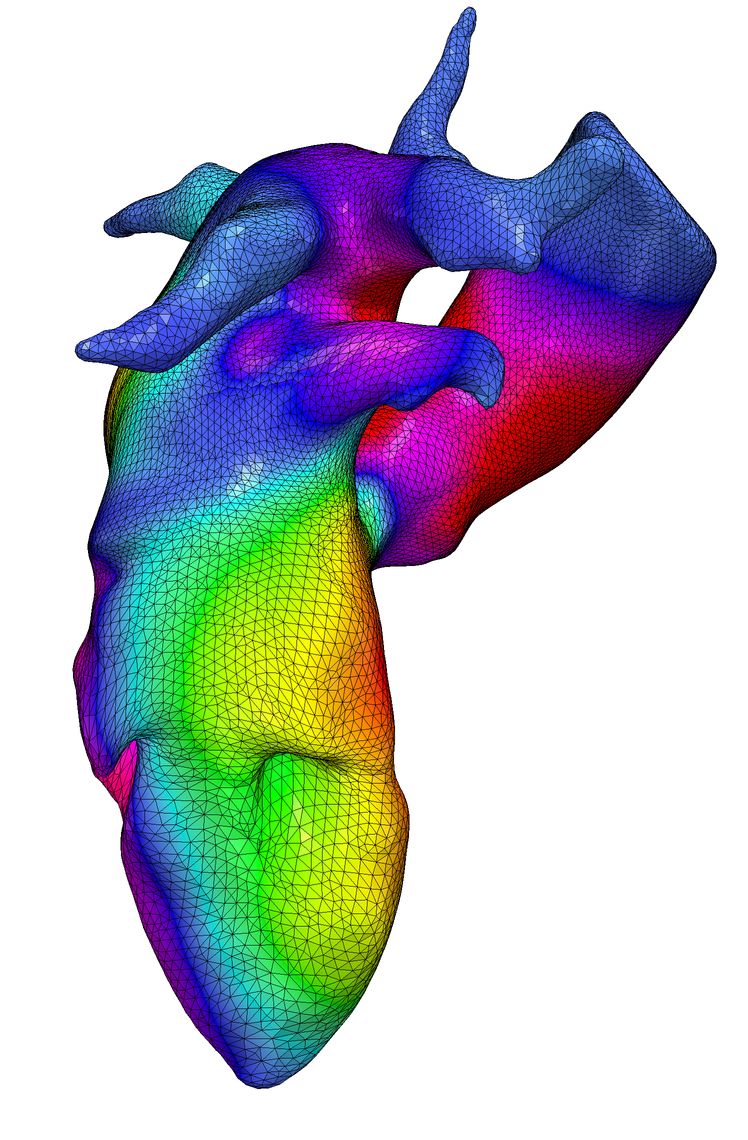} & \includegraphics[width=4.5cm]{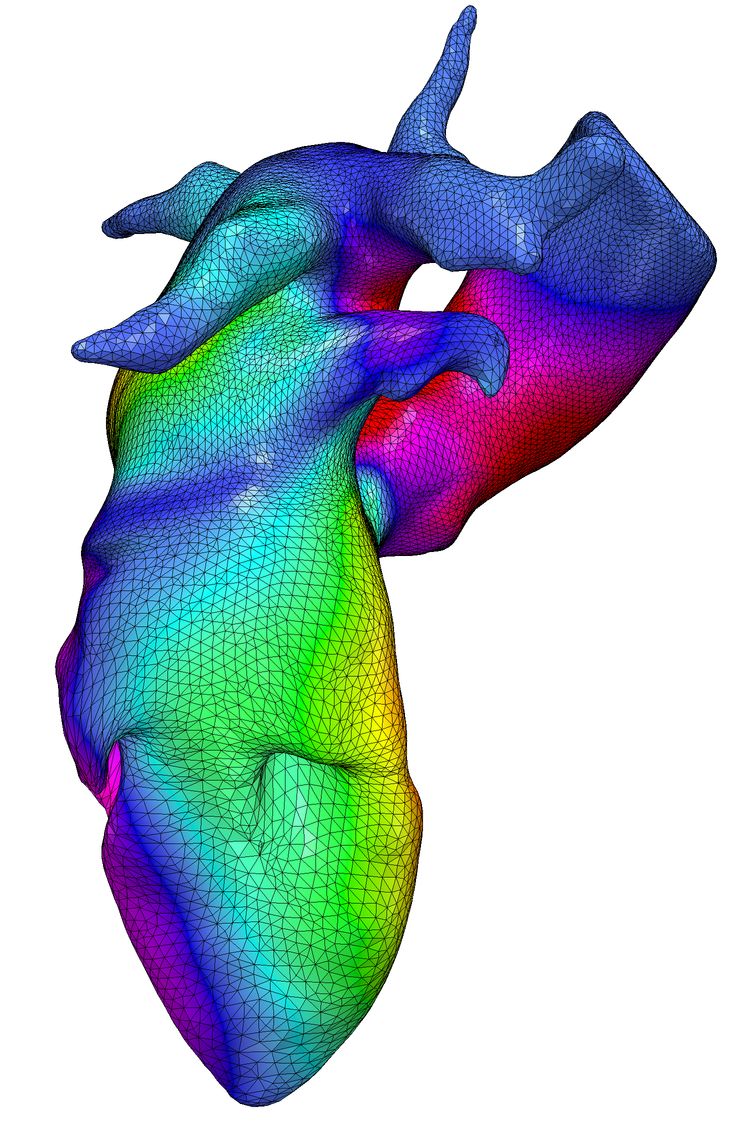} & \includegraphics[width=4.5cm]{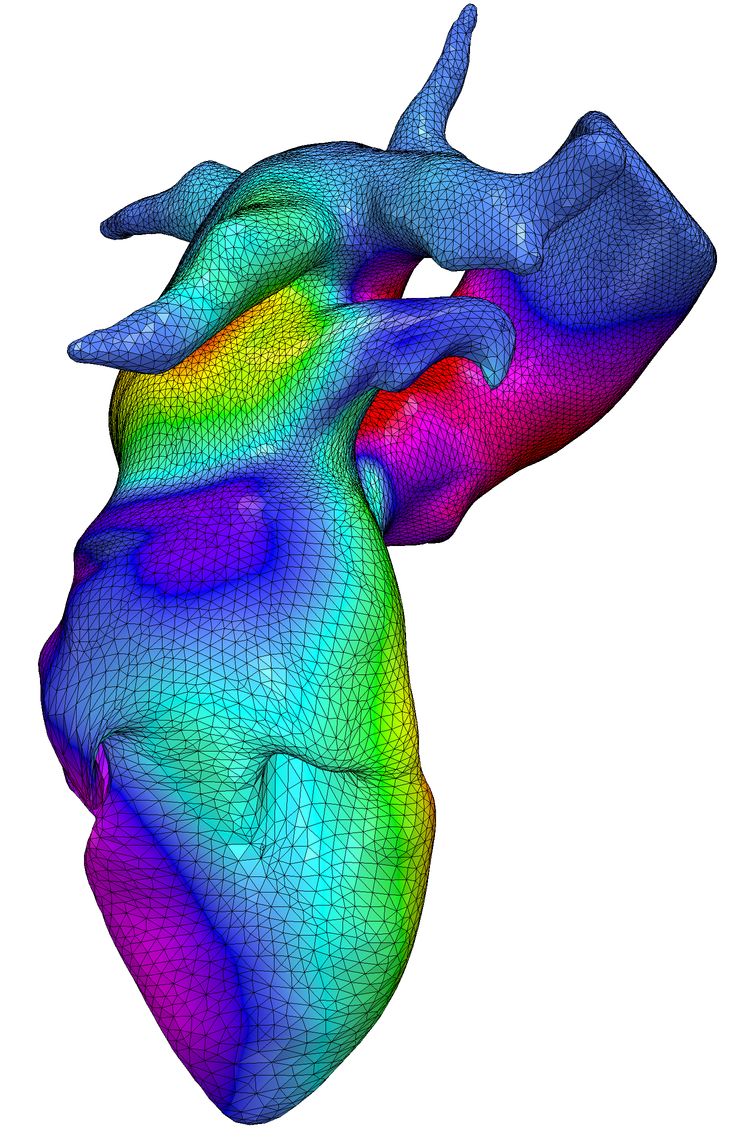}\\ $t=0.6$ & $t=0.8$ & $t=1$ 
 \end{tabular}
 \caption{Comparison with the transformation obtained from the 'tangential model' (with $H^1$ penalty)}
\label{fig:tan_heart}
\end{figure}

\subsection{Conclusion and discussion}
We have presented a new model for the representation and registration of fshapes, i.e objects combining a deformable geometric support with a photometric component. From a theoretical standpoint, this model extends the existing idea of metamorphosis on flat images and, unlike earlier approaches like the tangential model of \cite{Charlier15}, leads to a well-defined complete metric space structure when restricting to fshape bundles. In addition, the framework was derived for the class of signals of higher Sobolev regularity on the manifolds which we showed is necessary in certain instances. 

This was then cast into a formulation for geometric-functional matching between two given fshapes, combining metamorphosis energy with data attachment terms based on functional varifolds. We have shown that it is a well-posed optimal control problem (with some conditions on energy weights in the case $s=0$) and investigated carefully the Hamiltonian dynamics of minimizers as well as the equivalent of the EP-diff conservation equation for that model. We have also derived the corresponding discrete model and algorithm to numerically solve the matching problem in the cases $s=0$ and $s=1$. Questions regarding the $\Gamma$-convergence of the discrete to the continuous models is left for future study, although a significant step was made in that direction with the results of \cite{Nardi2015}. Still, numerical simulations show the ability of this approach to recover joint geometric and photometric variations between a given template and target fshape at the price of extra parameters in the model and extra numerical cost compared to a pure diffeomorphic registration. 

The approach was restricted here to the problem of matching between two subjects, the direct follow-up being to extend the model and algorithm to \textit{atlas estimation} on populations, following the footsteps of \cite{Charlier15,Lee15}. One advantage to expect from it is that the metric framework we obtain from metamorphosis would provide a more theoretically suitable setting to statistical analysis on those geometric-functional transformations.  

A second clear restriction of the paper comes from the very nature of signals and the definition of geometric action equation \eqref{eq:action_defo_fshape} that was considered here. The model was indeed built on standard image deformation action and is therefore not necessarily adapted to all types of functional maps. Other typical cases could involve densities, vector fields, tensor fields on shapes for which the transport equations could significantly differ from equation \eqref{eq:action_defo_fshape} and so would the associated Hamiltonian dynamics and the behavior of geodesics. We postulate however that a very similar approach to the one developed here could be undertaken with other signal spaces or group actions and lead to interesting extensions of the present work.

\subsection*{Acknowledgements}
The authors would like to thank Sylvain Arguill\`{e}re for many interesting discussions on the optimal control aspects of this manuscript.

\appendix
\section{Proof of Theorem \ref{theo:Sobolev_control_phi}}
\label{appendix:proof_theorem_control_phi}  
Before the actual proof of Theorem \ref{theo:Sobolev_control_phi}, we shall introduce a few definitions and intermediate results. Let $s\geq 0$ and $s'=\max(s,1)$ and we recall that $X$ is a compact submanifold of $\R^n$ of dimension $d$ and class $C^{s}$ and that $V\hookrightarrow C^{s'}_0(\R^n,\R^n)$. For a given coordinate system $(x^i)_{1\leq i\leq d}$, we will denote respectively by $(\partial_{i})$ and $(dx^i)$ the corresponding frame and coframe. We introduce the following class of sections over the $(a,b)$ tensor bundle:
\begin{de}
We say that $A\in\Gampol^{p,s}(T^a_b(U))$ with $a,b,p,s \in \N$, $a,b\leq s$ and $p<s'$ if there exists a coordinate system $(x^i)_{1\leq i\leq d}$ on $U$ such that for any $(\phi,u)\in \Diff^s_0(\R^n)\times U$
$$A(\phi,u)=A^\alpha_\beta(\phi,u)\partial^a_\alpha\otimes dx^\beta$$
where for any compact $K\subset U$ there exists two polynomials $P$ and $Q$ such that for any multi-indices $\alpha,\beta$ and for any $\phi,\phi'\in\Diff^{s}_0$ we have 
$$u\mapsto A^\alpha_\beta(\phi,u))\in C^p(U,\R)\text{, }
\sup_{K,k\leq p}|\partial^kA^\alpha_\beta(\phi,u)|\leq P(\rho_{s}(\phi))\,,$$
and
$$\sup_{K,k\leq p}|\partial^kA^\alpha_\beta(\phi,u)-\partial^kA^\alpha_\beta(\phi',u)|\leq \rho_{s}(\phi'\circ\phi^{-1})Q(\rho_{s}(\phi),\rho_{s}(\phi'))$$
with the notation $\rho_s(\psi)\doteq \sum_{k\leq s} \|d^k (\psi-\Id)\|_\infty+\|d^k(\psi^{-1}-\Id)\|_\infty$ for any $\psi\in \Diff^s_0(\R^n)$.
\end{de}
In the previous, we point out that $\alpha$ and $\beta$ are multi-indices of integers between $1$ and $d$ such that $|\alpha|=a, \ |\beta|=b$. When $a=b=0$, the space $\Gampol^{p,s}(T^a_b(U))$  will be denoted $\Cpol^{p,s}(U)$.
\begin{remark}
\label{rem:27.8.1}
  A first important remark is that the definition is not dependent on the choice of the coordinate system.  Indeed, if $s=0$, we have $a=b=p=0$ and the definition does not depend on any coordinate system. If $s=1$ then $p=0$ and if $(y_1,\cdots,y^d)$ is another coordinate system, it is sufficient to notice that $\frac{\partial}{\partial x^i}=\frac{\partial y^j}{\partial x^i}\frac{\partial }{\partial y_j}$  and $dx^i=\frac{\partial x^i}{\partial y^j}dy^j$ where the mappings $\frac{\partial y^j}{\partial x^i}$ and $\frac{\partial x^i}{\partial y^j}$ are continuous and bounded on $K$. Last, if $s\geq 2$, we get $\tilde{A}^\talpha_\tbeta(\phi,u)\frac{\partial^a}{\partial y^\talpha}\otimes dy^\tbeta=A^\alpha_\beta(\phi,u)\frac{\partial^a}{\partial x^\alpha}\otimes dx^\beta$ for $\tilde{A}^\talpha_\tbeta(\phi,u)=A^\alpha_\beta(\phi,u)\frac{\partial y^\talpha}{\partial x^\alpha}\frac{\partial x^\beta}{\partial y^\talpha}$ with  $\frac{\partial y^\talpha}{\partial x^\alpha}=\prod_{i=1}^a\frac{\partial y^{\talpha_i}}{\partial x^{\alpha_i}}\in C^{s-1}(U,\R)$ and  $\frac{\partial x^\beta}{\partial y^\tbeta}=\prod_{i=1}^b\frac{\partial x^{\beta_i}}{\partial y^{\tbeta_i}}\in C^{s-1}(U,\R)$. Since $p\leq s-1$, we deduce that $\tilde{A}^\talpha_\tbeta(\phi,u)\in C^p(U,\R)$ for any $(\talpha,\tbeta)$ and satisfies the needed polynomial controls in the coordinate system $(y^1,\cdots,y^d)$ thanks to the Fa\`a di Bruno Formula.

 A second  useful remark is that $\Cpol^{p,s}(U)$ is an algebra over the field $\R$.
\end{remark}
\begin{lemma}
Assume here that $s\geq 2$. For any coordinate system $(x^i)_{1\leq i\leq d}$ on an open set $U\subset X$ we have for any $1\leq i\leq d$ that
$$\nabla\partial_i\in \Gampol^{s-2,s}(T^1_1(U))\text{ and }\nabla dx^i\in \Gampol^{s-2,s}(T^0_2(U))$$ 
where for $\phi\in \Diff^s_0(\R^n)$, $\nabla=\nabla^\phi$ is the Levi-Civita covariant derivative associated with the pullback metric $g=g^\phi$ on $X$ of the induced metric $g^{\phi(X)}$ on $Y=\phi(X)$ by the Euclidean metric on $\R^n$.
\end{lemma}
\begin{proof}
First we have $\nabla \partial_j=\Gamma^l_{ij}\partial_l\otimes dx^i$ where the $\Gamma^l_{ij}$ are the Christoffel symbols of second kind so that it is sufficient to prove that $\Gamma^k_{ij}\in\Cpol^{s-2,s}(U)$. For given $\phi\in G_V$, as a function of $u\in U$ we have $g_{ij}=\langle d \phi.\partial_i,d\phi.\partial_j\rangle\in C^{s-1}(U,\R)$. Using the chain rule, we get easily for $u\in K$ that, for any $k\leq s-1$,  $|\partial^kg_{ij}|\leq P_k(\|\phi\|_{k+1,\infty})$ where $P$ is a polynomial. Moreover, introducing $\psi=\phi'\circ\phi^{-1}-\Id$, 
\begin{align*}
 g_{ij}(\phi')-g_{ij}(\phi) &=\langle d(\psi+\Id)\circ\phi\cdot d\phi\cdot\partial_i,d(\psi+\Id)\circ\phi\cdot d\phi\cdot\partial_j\rangle -\langle d \phi\cdot\partial_i,d\phi\cdot\partial_j\rangle \\
 &=\langle d\psi\circ\phi\cdot d\phi\cdot\partial_i,d\psi\circ\phi\cdot d\phi\cdot\partial_j\rangle+\langle d\phi\cdot\partial_i,d\psi\circ\phi\cdot d\phi\cdot\partial_j\rangle+\langle d\psi\circ\phi\cdot d\phi\cdot\partial_i,d\phi\cdot\partial_j\rangle
\end{align*}
we get that $|\partial^kg_{ij}(\phi')-\partial^kg_{ij}(\phi)|\leq \|\psi\|_{k+1,\infty}Q_k(\|\psi\|_{k+1,\infty},\|\phi\|_{k+1,\infty})$ and we deduce immediately that $g_{ij}\in \Cpol^{s-1,s}(U)$.

We need now a similar control for the cometric $g^{ij}$. Denoting $\g=(g_{ij})_{1\leq i,j\leq d}$, we have $\g^{-1}=(g^{ij})_{1\leq i,j\leq d}$ and $\g^{-1}=\Com(\g)^T/\det(\g)$ where $\Com(\g)$ is the comatrix of the matrix $\g$. Since $\Com(\g)^T$ is a polynomial expression in the coefficients $g_{ij}$ we get, using the algebra structure property of Remark \ref{rem:27.8.1}, that all the coefficients of $\Com(\g)$ are in $\Cpol^{s-1,s}(U)$. Similarly, $\det(\g)\in \Cpol^{s-1,s}(U)$ so that, in order to get $\det(\g)^{-1}\in \Cpol^{s-1,s}(U)$,  it is sufficient to prove that for any compact $K\subset U$, there exists  a polynomial $P$ such that
  \begin{equation}
\label{eq:20.8.4}
\det(\g)^{-1}\leq  P(\rho_{s}(\phi))\,.
  \end{equation}
However, since  $T_{\phi(u)}Y=\text{Span}\{d\phi(u)\cdot\partial_i,\ 1\leq i\leq d\}$ where $Y=\phi(X)$, then for $(e_1,\cdots,e_d)$ an orthonormal basis of $T_{\phi(u)}Y$, we have $\det(\g)^{-1}=\det((\langle d\phi^{-1}(\phi(u)).e_i,d\phi^{-1}(\phi(u)).e_j\rangle)_{ij})\leq \|d\phi^{-1}\|^{2d}_\infty$. Using the fact that $\Gamma^k_{ij}=\frac{1}{2}(\partial_ig_{mj}+\partial_j g_{mi}-\partial_mg_{ij})g^{mk}$ we get immediately that $\Gamma^{k}_{ij}\in \Cpol^{s-2,s}(U)$ and $\nabla \partial_i\in \Gampol^{s-2,s}(T^1_1(U))$. Now since $0=\nabla(dx^i(\partial_j))=\nabla dx^i(\partial_j)+dx^i(\nabla\partial_j)$ we get $\nabla dx^i(\partial_j)=-\Gamma^i_{lj}dx^l$ and $\nabla dx^i=-\Gamma^i_{lj}dx^j\otimes dx^l$. Since we have just proved that $\Gamma^i_{lj}\in \Cpol^{s-2,s}(U)$, we get the result.
\end{proof}
\begin{lemma}
\label{lem:20.8.1}
  Let $\cI=\{\ (\alpha,\beta)\ |\ \alpha \in \indset^a,\ \beta\in\indset^b,\ 1\leq a<b\leq s\ \}$ and $(x^i)_{1\leq i\leq d}$ be a coordinate system defined on an open set $U\subset X$. 

There exists a family of functions $(c^\alpha_\beta)_{(\alpha,\beta)\in\cI}$ such that 
\begin{enumerate}
\item for any $(\alpha,\beta)\in\cI$, we have $c^\alpha_\beta\in \Cpol^{s-(1+|\beta|-|\alpha|),s}(U)\subset \Cpol^{0,s}(U)$
\item for any $0\leq k\leq s$ any $f\in H^k_{loc}(U)$ and any $\beta\in \indset^k$, we have (a.e.) on $U$
  \begin{equation}
    \label{eq:20.8.2}
    \partial^k_\beta f=\nabla^k_\beta f+\sum_{l=1}^{k-1}\sum_{\alpha,|\alpha|=l}c^\alpha_\beta\nabla^l_\alpha f
  \end{equation}
\end{enumerate}
where for $s\geq 1$ and $\phi\in \Diff^s_0(\R^n)$, $\nabla=\nabla^\phi$ is Levi-Civita covariant derivative associated with the pullback metric $g=g^\phi$ on $X$ on the Euclidean metric on $\phi(X)$ and where  $\partial^k_\beta f=\partial^k\!f(\partial^k_\beta)$ and $\nabla^l_\alpha f=\nabla^l\! f(\partial^l_\alpha)$.
\end{lemma}
\begin{proof}
For $k=0$ or $k=1$ the result is trivial. Let consider a proof by induction for $k\geq 1$. We have for $\beta\in\indset^k$ and $\tbeta=(i,\beta)$ that
$$\partial^{k+1}_\tbeta f=\partial_i(\partial^k_\beta f)=\partial_i[\nabla^k_\beta f+\sum_{l=1}^{k-1}\sum_{\alpha,|\alpha|=l}c^\alpha_\beta\nabla^l_\alpha f]\,.$$ 
However, $\partial_i(\nabla^k_\beta f)=\nabla^{k+1}_\tbeta f+\nabla^k f(\nabla_{\partial_i}\partial^k_\beta)$. Moreover, since we have
\begin{equation*}
 \nabla_{\partial_i}\partial^k_\beta=\sum_{l=1}^k\otimes_{j=1}^{l-1}\partial_{\beta_l}\otimes \nabla_{\partial_i} \partial_{\beta_l}\otimes _{j=l+1}^{k}\partial_{\beta_j}=\sum_{l=1}^k\Gamma_{i\beta_l}^m\otimes_{j=1}^{l-1}\partial_{\beta_j}\otimes\partial_{m}\otimes _{j=l+1}^{k}\partial_{\beta_j}
\end{equation*}
we get that $\nabla_{\partial_i}\partial^k_\beta\in \Gampol^{s-2,s}(T^k_0(U))$ and $\nabla^k f(\nabla_{\partial_i}\partial^k_\beta)$ can be written as $\sum_{\alpha,|\alpha|=k}c^\alpha_\tbeta\nabla^k_\alpha f$ for functions $c^\alpha_\tbeta\in \Cpol^{s-2,s}(U)$. 
Similarly, we have for $1\leq l\leq k$ and $\alpha\in\indset^l$ that 
\begin{equation*}
 \partial_i(c^\alpha_\beta\nabla^l_\alpha f)=\partial_i(c^\alpha_\beta)\nabla^l_\alpha f+ c^\alpha_\beta \nabla^{l+1}_{(i,\alpha)}f+c^\alpha_\beta\nabla^lf(\nabla_{\partial_i}\partial^l_\alpha) .
\end{equation*}
Denoting $c^\alpha_{\tbeta,1}=\partial_i(c^\alpha_\beta)\in \Cpol^{s-(1+|\tbeta|-|\alpha|),s}(U)$, $c^{(i,\alpha)}_{\tbeta,2}=c^\alpha_\beta\in \Cpol^{s-(1+|\beta|-|\alpha|),s}(U)=\Cpol^{s-(1+|\tbeta|-|(i,\alpha)|),s}(U)$ and since $\nabla_{\partial_i}\partial^l_\alpha\in \Gampol^{s-2,s}(T^{l+1}_0(U))$, $c^\gamma_{\tbeta,3}=c^\alpha_\beta dx^\gamma(\nabla_{\partial_i}\partial_\alpha^l)\in \Cpol^{s-(1+|\beta|-l),s}(U)\subset\Cpol^{s-(1+|\tbeta|-|\gamma|),s}(U)$ we get that $\partial_i(c^\alpha_\beta\nabla^l_\alpha f)$ can be written as $\sum_{m=l}^{l+1}\sum_{\gamma,|\gamma|=m}c^\gamma_{\tbeta}\nabla^m_\gamma f$ for some appropriate functions $c^\gamma_\tbeta\in \Cpol^{s-(1+|\tbeta|-|\gamma|),s}(U)$ and decomposition \eqref{eq:20.8.2} holds for the rank $k+1$.
\end{proof}
We finally get to the main result itself.\\
\\
\textbf{Proof of Theorem \ref{theo:Sobolev_control_phi}.}
The starting point is to recast the Sobolev norm on $Y=\phi(X)$ as an integral on $X$ through the pullback metric and pullback covariant derivative. Up to the introduction of a finite partition of unity $(\chi_l)$ subordinated to finite covering of $X$ with charts $(U_l,\psi_l)$, we can restrict to one open set $U=U_l$ and show that for $\chi=\chi_l$ and $K=\text{supp}(\rho)$, there exists a polynomial $P$ such that
  \begin{equation}
\sum_{k=0}^s\int_K \chi.g^0_k(\nabla^k f,\nabla^k f)\vol(g)\leq P(\rho_s(\phi))\sum_{k=0}^s\int_K \chi.\og^0_k(\onabla^k f,\onabla^k f)\vol(\og)\label{eq:20.8.3}
\end{equation}
where $\og=g^{\text{Id}}$ and $\onabla=\nabla^{\text{Id}}$.
For $s=0$ the results comes from the inequalities \eqref{eq:20.8.4}. Let assume that $s\geq 1$ (and thus $s'=s$).
From Lemma \ref{lem:20.8.1}, there exists universal functions $c^\alpha_\beta\in \Cpol^{0,s}(U)$ for any pair $(\alpha,\beta)\in \cI$ such that $\partial^k_\beta f=\nabla^k_\beta f+\sum_{l=1}^{k-1}\sum_{\alpha,|\alpha|=l}c^\alpha_\beta\nabla^l_\alpha f$. In particular, if we denote  $\cJ\doteq \{(k,\gamma)\ |\ 0\leq k\leq d,\ \gamma=\indset^k\ \}$, $\f=(\partial^k_\gamma f)_{(k,\gamma)\in\cJ}$ and $\btf=(\nabla^k_\gamma f)_{(k,\gamma)\in\cJ}$, then there exists $\M\in \Cpol^{0,s}(U,L(\R^\cJ,\R^\cJ))$ (invertible since triangular with ones on the diagonal) with coefficients in $\Cpol^{0,s}(U)$ such that $\f=\M\btf$. Moreover, since $\sum_{k=0}^sg^0_k(\nabla^k f,\nabla^k f)\vol(g)$ can be rewritten as $\q(\btf)$ where $\q$ is a non-degenerate positive quadratic form continuously depending on the location $u\in U$ and coefficients in $\Cpol^{s-1,s}(U)\subset \Cpol^{0,s}(U)$, we get that there exists a polynomial $\tilde{P}$ such that $\q(\btf)=\q(\M^{-1}\f)\leq P(\rho_s(\phi))|\f|^2$ so that
$$\sum_{k=0}^s\int_K \chi.g^0_k(\nabla^k f,\nabla^k f)\vol(g)\leq \tilde{P}(\rho_s(\phi))\sum_{k=0}^s\int_K \chi.|\partial^k f|^2dx\,.$$
Furthermore, considering $\M$ for $\phi=\text{Id}$ there exists a constant $R\geq 0$ such that we have $\sum_{k=0}^s\int_K \chi.|\partial^k f|^2dx\leq R\sum_{k=0}^s\int_K \chi.\og^0_k(\onabla^k f,\onabla^k f)\vol(\og)$ so that \eqref{eq:20.8.3} holds with $P=R\tilde{P}$ and we have obtained Theorem \ref{theo:Sobolev_control_phi}. \hfill $\qed$   \\

We conclude this appendix by adding an extra property of continuity with respect to $\phi$ of the pullback $H^s$ metrics, which is used in the proof of Theorem \ref{theo:distance_fshape_bundle}. From the previous developments, we get that for any chart $(U,\varphi)$ on $X$ associated with a coordinate system $(x^1,\cdots,x^d)$ on $U$  there exists a family of functions $c^\alpha_{\beta}$ such that for any $f\in H^s_{loc}(U)$
\begin{equation}
 \partial^k_\beta f=\nabla^k_\beta f+\sum_{l=1}^{k-1}\sum_{\alpha,|\alpha|=l}c^\alpha_\beta\nabla^l_\alpha f\,.\label{eq:27.8.3}
\end{equation}

Let us denote $E\doteq\bigoplus_{k=0}^s (\stackrel{k}{\otimes}T^*X)$, $E$ is a $C^{s-1}$ vector bundle over $X$. For any local chart $(U,\varphi)$ with coordinate functions $(x^1,\cdots,x^d)$, $(q_k(dx^\beta))_{\beta\in\indset^k,1\leq k\leq d}$ is a local frame of $E$ over $U$ where $q_k:\stackrel{k}{\otimes}T^*X\to E$ denotes the canonical embedding. We will also consider $\myEnd(E)\to X$ the endomorphism vector bundle where $\myEnd(E)_x\doteq\myEnd(E_x)$.
\begin{de}
  We say that $M\in \Gamma^{0,s}(\myEnd(E))$ where $M:\Diff^s_0(\R^n)\to \Gamma^0(\myEnd(E)))$ if for any coordinate system $(x^1,\cdots,x^d)$ defined on a open set $U\subset M$, all the coefficients of $M$ in the local frame $(dx^\beta)_\beta$ are in $\Cpol^{0,s}(U)$.
\end{de}

\begin{de}
  We say that $G\in \Gamma^{0,s}(E^*\otimes E^*)$ where $G:\Diff^s_0(\R^n)\to \Gamma^0(E^*\otimes E^*)$ if for any coordinate system $(x^1,\cdots,x^d)$ defined on a open set $U\subset M$, all the coefficients of $G$ in the local frame $(q_k(\partial^k_\alpha)\otimes q_{k'}(\partial^{k'}_{\alpha'}))$ for $0\leq k,k'\leq s$ and $(\alpha,\alpha')\in\indset^k\times\indset^{k'}$ are in $\Cpol^{0,s}(U)$, where $q_k:\stackrel{k}{\otimes}TM\to E^*$ is the canonical embedding.
\end{de}

Now, writing
$$ \|f\|_{H^{s,\phi}(X)} \doteq \|f \circ \phi^{-1} \|_{H^s(\phi(X))}$$
as the pullback $H^s$ metric on $X$ induced by $\phi \in \Diff^s_0(\R^n)$, we have the following property:
\begin{lemma}
\label{lemma:continuity_Hsphi}
 For any $f\in H^{s}(X)$, the application $\Diff^s_0(\R^n) \rightarrow \R_{+}$, $\phi \mapsto \|f\|_{H^{s,\phi}(X)}$ is continuous.
\end{lemma}
\begin{proof}
Let's introduce $p_E:\Diff^s_0(\R^n)\times H^s(X)\to L^2(X,E)$ such that 
$$p_E(\phi,f)\doteq \bigoplus_{k=0}^s q_k(\nabla^k f)$$
where $\nabla=\nabla^\phi$ for $\phi \in \Diff^s_0(\R^n)$ as well as the pullback of the metric $g_E(\phi)\doteq \oplus_{k=0}^s g^0_k$ with once again $g=g^\phi$. Then we have by definition
$$\|f\|_{H^{s,\phi}(X)}^2=\int_X g_E(\phi)\left(p_E(\phi,f),p_E(\phi,f)\right)\vol(g^\phi) \ .$$
Now, with \eqref{eq:27.8.3}, we see that there exists $M\in \Gamma^{0,s}(\myEnd(E))$ such that $p_E(\phi,f)=M(\phi)\cdot p_E(\Id,f)$. Similarly, thanks to the previously derived expressions of the metric $g^\phi$, we have $g_E\in \Gamma^{0,s}(E^*\otimes E^*)$ and thus there exists $S\in \Gamma^{0,s}(\myEnd(E))$ such that
\begin{align*}
  \int_X
  g_E(\phi)\left(p_E(\phi,f),p_E(\phi,f)\right)\vol(g^\phi)&=\int_X
  \og_E\left(S(\phi)\cdot p_E(\phi,f),p_E(\phi,f)\right)\vol(\og)\\
  &=\int_X \og_E\left(S(\phi)\cdot
    M(\phi)\cdot\op_E(f),M(\phi)\cdot\op_E(f)\right)\vol(\og)\\
&=\int_X \og_E\left(\Lambda(\phi)\cdot\op_E(f),\op_E(f)\right)\vol(\og)
\end{align*}
with $\Lambda\in \Gamma^{0,s}(\myEnd(E))$ and $\og_E=g_E(\Id)$, $\op_E(f) = p_E(\Id,f)$. Since the coefficients of $\Lambda$ in a local frame belong to $\Cpol^{0,s}$, they are in particular continuous with respect to $\phi$ for the norm of uniform convergence of $\phi$ and its derivatives up to order $s$ on the compact $X$. As a consequence, if $\phi^n \rightarrow \phi$ then $\Lambda(\phi^n) \rightarrow \Lambda(\phi)$ in $\Gamma^0(\myEnd(E))$ and:
$$\|f\|_{H^{s,\phi^n}(X)}^2 \xrightarrow[n\rightarrow \infty]{} \int_X \og_E\left(\Lambda(\phi)\cdot\op_E(f),\op_E(f)\right)\vol(\og) = \|f\|_{H^{s,\phi}(X)}^2 $$
which completes the proof.
\end{proof}

\section{Proof of Theorem \ref{theo_Hamiltonian_eq_metam}}
\label{appendix:proof_theorem_Hamiltonian}
 The proof follows similar steps as the pure diffeomorphic case derived in \cite{Arguillere2015b}. Let's introduce the total cost functional: 
 \begin{align*}
  J(q,\check{f},v,\check{h}) &= \int_0^1 \left[\frac{\gamma_V}{2} \|v_t\|_{V}^2 +\frac{\gamma_f}{2} \|\check{h_t}\|_{H^s_q}^2 \right]dt +g(q_1,\check{f_1}) \\
  &=\int_0^1 L(q_t,v_t,\check{h}_t)dt +g(q_1,\check{f_1}) 
 \end{align*}
where $L$ is by definition the Lagrangian function. It is differentiable with respect to $v \in V$, $\check{h} \in H^s(M)$ as well as $q \in C^{s'}(M,\R^n)$ since $s'\geq s$. The variation of $J$ writes: 
 \begin{align*}
   &(\delta J(q,\check{f},v,\check{h})|(\delta q, \delta \check{f},\delta v, \delta \check{h})) \\
   &=\int_0^1 \left[(\partial_q L(q_t,v_t,\check{h}_t)|\delta q_t) + (\partial_{v} L(q_t,v_t,\check{h}_t)|\delta v_t) + (\partial_{\check{h}} L(q_t,v_t,\check{h}_t)|\delta \check{h}_t)\right]dt \\
   &+ (\partial_{q}g(q_1,\check{f_1})|\delta q_1) + (\partial_{\check{f}}g(q_1,\check{f_1})|\delta \check{f}_1)
 \end{align*}
Note that the previous expression involves different duality brackets, in $(C^{s'}(M,\R^n)^*,C^{s'}(M,\R^n))$ for variation with respect to $\delta q$, in $(V^*,V)$ for the variation with respect to $\delta v$ and in $(H^{s}(M)^*,H^s(M))$ for the variation with respect to $\delta \check{h}$ and $\delta \check{f}$. The optimality of solutions $(q_t,\check{f}_t,v_t,\check{h}_t)$ means that formally $\delta J$ should vanish under variations satisfying the control evolutions $\dot{q_t} = \xi_{q_t}v_t$ and $\dot{\check{f_{t}}} = \check{h}_t$. 

Let $H_{(q_0,\check{f}_0)}^{1}([0,1],C^{s'}(M,\R^n)\times H^s(M))$ be the space of time-dependent states with $H^1$ regularity in time and initial conditions $(q_0,\check{f}_0)$. We define the constraint application 
\begin{equation*}
 \varUpsilon: \ H_{(q_0,\check{f}_0)}^{1}([0,1],C^{s'}(M,\R^n)\times H^s(M)) \times L^2([0,1],V\times H^s(M)) \rightarrow L^2([0,1],C^{s'}(M,\R^n)) \times L^2([0,1],H^s(M))
\end{equation*} 
by $\varUpsilon(q,\check{f},v,\check{h}) \doteq (\dot{q}-\xi_{q}v,\dot{\check{f}}- \check{h})$. It is clearly differentiable with respect to $\check{f},v$ and $\check{h}$. Now, since it is assumed that $V\hookrightarrow \Gamma^{s+1}$, the application $q \mapsto \xi_q v = v \circ q$ is differentiable with respect to $q \in C^{s'}(M,\R^n)$ and equal to $(\partial_q \xi_q v | \delta q) = d_{q} v(\delta q)$. It results that $\varUpsilon$ is differentiable with respect to $q$ as well. 

With these notations, we are considering minimizers of $J$ in the constraint set $\varUpsilon^{-1}(\{0\})$. In order to invoke Lagrange multipliers theorem in this infinite-dimensional setting (Theorem 4.1 in \cite{Kurcyusz1976}), it needs to be checked that $d_{(q,\check{f},v,\check{h})}\varUpsilon$ is surjective for all $(q,\check{f},v,\check{h})$. Writing $\varUpsilon_{1}(q,v) = \dot{q}-\xi_{q}v$ and $\varUpsilon_{2}(\check{f},\check{h}) = \dot{\check{f}}- \check{h}$, we have from \cite{Arguillere2015b} (lemma 3) that $d_{(q,v)}\varUpsilon_1$ is surjective and it is straightforward to verify that so is $d_{(\check{f},\check{h})}\varUpsilon_2$. We deduce the existence of Lagrange multipliers $p \in L^2([0,1],C^{s'}(M,\R^n))^*$ and $p^f \in L^2([0,1],H^s(M))^*$ such that:
\begin{align}
 \label{eq:proof_theorem_Hamiltonian1}
 0 &= \left(d_{(q,\check{f},v,\check{h})}J + (d_{(q,\check{f},v,\check{h})}\varUpsilon)^*(p,p^f) | (\delta q,\delta \check{f},\delta v, \delta \check{h}) \right ) \nonumber \\
 &= (p|\dot{\delta q}) - (p|\partial_q \xi_q v.\delta q) - (p|\xi_q \delta v) + (p^f|\dot{\delta \check{f}}) - (p^f|\delta \check{h}) \nonumber \\
 &\phantom{=}+\int_0^1 \left[(\partial_q L(q_t,v_t,\check{h}_t)|\delta q_t) + (\partial_{v} L(q_t,v_t,\check{h}_t)|\delta v_t) + (\partial_{\check{h}} L(q_t,v_t,\check{h}_t)|\delta \check{h}_t)\right]dt \nonumber \\
 &\phantom{=}+ (\partial_{q}g(q_1,\check{f_1})|\delta q_1) + (\partial_{\check{f}}g(q_1,\check{f_1})|\delta \check{f}_1)
\end{align}
Moreover, as $H^s(M)$ is reflexive, it satisfies the Radon-Nykodym property and we have $L^2([0,1],H^s(M))^* = L^2([0,1],H^s(M)^*)$ which allows to identify $p^f$ as a square-integrable function in $H^s(M)^*$. The case of the geometric momentum $p$ is however slightly more involved but was addressed separately in lemma 4 of \cite{Arguillere2015b}, leading to an equivalent identification $p \in L^2([0,1],C^{s'}(M,\R^n)^*)$. It is then straightforward from the expression of the Hamiltonian in \eqref{eq:Hamiltonian} that $\dot{q_t} = \xi_{q_t} v_t = \partial_p H(q_t,\check{f}_t,p_t,p_t^{f},v_t,\check{h}_t)$ and $\dot{\check{f}} = \check{h}_t = \partial_{p^f} H(q_t,\check{f}_t,p_t,p_t^{f},v_t,\check{h}_t)$. 

Considering the variation on $\delta q$ only (i.e with $\delta v = 0$, $\delta \check{f} = 0$ and $\delta \check{h} =0$) in \eqref{eq:proof_theorem_Hamiltonian1}, we obtain for all $\delta q \in C^s(M,\R^n)$: 
\begin{align}
\label{eq:proof_theorem_Hamiltonian2}
 (p|\dot{\delta q}) &= (p|\partial_q \xi_q v.\delta q)-\int_0^1 (\partial_q L(q_t,v_t,\check{h}_t)|\delta q_t) dt - (\partial_{q}g(q_1,\check{f_1})|\delta q_1) \nonumber \\
 &=\int_0^1 (p_t|(\partial_{q} \xi_{q_t} v_t)(\delta q_t)) dt - \int_0^1 (\partial_q L(q_t,v_t,\check{h}_t)|\delta q_t) dt - (\partial_{q}g(q_1,\check{f_1})|\delta q_1) \nonumber \\
 &=\int_0^1 (\underbrace{(\partial_{q} \xi_{q_t} v_t)^*p_t-\partial_q L(q_t,v_t,\check{h}_t)}_{\doteq \alpha_t}|\delta q_t) dt - (\partial_{q}g(q_1,\check{f_1})|\delta q_1)
\end{align}
Let's denote $r_t = \dot{\delta q}_t$ so that $\delta q_t = \int_{0}^{t} r_s ds$ and:
\begin{align*}
 \int_0^1 (\alpha_t|\delta q_t) dt &= \int_0^1 \int_0^t (\alpha_t|r_s) dt \\
 &= \int_{0}^1 \left(\int_{s}^1 \alpha_t dt \Big| r_s \right) ds 
\end{align*}
This together with \eqref{eq:proof_theorem_Hamiltonian2} shows that $p_t = \int_{t}^1 \alpha_s ds - \partial_{q}g(q_1,\check{f_1})$ for almost all $t\in[0,1]$. Now since $\alpha \in L^2([0,1],C^{s'}(M,\R^n)^*) \subset L^1([0,1],C^{s'}(M,\R^n)^*)$, it results that $p \in H^1([0,1],C^{s'}(M,\R^n)^*)$ and:
\begin{equation*}
 \dot{p_t} = -\alpha_{t} = \partial_q L(q_t,v_t,\check{h}_t) - (\partial_{q} \xi_{q_t} v_t)^*p_t = -\partial_q H(q_t,\check{f}_t,p_t,p_t^{f},v_t,\check{h}_t)
\end{equation*}
with the endpoint condition $p_1 = -\partial_{q}g(q_1,\check{f_1})$. 

Similarly, the variation with respect to $\delta \check{f}$ in \eqref{eq:proof_theorem_Hamiltonian1} leads to:
\begin{equation*}
 (p^f|\dot{\delta \check{f}}) + (\partial_{\check{f}}g(q_1,\check{f_1})|\delta \check{f}_1) =0 
\end{equation*}
If we write $\rho_t = \dot{\delta \check{f_t}}$, we obtain:
\begin{equation*}
 \int_0^1 \left(p^f_t + \partial_{\check{f}}g(q_1,\check{f_1}) \big| \rho_t \right) dt =0 
\end{equation*}
which thus holds for all $\rho \in L^2([0,1],H^s(M))$. It results that for almost all $t \in [0,1]$, $p^f_t = -\partial_{\check{f}}g(q_1,\check{f_1})$ or in other words $p^f \in H^1([0,1],H^s(M)^*)$ and:
\begin{equation*}
\dot{p_t^f} = 0 = -\partial_{\check{f}} H(q_t,\check{f}_t,p_t,p_t^{f},v_t,\check{h}_t)
\end{equation*}

Finally, the variations with respect to $v$ and $\check{h}$ give:
\begin{align*}
&\int_0^1 \left(\xi_{q_t}^* p_t - (\partial_{v} L(q_t,v_t,\check{h}_t)|\delta v_t) \right) dt =0 \\
&\int_0^1 \left(p^f_t - \partial_{\check{h}} L(q_t,v_t,\check{h}_t)|\delta \check{h}_t \right) dt =0
\end{align*}
for all $\delta v \in L^2([0,1],V), \delta \check{h} \in L^2([0,1],H^s(M))$. Therefore
\begin{align*}
&\xi_{q_t}^* p_t - (\partial_{v} L(q_t,v_t,\check{h}_t) = \partial_{v} H(q_t,\check{f}_t,p_t,p_t^{f},v_t,\check{h}_t) = 0\\
&p^f_t - \partial_{\check{h}} L(q_t,v_t,\check{h}_t) = \partial_{\check{h}} H(q_t,\check{f}_t,p_t,p_t^{f},v_t,\check{h}_t) = 0
\end{align*}
and the proof of Theorem \ref{theo_Hamiltonian_eq_metam} is complete.

\bibliographystyle{plain}
\bibliography{myRefs} 

\end{document}